\documentclass[10pt, a4paper]{amsart}

\usepackage[T1]{fontenc}
\usepackage{microtype}

\usepackage{amsmath}								
\usepackage{amssymb}
\usepackage{amsthm}
\usepackage{amscd}
\usepackage{mathtools}
\usepackage{amsfonts}
\usepackage{stmaryrd}

\usepackage{extarrows}

\usepackage[bookmarks=true, colorlinks=true, linkcolor=blue!50!black,
citecolor=orange!50!black, urlcolor=orange!50!black, pdfencoding=unicode]{hyperref}
\usepackage{color}

\usepackage{tikz}									
\usetikzlibrary{matrix}
\usetikzlibrary{patterns}
\usetikzlibrary{matrix}
\usetikzlibrary{positioning}
\usetikzlibrary{decorations.pathmorphing}
\usetikzlibrary{cd}

\usepackage{nicefrac}
\usepackage{fullpage}
\usepackage{enumerate}

\newtheorem*{theorem*}{Theorem}

\newtheorem{maintheorem}{Theorem}[section]

\newtheorem{theorem}{Theorem}[section]
\newtheorem{lemma}[theorem]{Lemma}
\newtheorem{proposition}[theorem]{Proposition}
\newtheorem{corollary}[theorem]{Corollary} 

\theoremstyle{definition}
\newtheorem{definition}[theorem]{Definition}
\newtheorem{example}[theorem]{Example}

\newtheoremstyle{myitemstyle}						
	{}			
	{}			
	{}			
	{}			
	{}			
	{.}			
	{ }			
	{}			
\theoremstyle{myitemstyle}
\newtheorem{myitemthm}{}

\newcommand{\R}{\mathbb{R}}

\newcommand{\Z}{\mathbb{Z}}
\newcommand{\ZZ}{\mathbb{Z}}
\newcommand{\Q}{\mathbb{Q}}

\newcommand{\PP}{\mathbb{P}}

\newcommand{\RR}{\mathbb{R}}

\newcommand{\G}{\mathbb{G}}

\newcommand{\TT}{\mathbb{T}}
\newcommand{\T}{\mathbb{T}}

\newcommand{\Mcheck}{{\check{M}}}
\newcommand{\Lambdacheck}{\check{\Lambda}}
\newcommand{\lambdacheck}{\check{\lambda}}
\newcommand{\mucheck}{\check{\mu}}
\newcommand{\Rcheck}{\check{R}}
\newcommand{\Phicheck}{\check{\Phi}}
\newcommand{\alphacheck}{\check{\alpha}}

\newcommand{\calG}{\mathcal{G}}
\newcommand{\calH}{\mathcal{H}}

\newcommand{\calM}{\mathcal{M}}

\newcommand{\bfG}{\mathbf{G}}
\newcommand{\bfT}{\mathbf{T}}
\newcommand{\bfL}{\mathbf{L}}
\newcommand{\bfP}{\mathbf{P}}
\newcommand{\bfB}{\mathbf{B}}

\newcommand{\calU}{\mathcal{U}}

\newcommand{\bfD}{\mathbf{D}}
\newcommand{\bGL}{\mathbf{GL}}
\newcommand{\bPGL}{\mathbf{PGL}}

\newcommand{\mc}[1]{{\mathcal{#1}}}
\newcommand{\wmc}[1]{{\widetilde{\mathcal{#1}}}}

\DeclareMathOperator{\Pic}{Pic}

\DeclareMathOperator{\Hom}{Hom}

\DeclareMathOperator{\Aut}{Aut}

\DeclareMathOperator{\val}{val}

\DeclareMathOperator{\Trop}{Trop}

\DeclareMathOperator{\id}{id}

\DeclareMathOperator{\PGL}{PGL}

\DeclareMathOperator{\Bun}{Bun}

\DeclareMathOperator{\rk}{rk}

\DeclareMathOperator{\sgn}{sgn}

\newcommand{\mPic}{\mathcal Pic}

\newcommand{\SL}{{\mathrm{SL}}}
\newcommand{\GL}{{\mathrm{GL}}}
\newcommand{\Sp}{{\mathrm{Sp}}}
\newcommand{\SO}{{\mathrm{SO}}}

\let\O\relax
\DeclareMathOperator{\O}{O}

\newcommand{\an}{\mathrm{an}}
\newcommand{\st}{\mathrm{st}}
\newcommand{\sat}{\mathrm{sat}}
\newcommand{\semist}{\mathrm{ss}}
\newcommand{\simc}{\mathrm{sc}}
\newcommand{\ind}{\mathrm{ind}}
\newcommand{\trop}{{\mathrm{trop}}}
\newcommand{\ad}{{\mathrm{ad}}}

\newcommand{\Ga}{\Gamma}

\DeclareMathOperator{\Mat}{Mat}
\DeclareMathOperator{\Bij}{Bij}

\newcommand{\Fratila}{Fr{\u{a}}{\c{t}}il{\u{a}}}

\title{Tropical reductive groups and principal bundles on metric graphs} 

\date{}

\author{Andreas Gross}
\address{Institut f\"ur Mathematik, Goethe--Universit\"at Frankfurt,
60325 Frankfurt am Main, Germany}
\email{gross@math.uni-frankfurt.de}

\author{Arne Kuhrs}
\address{Institut f\"ur Mathematik, Universit\"at Paderborn,
33098 Paderborn, Germany}
\email{kuhrs@math.uni-paderborn.de}

\author{Martin Ulirsch}
\address{Institut f\"ur Mathematik, Universit\"at Paderborn,
33098 Paderborn, Germany}
\email{ulirsch@math.uni-paderborn.de}

\author{Dmitry Zakharov}
\address{Department of Mathematics, Central Michigan University, Mount Pleasant, MI 48859, USA}
\email{dvzakharov@gmail.com}

\begin{document}

\begin{abstract} We propose an elementary tropical analogue of a reductive group that combines the datum of a Weyl group and the tropicalization of a fixed maximal torus. For the classical groups, as well as $G_2$, these tropical reductive groups admit descriptions as tropical matrix groups that resemble their classical counterparts. Employing this perspective, we introduce tropical principal bundles on metric graphs and study their explicit presentations as pushforwards of line bundles along covers with symmetries and extra data. Our main result identifies the essential skeleton of the moduli space of semistable principal bundles on a Tate curve with its tropical analogue.
\end{abstract}

\maketitle

\setcounter{tocdepth}{1}
\tableofcontents

\section*{Introduction}

Denote by $\mathbb{T}=(\R\sqcup\{\infty\}, \min, +)$ the semifield of tropical numbers. It is an elementary fact (see e.g. \cite[Lemma 1.4]{TropChern}) that the group $\GL_n(\mathbb{T})$ of invertible $n\times n$ matrices over $\TT$ is the group of generalized tropical permutation matrices. In other words, it is isomorphic to the semidirect product $\mathbb{R}^n\rtimes S_n$. In \cite{2022GrossUlirschZakharov} this observation was used to build a theory of tropical vector bundles on metric graphs, expanding on \cite{TropChern}, that is to say, principal bundles with structure group $\GL_n(\T)$. 

We observe that the two terms in the semidirect product $\GL_n(\mathbb{T})\simeq \mathbb{R}^n\rtimes S_n$ have the following interpretation: $S_n$ is the Weyl group of the reductive algebraic group $\mathbf{GL}_n$, while $\R^n$ is the tropicalization of the diagonal torus $\G_m^n\subseteq \mathbf{GL}_n$. In this article, we expand on this observation and introduce an elementary theory of tropical reductive groups in other Dynkin--Lie types, and an associated theory of principal bundles on metric graphs.

\subsection*{Tropical reductive groups}

Let $\bfG$ be a reductive algebraic group over an algebraically closed field $k$ with a maximal torus $\bfT \subseteq \bf G$. Then $\bfG$ is uniquely determined (up to isomorphism) by its root datum 
\begin{equation*}
\Phi = (M,R,\Mcheck,\Rcheck),
\end{equation*}
where $M$ and $\Mcheck$ are the character and cocharacter lattices of $\mathbf{T}$, and $R \subseteq M$ and $\Rcheck \subseteq \Mcheck$ are the sets of roots and coroots, respectively (see Definition \ref{def:rootdatum}). The Weyl group $W_{\Phi}$ is the group of automorphisms of $M$ generated by the reflections corresponding to the roots in $R$ and naturally acts on the dual space $\Mcheck_{\RR} = \Mcheck \otimes_\Z \R$. We define the \emph{tropical reductive group} associated to the root datum $\Phi$ as
\begin{equation*}
    \bfG^{\trop} = \Mcheck_\R\rtimes W_\Phi.
\end{equation*}
This construction depends only on the root datum $\Phi$ and, in particular, does not depend on the ground field $k$.

For $\bfG=\mathbf{GL}_n$, this construction recovers the matrix group $\GL_n(\mathbb{T})\simeq \mathbb{R}^n\rtimes S_n$  described above. In Section \ref{sec:reductivegroups}, we expand on this observation and describe tropical analogues of other classical groups in terms of tropical linear algebra, closely mirroring their classical counterparts.  

\begin{maintheorem}[Propositions~\ref{prop:GLSLPGL},~\ref{prop:SP},~\ref{prop:tropicalSO}, and~\ref{prop:tropicalG2}]
    The tropical reductive groups associated to the root data of $\SL_n, \PGL_n, \Sp_{2n}, \SO_{2n+1}, \SO_{2n}$ and $\mathrm{G}_2$ admit natural descriptions as matrix groups over $\T$ that are analogous to the matrix descriptions of the corresponding algebraic reductive groups. 
\end{maintheorem}

It would be interesting to determine matrix-theoretic descriptions of the tropical reductive groups associated to the remaining exceptional root systems. 

\subsection*{Principal $\bfG^{\trop}$-bundles on metric graphs}  Let $\Gamma$ be a compact metric graph and denote by $\calH_\Gamma$ the sheaf of continuous real-valued harmonic functions with integer slopes on $\Gamma$. We define a \emph{tropical principal $\bfG^{\trop}$-bundle} as a $(\Mcheck \otimes_\Z {\calH_\Gamma})\rtimes W_\Phi$-torsor on $\Gamma$. In~\cite{2022GrossUlirschZakharov} the authors described an equivalence of the category of tropical principal $\GL_n(\T)$-bundles on $\Gamma$ to the category of free covers $\Ga'\to \Ga$ together with a tropical line bundle on $\Ga'$. In Section \ref{subsec:Gbundlesviacovers}, we provide an explicit description of tropical principal bundles in the classical Lie types as line bundles on covers of $\Gamma$ with suitable Weyl group symmetries together with extra data, generalizing the $\GL_n(\TT)$-case described in~\cite{2022GrossUlirschZakharov}.

\begin{maintheorem}[Corollaries \ref{cor:exact sequence},~\ref{cor:torsors as multi-line bundles with rho-structure}, ~\ref{cor:exact sequence on lattices}, and Example ~\ref{exa:bundlescovers}]
Let $G=\Mcheck_\R\rtimes W_\Phi$ be a tropical reductive group associated to a root datum $\Phi$. Then the category of tropical principal $G$-bundles on $\Gamma$ is equivalent to the category of data consisting of
\begin{itemize}
\item[(i)] a free cover $\Gamma'\to \Gamma$ determined by the associated $W_\Phi$-torsor, and
\item[(ii)] a tropical line bundle on $\Gamma'$ equipped with additional structure reflecting the action of $W_\Phi$.
\end{itemize}
Specifically, for the classical Lie types this equivalence specializes to the following explicit descriptions:
\begin{itemize}
\item[$\GL_n$:] a multi-line bundle $(\Gamma'\!\to\Gamma,L)$, consisting of a free degree $n$ cover $\Gamma'\to\Gamma$ and a tropical line bundle $L$ on $\Ga'$ (\cite[Prop 3.2]{2022GrossUlirschZakharov});
\item[$\SL_n$:] as for $\GL_n$, with a trivialization of the determinant line bundle $\det(L)$;
\item[$\Sp_{2n}$:] multi-line bundles $(\Gamma'\!\to\Gamma,L)$, where $\Gamma'\to\Gamma$ is a degree $2n$ cover and $L$ a tropical line bundle on $\Ga'$ together with a fixed-point-free involution $\iota$ on the cover and a trivialization of  $(L \otimes \iota^{-1}L)/\iota$; 
\item[$\SO_{2n+1}$:] as for $\Sp_{2n}$,
\item[$\SO_{2n}$:] as for $\Sp_{2n}$, with a trivialization of the orientation double cover;
\item[$\mathrm{G}_2$:] a degree $6$ cover $\Gamma'\to \Gamma$ with a locally trivial identification of each fiber with the vertices of the Star of David, a tropical line bundle $L$ on $\Ga'$, and trivializations of $(L \otimes \iota^{-1}L)/\iota$ where the involution $\iota:\Ga'\to \Ga'$ exchanges the opposite vertices in each star, and a $\iota$-invariant trivialization of the line bundle on the domain of the associated $\Sp_2(\TT)$-cover whose fibers correspond to the two triangles.
\end{itemize}

\end{maintheorem}

Let $G=\Mcheck_\R\rtimes W$ be a tropical reductive group associated to a root datum $\Phi$, where $W = W_\Phi$ is the Weyl group. The moduli space $\calM_G(\Ga)$ of isomorphism classes of $G$-bundles on $\Ga$ decomposes as a finite disjoint union 
\[
\calM_G(\Ga)=\coprod_{\tau\in \calM_{W}(\Ga)}\calM_{G,\tau}(\Ga)
\]
indexed by the isomorphism type of the associated $W$-torsor. For a $W$-torsor $\tau$ on $\Gamma$ we show that $\mc M_{G,\tau}$ is the quotient of a disjoint union of torsors under tropical abelian varieties by the finite group $\Aut(\tau)$ (see Theorem \ref{thm:description of component of tropical moduli space} below). In the case where $\tau = W_{\Ga}$ is the trivial $W$-torsor on $\Ga$, we obtain $\mc M_{G,W_\Gamma} \cong \big(\Pic(\Gamma)\otimes_\Z \Mcheck\big)/W$ (see Proposition \ref{prop:fiber over trivial W-torsor} below), which allows classifying $G$-bundles on metric graphs of genus zero in analogy with the classical theorems of Grothendieck \cite{Grothendieck_vectorbundlesonP1} (for vector bundles) and Harder \cite{1968Harder} (in general). We refer the reader to Example \ref{exa:bundlestrees} below for details. 

\subsection*{Tropicalization of $\bfG$-bundles} In~\cite{2022GrossUlirschZakharov} we observed that $\GL_n(\TT)$ should be viewed as an incomplete tropicalization of $\mathbf{GL}_n$, since, for example, the former has dimension $n$ while the latter has dimension $n^2$. For this reason, we cannot expect the moduli space of principal $\GL_n(\TT)$-bundles on the skeleton $\Gamma_X$ of an algebraic curve $X$ to be the tropicalization of the moduli space of $\mathbf{GL}_n$-bundles on $X$. The same problem exists for almost all other reductive groups. Nonetheless, it turns out that the tropicalization map is defined for semistable bundles on an elliptic curve.

Let $X$ be a Tate elliptic curve over an algebraically closed and complete non-Archimedean field $K$ of equicharacteristic $0$. \Fratila~  \cite{fratila2016stack,Fratila_revisitingelliptic} provides an explicit description of the moduli spaces $\mc M_\mathbf{G}^{\lambdacheck_\mathbf{G},\mathrm{st}}(X)$ and $\mc M_\mathbf{G}^{\lambdacheck_\mathbf{G},\mathrm{ss}}(X)$ of stable and semistable $\mathbf{G}$-bundles of degree $\lambdacheck_\bfG \in \pi_1(\bfG)$ on $X$, respectively. We recall this description, which generalizes work of Atiyah \cite{Atiyah_ellipticvectorbundles}, Tu \cite{Tu_semistablebundles}, and Laszlo \cite{Laszlo_elliptic}, in Section \ref{sec:Degree and stability in the algebraic setting}. We prove an analogous tropical statement describing the moduli spaces of semistable and stable $\bfG^\trop$-bundles on a metric circle, see Section \ref{sec:moduliofbundles}. Finally, in Section~\ref{sec_tropicalization}, we tropicalize stable $\bfG$-bundles on $X$ by reducing them to $N_\bfG(\bfT)$-bundles, where $N_\bfG(\bfT) \subseteq \bfG$ is the normalizer of a fixed maximal torus $\bfT$ of $\bfG$, and semistable bundles by passing to a Levi subgroup. Our main result can be summarized as follows:

\begin{maintheorem}[Theorem \ref{thm:tropical semistable bundles skeleton of semistable bundles}]
\label{mainthm_skeleton=tropicalmodulispace}
    Let $X$ be a Tate elliptic curve over an algebraically closed and complete non-Archimedean field $K$ of equicharacteristic $0$, so that we have a non-Archimedean uniformization $X^{\an}=\G_m^{\an}/q^\Z$ and the minimal skeleton of $X^{\an}$ is given by the metrized circle $\Gamma_X=\R/\val(q)\Z$. Moreover, let $\mathbf{G}$ be a reductive algebraic group over $K$ and denote by $\mc M_{\mathbf{G}}^{\lambdacheck,\semist}(X)$ the moduli space of semistable principal $\mathbf{G}$-bundles of degree $\lambdacheck\in\pi_1(\mathbf{G})$.

    There is a natural continuous tropicalization map $\Trop: (\mc M_{\mathbf{G}}^{\lambdacheck,\semist}(X))^\an \rightarrow \mc M_{\bfG^\trop,\ind}^{\lambdacheck}(\Ga_X)$ together with a homeomorphism between the moduli space $\mc M_{\bfG^\trop,\ind}^{\lambdacheck}(\Ga_X)$ of indecomposable principal $\bfG^{\trop}$-bundles on $\Ga_X$ of degree $\lambdacheck$ and the essential skeleton $\Sigma(\mc M_{\mathbf{G}}^{\lambdacheck,\semist}(X))$ of $(\mc M_{\mathbf{G}}^{\lambdacheck,\semist}(X))^\an$ that makes the diagram
\begin{equation*}
    \begin{tikzcd}
    & 
        \Sigma(\mc M_{\mathbf{G}}^{\lambdacheck,\semist}(X))\arrow[dd,"\cong"] \\
    (\mc M_{\mathbf{G}}^{\lambdacheck,\semist}(X))^\an \arrow[ur,"\rho"]\arrow[dr,"\Trop"]&
        & \\
    &
        \mc M_{\bfG^\trop,\ind}^{\lambdacheck}(\Ga_X)\\
    \end{tikzcd}
\end{equation*}
    commute.
\end{maintheorem}

Theorem \ref{mainthm_skeleton=tropicalmodulispace} fits into a sequence of results establishing relationships between tropical moduli spaces and non-Archimedean skeletons/tropicalizations of their algebraic counterparts that started with ~\cite{BakerRabinoff} in the case of the Jacobian of an algebraic curve and ~\cite{ACP} for the moduli space of curves. It generalizes \cite[Theorem D]{2022GrossUlirschZakharov}, which covers the case of vector bundles, i.e. the case $\bfG=\bGL_n$.

\subsection*{Further discussion and remarks}

We expect that, in order to generalize Theorem \ref{mainthm_skeleton=tropicalmodulispace} to moduli spaces of semistable bundles on Mumford curves of higher genus, we will need a more refined theory of tropical principal bundles than the one proposed in this article. The underlying deeper reason for this seems to be that the tropical reductive groups proposed here are relatively sparse matrix groups, so that there is no good way to directly tropicalize a reductive algebraic group $\bfG$ onto its tropical counterpart $\bfG^{\trop}$. For example, the dimension of $\bfG^{\trop}$ is usually strictly less than that of $\bfG$, and the same holds for the dimensions of the corresponding moduli spaces. On a Tate curve, strong classification results for algebraic principal bundles on elliptic curves allow us to circumvent this problem.

In \cite{GKUZ}, the authors expand on the elementary framework of tropical vector bundles developed in \cite{2022GrossUlirschZakharov} and show in \cite[Theorem B]{GKUZ} how the essential skeletons of the moduli spaces of (semi-)homogeneous bundles (in the sense of \cite{Mukai}) on abelian varieties with maximally degenerate reduction can be identified with suitable moduli spaces of tropical semi-homogeneous vector bundles on the tropicalized abelian variety. We believe that a common generalization of Theorem \ref{mainthm_skeleton=tropicalmodulispace} and \cite[Theorem B]{GKUZ} to homogeneous principal bundles on abelian varieties is, in principle, possible. Thanks to the comparative lack of moduli-theoretic classification results on the classical side, this could, however, turn out to be technically quite demanding. 

In two parallel articles \cite{KhanMaclagan_tropicaltoricvectorbundles} and \cite{KavehManon_tropicaltoricvectorbundles}  Khan and Maclagan as well as Kaveh and Manon propose a seemingly quite different approach to the tropical geometry of vector bundles. Their central idea is a definition that abstracts the combinatorial data coming from Klyachko's classification of toric vector bundles. By its very nature, this approach leads to rather satisfying results when studying the tropicalization of toric vector bundles on toric varieties (and restrictions thereof to subvarieties of toric varieties). At this point, the framework proposed in \cite{KavehManon_tropicaltoricvectorbundles} and \cite{KhanMaclagan_tropicaltoricvectorbundles} does not seem to be able to encode monodromy phenomena on tropical varieties with nontrivial fundamental group and, hence, seems to lead to results different from ours in the case of the Tate curve. The generalization of Klyachko's classification to the setting of torus-equivariant principal bundles on toric varieties (see e.g. \cite{KavehManon_toricprincipalbundles}) could be the starting point of a satisfying theory of tropical toric principal bundles that generalizes 
\cite{KhanMaclagan_tropicaltoricvectorbundles} and \cite{KavehManon_tropicaltoricvectorbundles}. 

An alternative approach to understand the nature of tropical (vector) bundles might arise from the work of Kennedy-Hunt and Ranganathan \cite{KennedyHuntRanganathan_logQuot} on the construction of logarithmic Quot schemes, where the authors build upon ideas introduced in \cite{MolchoWise} for the logarithmic Picard variety. The central objects in \cite{KennedyHuntRanganathan_logQuot} are coherent sheaves on suitable logarithmic modifications of a given logarithmic curve. The combinatorial shadow of a generalization of the chip-firing equivalence of line bundles (see \cite[Sect. 2 and 3]{BakerJensen}) would be another contender for an object that could be named "tropical vector bundle". 

Essential skeletons of non-Archimedean analytic spaces were introduced and studied in \cite{MustataNicaise,NicaiseXu, NicaiseXuYue} in order to make precise ideas of Kontsevich and Soibelman \cite{KontsevichSoibelman} for a non-Archimedean approach to the SYZ-fibration in mirror symmetry. Our Theorem \ref{mainthm_skeleton=tropicalmodulispace} may be seen as an explicit example of a non-Archimedean SYZ-fibration. Our approach is indebted to the results in \cite{BrownMazzon}, which allow us to study the behaviour of essential skeletons of finite group quotients.

\subsection*{Acknowledgments} We thank Luca Battistella, Kiumars Kaveh, Bivas Khan, Oliver Lorscheid, Chris Manon, Diane Maclagan, and Dhruv Ranganathan for helpful conversations and interactions during the creation of this article. 

\subsection*{Funding} This project has received funding from the Deutsche Forschungsgemeinschaft (DFG, German Research Foundation) TRR 326 \emph{Geometry and Arithmetic of Uniformized Structures}, project number 444845124, and TRR 358 \emph{Integral Structures in Geometry and Representation Theory}, project number 491392403, as well as from the Deutsche Forschungsgemeinschaft (DFG, German Research Foundation) Sachbeihilfe \emph{From Riemann surfaces to tropical curves (and back again)}, project number 456557832, the DFG Sachbeihilfe \emph{Rethinking tropical linear algebra: Buildings, bimatroids, and applications}, project number 539867663, within the  
SPP 2458 \emph{Combinatorial Synergies}, and the Marie-Sk\l{}odowska-Curie-Stipendium Hessen (as part of the HESSEN HORIZON initiative).

\section{Tropical reductive groups} \label{sec:reductivegroups}

In this section, we describe tropical versions of the classical reductive groups by means of canonical real extension of the corresponding Weyl groups. We show that these groups have natural descriptions as matrix groups over the tropical semifield. The theory developed in this section may be seen as a generalization of an analogy proposed by Tits \cite{Tits_analogy}, namely that Weyl groups should be seen as analogues of the classical groups over the field $\mathbb{F}_1$ with one element. In a certain sense, we obtain the corresponding tropical reductive groups by a base change to $\TT$. While we do not explicitly make use of any of the various approaches to $\mathbb{F}_1$-geometry, our treatment is informed and inspired by the work of Lorscheid in \cite{Lorscheid_blueprintsII}, which provides a theoretical foundation for Tits' analogy using the perspective of blueprints, as introduced in \cite{Lorscheid_blueprintsI}. For more background on tropical matrix groups we also refer the reader to \cite{IJK_tropicalmatrixgroups}.

\subsection{Root systems and tropical reductive groups} 
According to a classical result of Chevalley, split reductive algebraic groups over a fixed field are classified by their root data (for example, see Theorems~9.6.2 and~10.1.1 in~\cite{Springer_linalggroups}). We recall the definition.

\begin{definition}
\label{def:rootdatum}
A \emph{root datum} is a quadruple $\Phi=(M,R,\Mcheck,\Rcheck)$ consisting of 
\begin{itemize}
    \item free abelian groups $M$ and $\Mcheck$ of finite rank with a duality pairing $\langle\cdot,\cdot\rangle:M\times \Mcheck\to \ZZ$, and
    \item finite subsets of \emph{roots} $R\subseteq M$ and \emph{coroots} $\Rcheck\subseteq \Mcheck$ together with a bijection $\check{(\cdot)}\colon R\rightarrow \Rcheck$
\end{itemize}
subject to the following two axioms:
\begin{enumerate}[(i)]
    \item For all $\alpha\in R$ we have $\langle \alpha,\alphacheck\rangle=2$.
    \item The reflection homomorphisms $s_\alpha\colon M\rightarrow M$ and $s_{\check\alpha}\colon \Mcheck\rightarrow \Mcheck$ given by 
    \begin{equation*}
        u\longmapsto u-\langle u,\alphacheck\rangle \alpha \qquad \textrm{and} \qquad v\longmapsto v-\langle \alpha,v\rangle \alphacheck
    \end{equation*}
    satisfy
    \begin{equation*}
        s_\alpha(R)=R \qquad \textrm{and} \qquad s_{\check\alpha}(\Rcheck)=\Rcheck
    \end{equation*}
    for all $\alpha\in R$.
\end{enumerate}

A root datum $\Phi=(M,R,\Mcheck,\Rcheck)$ is said to be \emph{reduced}, if for all $\alpha\in R$ we have $2\alpha\notin R$. From now on, all of our root data will be assumed to be reduced and the term \emph{root datum} will mean a \emph{reduced root datum}.
\end{definition}

The \emph{Weyl group} $W_\Phi$ of the root datum $\Phi$ is the (necessarily finite) automorphism group of $M$ generated by the reflections $s_\alpha$ for all $\alpha\in R$. The action of $W_\Phi$ on the lattice $M$ defines a dual action on the dual lattice $\Mcheck$, which extends to the vector space $\Mcheck_{\RR}=\Mcheck\otimes_{\ZZ}\RR$.

\begin{definition} The \emph{tropical reductive group} associated to the root datum $\Phi=(M,R,\Mcheck,\Rcheck)$ is the semidirect product $G_{\Phi}=\Mcheck_\R\rtimes W_{\Phi}$. 
\label{def:tropicalreductivegroup}
\end{definition}

We emphasize that the Weyl group $W_{\Phi}$ of a root datum $\Phi$ is not defined as an abstract group, but is explicitly presented via its action on the lattice $M$. We use this presentation to construct our tropical reductive group, hence we do not simply associate a tropical object to an abstract group.

We now define homomorphisms of tropical reductive groups.

\begin{definition} Let $G_{\Phi_1}=\Mcheck_{1,\RR}\rtimes W_{\Phi_1}$ and $G_{\Phi_2}=\Mcheck_{2,\RR}\rtimes W_{\Phi_2}$ be tropical reductive groups associated to root data $\Phi_1=(M_1,R_1,\Mcheck_1,\Rcheck_1)$ and $\Phi_2=(M_2,R_2,\Mcheck_2,\Rcheck_2)$, respectively. Let $f:\Mcheck_1\to \Mcheck_2$ be a $\ZZ$-linear homomorphism and let $\phi:W_{\Phi_1}\to W_{\Phi_2}$ be a group homomorphism such that for any $m \in \Mcheck_1$ and any $g\in W_{\Phi_1}$ we have $\phi(g)(f(m))=f(g(m))$. The pair $(f,\phi)$ defines a \emph{homomorphism of tropical reductive groups}
\begin{equation*}\begin{split}
    F:\Mcheck_{1,\RR}\rtimes W_{\Phi_1}&\longrightarrow \Mcheck_{2,\RR}\rtimes W_{\Phi_2}\\
    (m,g)&\longmapsto(f(m),\phi(g)).
\end{split}\end{equation*}
\label{def:homomorphism}
\end{definition}

\subsection{Type $A_n$: the tropical general, special, and projective linear groups} We now calculate the tropical reductive groups associated to the classical root data and show that they admit natural descriptions as matrix groups over the tropical semifield $\TT$. We start with type $A_n$.

Recall that $\TT=\RR\cup \{\infty\}$ with operations
\[
x\oplus y=\min(x,y)\quad \text{and}\quad x\odot y=x+y.
\]
The additive and multiplicative identities are $\infty$ and $0$, respectively. We note that $\TT$ contains no nontrivial roots of unity,
\[
\mu_n(\TT)=\{x\in \TT: x^{\odot n}=0\}=\{x\in \TT: nx=0\}=\{0\},
\]
hence the element $0$ plays the role of both $+1$ and $-1$. The semifield operations on $\TT$ extend to a matrix product on the set $\Mat(n\times n,\TT)$ of $(n\times n)$-matrices with entries in $\TT$: 
\[
(A\odot B)_{ij}=\bigoplus_{k=1}^n a_{ik}\odot b_{kj}.
\]
The multiplicative identity in $\Mat(n\times n,\TT)$ is the matrix
\[
I_n=\left(\begin{array}{cccc}
    0 & \infty & \cdots & \infty  \\
    \infty & 0 & \cdots & \infty  \\
    \cdots & \cdots & \cdots & \cdots  \\
    \infty & \infty & \cdots & 0  
\end{array}
\right).
\]

We first describe the \emph{tropical general linear group}, which is the group of invertible elements in $\Mat(n\times n,\TT)$. Allermann shows (see Lemma 1.4 in~\cite{TropChern}) that these elements are the products of diagonal and permutation matrices:
\[
\begin{split}
\GL_n(\TT)&=\{A\in \Mat(n\times n,\TT): A\odot A^{-1}=A^{-1}\odot A=I_n\mbox{ for some }A^{-1}\in \Mat(n\times n,\TT)\}\\
&=\{ D(y_1,\ldots,y_n)\odot P_{\sigma}:y_1,\ldots,y_n\in \RR,\sigma\in S_n\}\\
&=\RR^n\rtimes S_n.
\end{split}
\]
Here $D(y_1,\ldots,y_n)$ is the tropical diagonal matrix with finite entries $y_1,\ldots,y_n\in \RR$ on the diagonal and $\infty$ everywhere else, and $P_{\sigma}$ for $\sigma\in S_n$ is the tropical permutation matrix 
\[
(P_{\sigma})_{ij}=\begin{cases}
    0, & \mbox{if }i=\sigma(j),\\
    \infty, & \mbox{otherwise.}
\end{cases}
\]

To define $\SL_n(\TT)$, we recall that the \emph{tropical determinant}~\cite{maclagan2015introduction} of a matrix $A\in \Mat(n\times n,\TT)$ is 
\[
\det A=\bigoplus_{\sigma\in S_n} A_{1\sigma(1)}\odot\cdots\odot A_{n\sigma(n)}.
\]

We note that the tropical determinant is the same as the tropical permanent, because both $+1$ and $-1$ tropicalize to $0$ in $\TT$. The determinant of an invertible matrix is finite (the converse is not true in general), is equal to the sum of the finite entries, and restricts to a homomorphism $\det:\GL_n(\TT)\to \RR=\TT^\ast$ given by
\[
\det (D(y_1,\ldots,y_n)\odot P_{\sigma})=y_1+\cdots+y_n.
\]

We now define the \emph{tropical special linear group} as
\[\begin{split}
\SL_n(\TT)&=\{A\in \GL_n(\TT): \det A=0\}\\
&=\{D(y_1,\ldots,y_n)\odot P_{\sigma}:y_i\in \RR,y_1+\cdots+y_n=0,\sigma\in S_n\}\\ &=\RR^n_0\rtimes S_n.
\end{split}\]
Finally, we define the \emph{tropical projective linear group} as the quotient of $\GL_n(\TT)$ by its center, which is the subgroup of scalar matrices:
\[
\PGL_n(\TT)=\GL_n(\TT)/\TT^*=(\RR^n/\RR)\rtimes S_n. 
\]

We now recall the root data of $\GL_n$, $\SL_n$, and $\PGL_n$. 
The root datum $(M,R,\Mcheck,\Rcheck)$ of $\GL_n$ has lattices $M=\Mcheck=\ZZ^n$ with the standard pairing, and the roots and coroots are $R=\Rcheck=\{e_i-e_j:i\neq j\}$, where the $e_i$ are the standard basis vectors. The root datum of $\SL_n$ has the same roots and coroots, but the lattices are $M=\ZZ^n/(e_1+\cdots+e_n)\ZZ$ and $\Mcheck=\ZZ^n_0$, where $\ZZ^n_0\subseteq \ZZ^n$ is the set of vectors whose coordinates sum to zero. Finally, the root datum of $\PGL_n$ is the same as for $\SL_n$, but with the lattices exchanged. Reflection through $e_i-e_j$ exchanges the $i$th and $j$th coordinates and fixes the rest, so the Weyl group in all three cases is the symmetric group $S_n$ acting by permutation matrices. 

We therefore obtain the following result.

\begin{proposition} The tropical reductive groups associated to the root data of $\GL_n$, $\SL_n$, and $\PGL_n$ are respectively $\GL_n(\TT)$, $\SL_n(\TT)$, and $\PGL_n(\TT)$.
\label{prop:GLSLPGL}
\end{proposition}

This proposition explains why we define the tropical reductive group of a root datum $\Phi=(M,R,\Mcheck,\Rcheck)$ as $\Mcheck_{\RR}\rtimes W_{\Phi}$ and not $M_{\RR}\rtimes W_{\Phi}$: exchanging $M$ and $\Mcheck$ would exchange $\SL_n(\TT)$ and $\PGL_n(\TT)$. We now consider two basic examples of homomorphisms of tropical reductive groups. 

\begin{example} Let $f:\ZZ^n\to \ZZ$ and $\phi:S_n\to S_1$ be respectively the sum map and the trivial map. Then the induced homomorphism of tropical reductive groups $F=(f,\phi):\GL_n(\TT)\to \GL_1(\TT)=\RR$ is the tropical determinant.
\end{example}

\begin{example} Let $\ZZ^n_{0}$, $\ZZ^n$, and $\ZZ^n/\ZZ(1,\ldots,1)$ be the cocharacter lattices of $\SL_n$, $\GL_n$, and $\PGL_n$, respectively. The canonical maps $\ZZ^n_{0}\to\ZZ^n\to\ZZ^n/\ZZ(1,\ldots,1)$ and the trivial maps on $S_n$ induce homomorphisms
\[
\SL_n(\TT)\to \GL_n(\TT)\to \PGL_n(\TT)
\]
of tropical reductive groups. We note that the composed map $\SL_n(\TT)\to \PGL_n(\TT)$ is an isomorphism of abstract groups but not an isomorphism of tropical reductive groups, since the lattice map $\ZZ^n_{0}\to\ZZ^n/\ZZ(1,\ldots,1)$ is not surjective. This may be seen as a characteristic one shadow of the fact that for an algebraically closed field $k$ of characteristic $p$, the map $\SL_p(k)\to \PGL_p(k)$ is a bijection on the $k$-points, but not an isomorphism of schemes.
\end{example}

\subsection{Type $C_n$: the tropical symplectic group} We now define the tropical symplectic group in complete analogy with the algebraic setting, which we now recall (see Section~5.3.3 in~\cite{Lorscheid_blueprintsII}). Let $J$ be the $2n\times 2n$ block matrix with off-diagonal blocks $I_n$ and $-I_n$ and zero diagonal blocks. For any ring $R$, the set $\Sp_{2n}(R)$ of $R$-rational points of the symplectic group is the set of $2n\times 2n$-matrices $A$ with entries in $R$ satisfying $A^tJA=J$.

In the semifield $\TT$, $0$ plays the role of both $1$ and $-1$, hence we replace $J$ with the matrix 
\[
J=\left(\begin{array}{cc}
    \infty & I_n \\
     I_n & \infty
\end{array}\right)
\]
and make the following definition.

\begin{definition} The \emph{tropical symplectic group} is
\[
\Sp_{2n}(\TT) = \{A \in \GL_{2n}(\TT):  A^t \odot J \odot A=J  \}.
\]
    
\end{definition}

We now give an explicit description of $\Sp_{2n}(\TT)$. We label the columns of a $2n\times 2n$-matrix using the index set $[\pm n]=\{1,\ldots,n,-1,\ldots,-n\}$, which carries the fixed-point-free sign involution
\begin{equation*}
\iota:[\pm n]\to [\pm n] \quad \text{given by} \quad \iota(k)=-k.
\end{equation*}
In terms of this identification, the matrix $J=P_{\iota}$ is the tropical permutation matrix associated to $\iota$. We recall that the \emph{signed permutation group} $S^B_n=S^C_n$ is the set of permutations of $[\pm n]$ commuting with $\iota$: 
\[
S^B_n = \big\{\sigma\in S_{2n}: \sigma(-k) = - \sigma(k)\mbox{ for all }k\in [\pm n ] \big\} \subseteq S_{2n}.
\]
An element of $S_n^B$ permutes the set of pairs $\{1,-1\},\ldots,\{n,-n\}$ and acts inside each pair, hence $S_n^B$ is an extension of $S_n$ by $(\ZZ/2\ZZ)^n$. 

\begin{proposition} The tropical symplectic group is the semidirect product
\[
\Sp_{2n}(\T) = \big\{D(y_1,\ldots,y_n,-y_1,\ldots,-y_n)\odot P_{\sigma}
:\sigma \in S_n^B, y_1,\ldots,y_n \in \RR\big\}=\R^n\rtimes S_n^B.
\]
\end{proposition} 

We note that $\det A=0$ for $A\in \Sp_{2n}(\TT)$, as one would expect.

\begin{proof}
    Let  $A =  D(y_1,\ldots,y_n,y_{-1},\ldots,y_{-n}) \odot P_{\sigma} \in \GL_{2n}(\T)$ be an invertible matrix with $\sigma \in S_{2n}$ and $y_1,\ldots,y_n,y_{-1},\ldots,y_{-n} \in \RR$. Plugging this into $A^t \odot J \odot A=J$, we obtain
    \[
     P_{\sigma^{-1}} \odot D(y_1,\ldots,y_n,y_{-1},\ldots,y_{-n})\odot P_\iota \odot  D(y_1,\ldots,y_n,y_{-1},\ldots,y_{-n}) \odot P_{\sigma} = P_{\iota},
    \]
    which is equivalent to
\[
       D(y_1,\ldots,y_n,y_{-1},\ldots,y_{-n})\odot D(y_{-1},\ldots,y_{-n},y_{1},\ldots,y_{n}) \odot P_{\iota} \odot P_{\sigma} = P_{\sigma} \odot P_{\iota}.
    \]
This is satisfied if and only if $y_{-i}=-y_i$ for all $i=1,\ldots,n$ and furthermore $\iota\sigma=\sigma \iota$, so that $\sigma\in S_n^B$. \end{proof}

We now compute the tropical reductive group associated to the root datum $(M,R,\Mcheck,\Rcheck)$ of $\Sp_{2n}$. The lattices of this root datum are $M=\Mcheck=\ZZ^n$ with the standard pairing. The roots $R\subseteq M$ are the vectors $\pm 2e_i$ and $\pm e_i\pm e_j$ for $i\neq j$, while the coroots $\Rcheck\subseteq \Mcheck$ are $\pm e_i$ and $\pm e_i\pm e_j$ for $i\neq j$. Reflection in $e_i-e_j$ exchanges the $i$th and $j$th coordinates, while reflection in $e_i$ changes the sign of the $i$th coordinate, so the Weyl group of this root datum is the signed permutation group $S^B_n$. Hence we have the following result.

\begin{proposition}
\label{prop:SP}
The tropical reductive group associated to the root datum of $\Sp_{2n}$ is $\Sp_{2n}(\TT)$.
    
\end{proposition}

We note that the embedding $\Sp_{2n}(\TT)\to \GL_{2n}(\TT)$ described above is a homomorphism $F=(f,\phi)$ of tropical reductive groups, given by the $\Z$-linear homomorphism 
\[f:\Z^n \to \Z^{2n}, \quad f(x_1,\ldots,x_n) = (x_1,\ldots,x_n,-x_1,\ldots,-x_n)\]
on the lattices that is compatible with the embedding $\phi:S_n^B \hookrightarrow S_{2n}$.

\subsection{Types $B_n$ and $D_n$: the tropical orthogonal and special orthogonal groups} Our description of the tropical orthogonal groups is likewise inspired by Lorscheid's integral models (see Section~4.3.4 in~\cite{Lorscheid_blueprintsII}). In the algebraic setting, given a ring $R$, the orthogonal group $\O_m(R)$ is defined as the group of $m\times m$ invertible matrices over $R$ preserving the standard split quadratic form, a notion that we can tropicalize directly. Defining the special orthogonal group in a characteristic-independent manner requires additional work. Namely, if $m$ is odd, then the subgroup $\SO_m(R)\subseteq \O_m(R)$ is defined as the kernel of the determinant map. If $m$ is even, however, then $\SO_m(R)\subseteq \O_m(R)$ is instead defined to be the kernel of the Dickson homomorphism $D_m:\O_m(R)\to \ZZ/2\ZZ$, which counts the number (mod 2) of terms in any factorization of an orthogonal matrix as a product of reflection matrices.

Let $\TT^m$ be a semimodule over $\TT$ of dimension $m$ with coordinates $x=(x_1,\ldots,x_n,x_{-1},\ldots,x_{-n})$ when $m=2n$ and $x=(x_0,\ldots,x_n,x_{-1},\ldots,x_{-n})$ when $m=2n+1$. We define the \emph{standard split tropical quadratic form} $q_m:\TT^m\to \TT$ by the formulas
\[
q_{2n}(x)=\bigoplus_{k=1}^{n} x_k \odot x_{-k}\quad \text{and}\quad q_{2n+1}(x)= x_{0}^{\odot 2} \oplus \bigoplus_{k=1}^{n} x_k \odot x_{-k}.
\]

\begin{definition} The \emph{tropical orthogonal group} $\O_m(\TT)$ is
\[
\O_m(\TT)=\{A\in \GL_m(\TT):q_m(A\odot x)=q_m(x)\mbox{ for all }x\in \TT^m\}.
\]
\end{definition}

We defined the signed permutation group $S_n^B\subseteq S_{2n}$ as the group of permutations of the set $[\pm n]$ preserving the fixed-point-free sign involution. We also view $S_n^B$ as a subgroup of $S_{2n+1}$, consisting of those permutations of the set $[\pm n]\cup \{0\}$ that preserve the sign involution (which now has the unique fixed point $0$).

\begin{proposition}  For $m=2n+1$, the tropical orthogonal group is the semidirect product
        \[\begin{split}
        \O_{2n+1}(\TT)&=\{ D(0,y_1,\ldots,y_n,-y_1,\ldots,-y_n)\odot P_{\sigma}:y_1,\ldots,y_n \in \RR,\sigma \in S_n^B\subseteq S_{2n+1}\}\\
        &=\R^n\rtimes S_n^B.
        \end{split}
        \]
For $m=2n$, the tropical orthogonal group is the semidirect product
        \[\begin{split}
        \O_{2n}(\TT)&=\{D(y_1,\ldots,y_n,-y_1,\ldots,-y_n)\odot P_{\sigma}:y_1,\ldots,y_n \in \RR, \sigma \in S_n^B\subseteq S_{2n},\}\\
        &=\R^n\rtimes S_n^B.
        \end{split}
        \]
\end{proposition}

\begin{proof}
    We consider the case $m=2n+1$, the case of even $m$ being similar. Let
\[
x=(x_0,\ldots,x_n,x_{-1},\ldots,x_{-n})\in \TT^m\quad\text{and}\quad A= D(y_0,\ldots,y_n,y_{-1},\ldots,y_{-n}) \odot P_{\sigma}  \in \GL_{m}(\T),
\]
then we have
    \begin{align*}
        q_m(A \odot x) = (y_{0} \odot 
        x_{\sigma^{-1}(0)})^{\odot 2} \oplus \bigoplus_{k=1}^{n} 
        y_{k} \odot x_{\sigma^{-1}(k)} \odot
        y_{-k} \odot x_{\sigma^{-1}(-k)}.
    \end{align*}
    Suppose that $\sigma \in S_n^B$, so that $\sigma(-k) =  -  \sigma(k)$ and also $\sigma^{-1}(-k) =  -  \sigma^{-1}(k)$ for all $k=0,\ldots,n$, and that $y_{-k}=-y_k$ for all $k=0,\ldots,n$. In particular, $\sigma(0) = 0$ and $y_{0} = 0$. It follows that $q_m(A \odot x) = q_m(x)$ for all $x \in \T^m$, i.e.~  $A \in O_{m}(\T)$. 
    Conversely, assume that $A \in O_{m}(\T)$. Choose $x \in \T^m$ such that $x_{0} \in \R$ and $x_i = \infty $ for all $i \neq 0$. From $q_m(A \odot x) = q_m(x)$ we obtain $\sigma(0) = 0$ and $y_{0} = 0$. Now let $j \neq 0$ and choose $x \in \T^m$ such that $x_j, x_{-j} \in \R$ and $x_i = \infty $ for all $i \in [\pm n] \cup \{0\} \setminus \{j, -j\}$. Since $q_m(A \odot x) = q_m(x) = x_j \odot x_{-j}$  it follows that $y_{i} \odot x_{\sigma^{-1}(i)} \odot y_{-i} \odot x_{\sigma^{-1}(-i)} = x_j \odot x_{-j}$ for $i$ such that $\sigma^{-1}(i) = j$. Hence, $\sigma^{-1}(-i) = - \sigma^{-1}(i)$ and $y_{-i}  = - y_{i}.$ Since $j \neq 0$ was arbitrary, it follows that $\sigma^{-1} \in S_n^B$ and thus $\sigma \in S_n^B$ and $y_{-i} = -y_{i}$ for all $i \in [\pm n] \cup \{0\}$.
\end{proof}

We now define the tropical special orthogonal groups, informed by the characteristic-independent algebraic definitions. First, we note that the determinant of a tropical orthogonal matrix is zero, reflecting the fact that $\TT$ has no nontrivial roots of unity. For $m=2n+1$ odd, we define $\SO_{2n+1}(\T)$ to be the kernel of the determinant on $\O_{2n+1}(\TT)$, which is all of $\O_{2n+1}(\TT)$. For $m=2n$ even, we define $\SO_{2n}(\TT)$ as the kernel of the \emph{tropical Dickson invariant,}
\begin{equation*}
    \O_{2n}(\T) \longrightarrow \{\pm 1\}\quad \text{given by} \quad D(y_i,-y_i)\odot P_\sigma\longmapsto\sgn(\sigma),
\end{equation*}
which, in our setting, is simply the parity of the permutation.

\begin{definition} For $m=2n+1$, the \emph{tropical special orthogonal group} is
\[
\SO_{2n+1}(\TT)=\O_{2n+1}(\TT)=\RR^n\rtimes S_n^B.
\]
For $m=2n$, the \emph{tropical special orthogonal group} is
\[
\SO_{2n}(\TT)=\{D(y_i,-y_i)\odot P_\sigma\in \O_{2n}(\T):\sigma\in S_n^D=S_n^B\cap A_{2n}\}=\RR^n\rtimes S_n^D,
\]
where $S_n^D = S_n^B \cap A_{2n} \subseteq S_{2n}$ is the even signed permutation group.
    
\end{definition}

We now compute the tropical reductive groups associated to the root data of $\SO_{2n+1}$ and $\SO_{2n}$, respectively. The root datum $(M,R,\Mcheck,\Rcheck)$ of $\SO_{2n+1}$ is dual to that of $\Sp_{2n}$: the lattices are $M = \Mcheck = \Z^n$ with the standard pairing, the roots $R \subseteq M$ are the vectors $\pm e_i$ and $\pm e_i \pm e_j$ for $i \neq j$, and the Weyl group is the signed permutation group $S_n^B$. The root datum $(M,R,\Mcheck,\Rcheck)$ of $\SO_{2n}$ has lattices $M = \Mcheck = \Z^n$ with the standard pairing and roots $R = \Rcheck = \{\pm e_i \pm e_j| i \neq j\}$. Reflection in $e_i-e_j$ exchanges the $i$th and $j$th coordinates, while reflection in $e_i + e_j$ switches $e_i$ to $-e_j$ and $e_j$ to $-e_i$. Hence the Weyl group consists of all permutations of $n$ elements that switch an even number of their signs, hence it is isomorphic to the even signed permutation group $S^D_n$. Therefore, we have the following result.

\begin{proposition} The tropical reductive groups associated to the root systems of $\SO_{2n+1}$ and $\SO_{2n}$ are $\SO_{2n+1}(\TT)$ and $\SO_{2n}(\TT)$, respectively.
\label{prop:tropicalSO}
\end{proposition}

\subsection{Tropical $G_2$} \label{subsec:G2} As our last example, we define the tropical analogue of the group $G_2$. We recall (see~\cite{1988CohenHelminck} and references therein) that the reductive group $G_2(F)$ over an algebraically closed field $F$ (of characteristic $\neq 2,3$) can be constructed as the isotropy group of a generic alternating trilinear form on a seven-dimensional vector space. Specifically, let $V=F^7$ with standard basis $e_1,\ldots,e_7$. The group $\bGL(V)$ acts on the vector space $\bigwedge^3(F^*)$ with a unique open orbit, which contains the 3-form
\[
\omega=e_1^*\wedge e_3^*\wedge e_5^*+e_2^*\wedge e_4^*\wedge e_6^*+
e_1^*\wedge e_4^*\wedge e_7^*+e_2^*\wedge e_5^*\wedge e_7^*+
e_3^*\wedge e_6^*\wedge e_7^*.
\]
We then define
\[
G_2(F)=\{A\in \bGL(V):\omega(Av_1,Av_2,Av_3)=\omega(v_1,v_2,v_3)\mbox{ for all }v_1,v_2,v_3\in V\}.
\]
We now translate this definition into the tropical setting. Since $\TT$ has no subtraction, we replace $\omega$ with a cubic form using the same formula, in the same manner that a symmetric bilinear form may be replaced with the associated quadratic form:

\begin{definition} Define the tropical cubic form $c:\TT^7\to \TT$ by the formula
\[
c(x_1,\ldots,x_7)=x_1\odot x_3\odot x_5\oplus x_2\odot x_4\odot x_6\oplus x_1\odot x_4\odot x_7\oplus x_2\odot x_5\odot x_7\oplus x_3\odot x_6\odot x_7.
\]
We define
\[
G_2(\TT)=\{A\in \GL_7(\TT):c(Ax)=c(x)\mbox{ for all }x\in \TT^7\}.
\]
    
\end{definition}

We first describe $G_2(\TT)$ explicitly. Let $D_6\subseteq S_6$ be the group of symmetries of the regular hexagon, whose vertices are labeled 1 through 6 in order. The action of $D_6$ on $\RR^6$ by permutation of coordinates preserves the two-dimensional subspace
\[
U=\{(y_1,\ldots,y_6)\in \RR^6:y_1+y_3+y_5=y_2+y_4+y_6=y_1+y_4=y_2+y_5=y_3+y_6=0 \}\subseteq \RR^6,
\]
where we note that either of the two relations $y_1+y_3+y_5=0$ and $y_2+y_4+y_6=0$ is redundant. We extend the embedding $D_6\subseteq S_6$ to $D_6\subseteq S_7$ by acting trivially on the $7$. 

\begin{proposition} The group $G_2(\TT)$ is isomorphic to
\[
\begin{split}
G_2(\TT)&=\{D(y_1,\ldots,y_7)\odot P_{\sigma}\in \GL_7(\TT):(y_1,\ldots,y_6)\in U,y_7=0,\sigma\in D_6\subseteq S_7\}\\
&=\RR^2 \rtimes D_6.
\end{split}
\]
\end{proposition}

\begin{proof} Let $A=D(y_1,\ldots,y_6,0)\odot P_{\sigma}$ with $(y_1,\ldots,y_6)\in U$ and $\sigma \in D_6$. The verification that $c(Ax)=c(x)$ for any $x\in \TT^7$ is straightforward and left to the avid reader. For the converse implication, denote by $T$ the set of three-element subsets of $\{1,\ldots,7\}$, and let $T_0\subseteq T$ be the five-element subset indexing the monomials in $c$:
\[
T_0=\big\{\{1,3,5\},\{2,4,6\},\{1,4,7\},\{2,5,7\},\{3,6,7\}\big\}.
\]
It is elementary to verify that a permutation $\sigma\in S_7$ lies in $D_6$ if and only if $\sigma(I)\in T_0$ for all $I\in T_0$.

Now let $A=D(y_i)\odot P_{\sigma}\in G_2(\TT)$. If $\sigma\notin D_6$, then there exists $I\in T_0$ such that $\sigma^{-1} (I)\notin T_0$. Define $x=(x_1,\ldots,x_7)$ by $x_i=0$ if $i\in I$ and $\infty$ otherwise, then $x_a\odot x_b\odot x_c=0$ if $\{a,b,c\}=I$ and $\infty$ otherwise. Hence $c(x)=0$ but $c(Ax)=\infty\neq c(x)$, since each monomial in $c(Ax)$ has at least one infinite coordinate. Therefore $\sigma\in D_6$.

We similarly verify that $(y_1,\ldots,y_6)\in U$ and $y_7=0$. Since we already know that $P_{\sigma^{-1}}\in G_2(\TT)$, we may replace $A$ with $D(y_1,\ldots,y_6,y_7)$. If $y_7\neq 0$, then setting $x_1=\cdots=x_6=0$ and $x_7=2|y_7|$ we get $c(x)=2|y_7|$ and $c(Ax)=2|y_7|+y_7\neq c(x)$. Similarly, the five linear expressions in the $y_i$ defining $U$ correspond to the five monomials in $c$. If any of these expressions are nonzero, we can pick $x$ such that $c(x)$ is minimized at the corresponding monomial and such that $c(x)\neq c(Ax)$. This concludes the proof.

\end{proof}

We now recall the root datum $(M,R,\Mcheck,\Rcheck)$ of $G_2$. The lattices $M$ and $\Mcheck$ are the hexagonal lattices embedded in $\RR^2$ with the standard Euclidean product, the roots are the 12 lattice points closest to the origin, and the Weyl group is $D_6$, acting by symmetries of the lattice. Comparing with the description of $G_2(\TT)$ given above, we obtain the following result.

\begin{proposition} The tropical reductive group associated to the root system of $G_2$ is $G_2(\TT)$. \label{prop:tropicalG2}
\end{proposition}

This concludes our study of tropical reductive matrix groups.

\section{Tropical principal bundles}

We now define the tropical analogue of a principal bundle on an algebraic curve. A \emph{metric graph} $\Ga$ is a metric space obtained by identifying the edges of a finite graph, called a \emph{model} of $\Ga$, with real intervals of given positive lengths. A \emph{free cover} $\Ga'\to \Ga$ of metric graphs is a covering space in the topological sense that preserves the metric structure; equivalently, a free cover is a harmonic morphism having local degree one at all points of $\Ga'$. Free covers are the only maps between metric graphs that we will consider (our restricted framework does not allow us to consider dilated harmonic morphisms).

\subsection{Tropical $G$-covers and torsors over the Weyl group}\label{subsect:principalbundles} Let $\Ga$ be a metric graph and let \( \mathcal{G} \) be a sheaf of (possibly non-abelian) groups on \( \Gamma \). We recall that a \textit{\( \mathcal{G} \)-torsor} on \( \Gamma \) is a sheaf of \( \mathcal{G} \)-sets $F$ such that \( \Gamma \) can be covered by open subsets \( U \) for which \( F|_U \) and \( \calG_U \cong \mathcal{G}|_U \) are isomorphic as sheaves of \( \calG_U \)-sets. Note that $\mathcal{G}$-torsors are classified up to isomorphism by the non-abelian cohomology set $H^1(\Ga,\mathcal{G})$, which is a pointed set with a distinguished element given by the trivial torsor on $\Gamma$.

We now define principal bundles on $\Gamma$ whose structure group $G=(\Mcheck \otimes_{\ZZ} \RR)\rtimes W_{\Phi}$ is the tropical reductive group associated to a root datum $\Phi=(M, R,\Mcheck,\Rcheck)$. We recall that a metric graph $\Ga$ comes equipped with a sheaf of \emph{harmonic functions} $\calH_{\Ga}$; these are the continuous real-valued piecewise linear functions with integer slopes whose outgoing (or incoming) slopes at every point add up to zero. Taking the tensor product, we obtain the sheaf $G_{\Gamma}=(\Mcheck \otimes_{\ZZ}\calH_{\Gamma})\rtimes W_{\Phi}$ of $G$-valued harmonic functions on $\Ga$.

\begin{definition} A \textit{tropical $G$-bundle} on $\Gamma$ is a torsor over the sheaf $G_{\Gamma}=(\Mcheck \otimes_{\ZZ}\calH_{\Gamma})\rtimes W_{\Phi}$.

\end{definition}

In Section~\ref{sec:moduliofbundles} below, we investigate the set $H^1(\Ga,G_{\Ga})$ of isomorphism classes of $G$-bundles on $\Ga$, which we interpret as the \emph{moduli space of $G$-bundles on $\Gamma$}. The purpose of this section is to describe $G$-bundles on $\Gamma$ in terms of line bundles on certain covers of $\Gamma$ that are determined by their associated $W_{\Phi}$-torsors. 

We first recall from~\cite{2022GrossUlirschZakharov} this description for the vector bundle case $G=\GL_n(\TT)=\RR^n\rtimes S_n$. Let $E$ be a $\GL_n(\TT)$-bundle on a metric graph $\Gamma$. Projecting onto the second component defines an $S_n$-torsor on $\Gamma$, which in turn defines a free cover $f:\Gamma'\to \Gamma$ of degree $n$. In~\cite{2022GrossUlirschZakharov}, it was shown that the $\RR^n$-part of the torsor $E$ is canonically determined by a tropical line bundle $L$ on $\Gamma'$, so that $E$ is the direct image of $L$ along $f$. We now extend this description to other tropical reductive groups.

First, we explain how to construct the covers. Let $G=\Mcheck_{\RR}\rtimes W$ be a tropical reductive group corresponding to the root datum $\Phi=(M,R,\Mcheck,\Rcheck)$, where $W=W_{\Phi}$ is the Weyl group. We choose a finite set $T$ with $n$ elements and an injective homomorphism $\rho:W\to S_T$, where $S_T$ is the permutation group of $T$. For every metric graph $\Ga$, we have an induced functor from the category of $W$-torsors on $\Ga$ to the category of $S_T$-torsors on $\Ga$, which, in turn, is equivalent to the category of degree $n$ covers of $\Ga$. We now discuss what additional structure is necessary to put on a degree $n$ cover to recover from it a $W$-torsor.

\begin{definition} A \textit{$\rho$-cover} of $\Gamma$ is a free degree $n$ cover $\Gamma'\to \Gamma$ together with an element
    \[
    \xi_x\in \Bij(T,\Gamma'_x)/W
    \]
    for each $x\in \Gamma$, where $\Gamma'_x$ is the fiber over $x\in \Gamma$ and $W$ acts on the set $\Bij(T,\Gamma'_x)$ of bijections between $T$ and $\Gamma'_x$ via $\rho$, such that each $x\in \Gamma$ has an open neighborhood $U$ on which there is a trivialization
    \[
    \phi\colon U\times T\xrightarrow{\cong} \Gamma'_U
    \]
    for which $\phi_y$ represents $\xi_y$ for every $y\in U$.

    A morphism of two $\rho$-covers $(\Gamma'\to \Gamma, (\xi_x)_x)$ and $(\widetilde \Gamma\to \Gamma, (\widetilde\xi_x)_x)$ is a morphism $f\colon \Gamma'\to \widetilde \Gamma$ of covers such that $f_x\circ \xi_x=\widetilde \xi_x$ for all $x\in \Gamma$.
\end{definition}    

\begin{proposition}
    \label{prop:category of rho-covers}
    The category of $\rho$-covers of $\Gamma$ is equivalent to the category of $W$-torsors on $\Gamma$.
\end{proposition}

\begin{proof}
    The trivial cover $I=\Gamma\times T$ becomes a $\rho$-cover by choosing $\xi_x=\id_{T}$ for all $x\in \Gamma$. Since $\rho$-covers are locally isomorphic to $I$, the category of $\rho$-covers is equivalent to the category of torsors over the sheaf of automorphisms of $I$. Hence it suffices to prove that the morphism 
    \[
    \underline W\longrightarrow \underline{\Aut}(I)\quad \text{given by} \quad w\longmapsto \id\times \rho(w) \ , 
    \]
    of sheaves on $\Gamma$ is an isomorphism\footnote{The two categories in question are neutral gerbes with global objects $I$ and $\underline W$, respectively}. The automorphism group of the cover $I_U$ (ordinary cover, not $\rho$-cover for now) of a connected open subset $U$ of $\Gamma$ is precisely $S_T$. So it suffices to show that $\sigma\in S_T$ defines a morphism of $\rho$-covers if and only if $\sigma\in \rho(W)$. To see this, we observe that $\sigma$ defines a morphism of $\rho$-covers if and only if $\sigma$ is in the same $W$-orbit as the identity in $\Bij(T,T)$, which happens if and only if there exists a $w\in W$ with
    \[
    \sigma=\id_{T}\circ \rho(w)=\rho(w) \ .
    \]
\end{proof}

We now work out this correspondence for the tropical reductive groups that we considered in Section~\ref{sec:reductivegroups}.

\begin{example}
\label{exa:phocovers}
    \begin{itemize}
        \item []
        
        \item [$\SL_n$:] Here we choose $T=[n]$ and an isomorphism $\rho\colon W\to S_n$, so that $\Bij([n],\Ga'_x)/W$ is a singleton for every fiber $\Ga'_x$. Therefore, the category of $W$-torsors on $\Gamma$ is equivalent to the category of degree $n$ covers on $\Gamma$.
        
        \item[$\PGL_n$:] As for $\SL_n$, the Weyl group $W$ is isomorphic to $S_n$, and $W$-torsors are equivalent to degree $n$ covers.
        
        \item [$\Sp_{2n}:$] Here $T=[\pm n]$ is the $2n$-element set with a fixed-point-free sign involution $\iota:[\pm n]\to [\pm n]$. The Weyl group $W$ of $\Sp_{2n}$ is the signed permutation group $S_n^B$, which comes with a natural embedding $\rho:S_n^B\hookrightarrow S_T$. Let $\mathrm{Inv}(S)$ denote the set of fixed-point-free involutions of a set $S$ of size $2n$. Then there is a natural isomorphism 
        \[
        \Bij([\pm n], S)/W \xlongrightarrow{\sim} \mathrm{Inv}(S) \quad \text{given by} \quad \overline\xi \longmapsto \xi\circ \iota\circ\xi^{-1} \ .
        \]
        It follows that the category of $W$-torsors is equivalent to the category of degree $2n$ covers together with a fixed-point-free involution. 
        
        \item [$\SO_
        {2n}$:] Here $T=[\pm n]$ as above, the Weyl group $W$ is the even signed permutation group $S_n^D$, and the image of the embedding $\rho:S^D_n\hookrightarrow S_T$ lies in the alternating group $A_T$. Because we have
        \[
        S^D_{n}=S^B_{n}\cap A_{2n} \ ,
        \]
        we obtain, for every $2n$-element set $S$, a natural bijection
        \[
        \Bij(T,S)/W \xlongrightarrow{\cong} \mathrm{Inv}(S)\times \Bij(T,S)/A_T \ .
        \]
        We recall that, given a degree $m$ free cover $\Ga'\to \Ga$ defined by an $S_m$-torsor, the \emph{orientation cover} $O(\Ga')\to \Ga$ is the degree 2 free cover defined by taking the quotient by $A_m$. In other words, if $\Gamma'_x$ is the fiber of a cover $\Gamma'\to \Gamma$ over a point $x\in \Gamma$, then $\Bij([\pm n],\Gamma'_x)/A_{2n}$ is the fiber over $x$ of the associated orientation cover. We thus obtain an equivalence of categories between the category of $W$-torsors and the category of degree $2n$ covers together with a fixed-point-free involution and a trivialization of the orientation cover.

        \item [$\SO_{2n+1}$:] In this case, $T=\{-n,\ldots,n\}$ and the sign involution acts with a fixed point. The Weyl group $W$ of $\SO_{2n+1}$ is $S_n^B$, the same as for $\Sp_{2n}(\TT)$ but now viewed as lying in the larger group $S_T$. The category of $W$-torsors on $\Ga$ is equivalent to the category of degree $2n+1$ covers of $\Ga$ together with an involution having a unique fixed point in every fiber. Removing the fixed points (which form a copy of $\Ga$), we obtain a degree $2n$ cover together with a fixed-point-free involution, as for $\Sp_{2n}$.
        
        \item[$\mathrm{G}_2$:] Here $T=\{1,\ldots,6\}$ and the image of $\rho\colon W\to S_6$ is the dihedral group $D_6$. Therefore, for every $6$-element set $S$, the set $\Bij(T,S)/W$ is identified with the possible arrangements of six distinct keys on a keychain, in other words the set of labelings of the vertices of a regular hexagon by elements of $T$, modulo rotations and reflections. Hence a $W$-torsor on $\Gamma$ is a degree $6$ free cover $\Gamma'\to \Gamma$ together with a locally trivial identification of the points of each fiber with the keys on a fixed keychain.  
    \end{itemize}
\end{example}

\subsection{Tropical $G$-bundles via line bundles on covers} \label{subsec:pushforward} We now upgrade the injective homomorphism $\rho:W\to S_n$ to a \emph{representation of tropical reductive groups}
\[
F=(f,\rho):G\longrightarrow\GL_n(\TT) \ ,
\]
where $G=\Mcheck_{\RR}\rtimes W$ is our tropical reductive group, $\GL_n(\TT)=\RR^n\rtimes S_n$, and the lattice map $f:\Mcheck \to \ZZ^n$ is injective. We obtain, for any metric graph $\Gamma$, a morphism from the category of $(\Mcheck\otimes \mc H_\Gamma)\rtimes W$-torsors on $\Gamma$ to the category of $(\ZZ^n\otimes\mc H_\Gamma)\rtimes S_n$-torsors on $\Gamma$; that is to say, tropical vector bundles on $\Gamma$. As described in~\cite{2022GrossUlirschZakharov}, the category of tropical vector bundles on $\Gamma$ is equivalent to the category of free covers $\Ga'\to \Ga$ together with a tropical line bundle on $\Ga'$; we refer to such a pair as a \emph{multi-line bundle} on $\Ga$. In particular, a $(\Mcheck\otimes \mc H_\Gamma)\rtimes W$-torsor on $\Ga$ induces a cover $\Ga'\to \Ga$ and a line bundle on $\Ga'$. We now describe the extra structure needed on the multi-line bundle to recover the category of $(\Mcheck\otimes \mc H_\Gamma)\rtimes W$-torsors.

First, we temporarily consider a broader category of tropical groups. Let $W$ be a finite group acting on a lattice $\Mcheck$. We call the semidirect product $\Mcheck_{\RR}\rtimes W$ a \emph{tropical linear group}. Similarly, a $\ZZ$-linear homomorphism $f:\Mcheck_1\to \Mcheck_2$ and a group homomorphism $\phi:W_1\to W_2$ satisfying $\phi(g)(f(m))=f(g(m))$ define a \emph{homomorphism of tropical linear groups}
\[
F=(f,\phi):\Mcheck_{1,\RR}\rtimes W_1 \to \Mcheck_{2,\RR}\rtimes W_2\quad \text{given by} \quad F(m,g)=(f(m),\phi(g)) \ .
\]
Given a metric graph $\Ga$ and a tropical linear group $G=\Mcheck_{\RR}\rtimes W$, we consider torsors over the sheaf $G_{\Ga}=(\Mcheck\otimes_{\ZZ}\calH_{\Ga})\rtimes W$ as in Section~\ref{subsect:principalbundles}.

Given a morphism $F\colon G\to H$ of tropical linear groups and a $G_\Ga$-torsor $E$ on $\Gamma$, we define the induced $H_\Ga$-torsor, denoted by $F_*(E)$ or $E_H$ if $F$ is clear from the context, by 
\[
    F_*(E)=E_H=(E\times H_{\Gamma})/G_{\Gamma} \ ,
\]
where $G_{\Gamma}$ acts by the rule $g.(e,h)=(ge,hF(g)^{-1})$.

\begin{proposition}
\label{prop:fiber product of neutral gerbes}
    Let $G_i=\Mcheck_{i,\R}\rtimes W_i$ for $i=1,2,3$ be tropical linear groups and let 
    \[  
        G_1 \xlongrightarrow{F}
            G_2 \xlongrightarrow{H}
                G_3  
    \]
    be morphisms of tropical linear groups. Assume the following:
    \begin{enumerate}
        \item $F$ is injective.
        \item There is a sublattice $\check{L}\subseteq \Mcheck_3$ and a subgroup $Y\subseteq W_3$ such that $Y\check L\subseteq \check L$, the image of the map $\Mcheck_1\rtimes W_1\to \Mcheck_2\rtimes W_2$ is the preimage of $\check{L}\rtimes Y$, and such that $\check{L}\rtimes Y$ and the image of $\Mcheck_2\rtimes W_2$  generate $\Mcheck_3\rtimes W_3$.

    \end{enumerate}
       Moreover, let $\Gamma$ be a metric graph and let $K=\check L_{\RR}\rtimes Y\subseteq G_3$ be the tropical linear group determined by $\check{L}$ and $Y$. Then there is an equivalence of categories between the category of $G_{1,\Gamma}$-torsors on $\Gamma$ and the category of triples $(T, T', \phi)$ consisting of an $G_{2,\Gamma}$-torsor $T$, a $K_\Gamma$-torsor $T'$ and an isomorphism $(T')_{G_{3,\Gamma}}\xrightarrow{\phi} T_{G_{3,\Gamma}}$, where $(T')_{G_{3,\Gamma}}$ and $T_{G_{3,\Gamma}}$ are the $G_{3,\Gamma}$-torsors induced by the homomorphisms $K\hookrightarrow G_3$ and $H:G_2\to G_3$, respectively.
\end{proposition}

\begin{proof}
    Let $I=(G_2, K,  K_{G_3}\cong G_3\xrightarrow{\id}G_3\cong (G_2)_{G_3})$ be the trivial element. By definition of the category, we have
    \[
    \underline \Aut(I) = \underline \Aut(G_2)\times_{\underline \Aut(G_3)}\underline \Aut(K) \ .
    \]
    As we have $\underline\Aut(L_i)=L_i$ and $\underline\Aut(K)=K$ it follows that 
    \[
    \underline\Aut(I)= G_2\times_{G_3} K \cong H^{-1}K \cong G_1 \ .
    \]
    It thus suffices to prove that every object $(T,T',\phi)$ is locally isomorphic to $I$. Working locally, we may assume that we are given trivializations
    \begin{equation*}
     G_2 \xlongrightarrow{\psi} T \quad \text{and}
     \quad K \xlongrightarrow{\chi} T' \ .
    \end{equation*}
    These induce trivializations
    \begin{equation*}
        G_3\cong (G_2)_{G_3}\xlongrightarrow{\psi_{G_3}}T_{G_3}  \quad \text{and} \quad
        G_3\cong K_{G_3}\xrightarrow{\chi_{G_3}}T'_{G_3} \ .    
    \end{equation*}
    Let $ \delta= \psi_{G_3}^{-1} \circ \phi \circ \chi_{G_3}$. This is a section of $\underline \Aut(G_3)=G_3$, and by assumption we can locally decompose it as $\delta=   H(\alpha) \cdot \beta^{-1}$ for some $\alpha\in G_2$ and $\beta\in K$. Denote by $r_{\alpha}:G_2\to G_2$, $r_{H(\alpha)}:G_3\to G_3$, and $r_{\beta}:K\to K$ right multiplication by $\alpha$, $H(\alpha)$, and $\beta$, respectively. We claim that $(\psi\circ r_{\alpha} ,\chi\circ r_{\beta})$ defines an isomorphism $I\to (T,T',\phi)$. And indeed, we have
    \[
    \phi\circ (\chi\circ r_{\beta})_{G_3}=
    \phi \circ \chi_{G_3} \circ (r_\beta)_{G_3}= 
    \psi_{G_3} \circ \delta \circ (r_\beta)_{G_3}=
    \psi_{G_3}\circ r_{H(\alpha)}=
     (\psi\circ r_\alpha)_{G_3} \circ \id_{G_3}\ .
    \]
\end{proof}

Let $G_i=(\Mcheck_i,W_i)$ for $i=1,2,3$ be tropical linear groups. We say that a sequence of morphisms $F=(f,\phi):G_1\to G_2$ and $H=(H,\psi):G_2\to G_3$ is a \emph{short exact sequence of tropical linear groups} if $0\to \Mcheck_1\xrightarrow{f}\Mcheck_2\xrightarrow{h}\Mcheck_3\to 0$ and $1\to W_1\xrightarrow{\phi}W_2\xrightarrow{\psi}W_3\to 1$ are short exact sequences. 

\begin{corollary}
\label{cor:exact sequence}
    Let 
     \[
    1\longrightarrow
        G_1 \longrightarrow
            G_2 \longrightarrow
                G_3 \longrightarrow
                    1
     \]
     be a short exact sequence of tropical linear groups and let $\Gamma$ be a metric graph. Then there is an equivalence of categories between the category of $G_{1,\Ga}$-torsors on $\Gamma$ and the category of pairs $(T,\phi)$ consisting of an $G_{2,\Ga}$-torsor $T$ on $\Gamma$ and a trivialization $G_3\xrightarrow{\phi} T_{G_3}$.
\end{corollary}

\begin{proof}
    This follows directly from Proposition \ref{prop:fiber product of neutral gerbes} with $K=1$.
\end{proof}

We now describe $G$-covers on a metric graph $\Ga$ in terms of line bundles, in the case when the lattice map $f:\Mcheck \to \ZZ^n$ associated to the chosen representation $F:G\to \GL_n(\TT)$ is the identity map.

\begin{corollary}
\label{cor:torsors as multi-line bundles with rho-structure}
    Let $\rho:W\to S_n$ be an injective homomorphism and let $\Gamma$ be a metric graph. Then  there is an equivalence of categories of $ \mc H_\Gamma^n \rtimes W$-torsors on $\Gamma$ and degree $n$ multi-line bundles on $\Gamma$ together with the structure of a $\rho$-cover on the underlying cover.
\end{corollary}

\begin{proof}
    This follows directly from Proposition \ref{prop:category of rho-covers} combined with Proposition \ref{prop:fiber product of neutral gerbes} applied to the sequence
    \[
     \R^n \rtimes W \to 
        \R^n \rtimes S_n\to 
            S_n 
    \]
    with $K=W$.
\end{proof}

We also consider the more general setting where the map $f:\Mcheck\to \ZZ^n$ associated to $F:G\to \GL_n(\TT)$ is injective. Let $\Gamma$ be a metric graph and let $G=\Mcheck_{\RR}\rtimes W$ be a tropical reductive group. Since the inclusion $W\to G$ splits canonically, every $G_\Gamma$-torsor $T$ has an associated $W$-torsor $T_W$, which in turn has an associated $G_\Ga$-torsor $(T_W)_{G_\Gamma}$ that we denote by $T^0$.

\begin{corollary}
\label{cor:exact sequence on lattices}
    Let
    \[ 
        \Mcheck_{1,\RR}\rtimes W \xrightarrow{F=(f,\phi)}
            \Mcheck_{2,\RR}\rtimes W\xrightarrow{H=(h,\psi)}
                \Mcheck_{3,\RR}\rtimes W' 
    \]
    be a sequence of tropical reductive groups with the following properties:
    \begin{enumerate}
        \item The kernel of $\psi$ contains the image of $\phi$. 
        \item The sequence $0\to \Mcheck_1\xrightarrow{f} \Mcheck_2\xrightarrow{h} \Mcheck_3\to 0$ is exact.
    \end{enumerate}
For any metric graph $\Gamma$, there is an equivalence of categories between the category of $(\Mcheck_1\otimes_{\ZZ} \mc H_\Gamma)\rtimes W$-torsors on $\Gamma$ and the category of pairs $(T,\xi)$ consisting of an $(\Mcheck_2\otimes_{\ZZ} \mc H_\Gamma) \rtimes W$-torsor $T$ and an isomorphism $T_{(\Mcheck_3\otimes \mc H_\Gamma)\rtimes W'}^0\xrightarrow{\xi }T_{(\Mcheck_3\otimes \mc H_\Gamma)\rtimes W'}$.
\end{corollary}

\begin{proof}
    This follows directly from Proposition \ref{prop:fiber product of neutral gerbes} when taking $K=W'$.
\end{proof}

\label{subsec:Gbundlesviacovers}

We now explicitly describe $G$-bundles on a metric graph $\Ga$ in terms of line bundles on covers of $\Ga'$, in the case when $G$ is one of the tropical reductive groups described in Section \ref{sec:reductivegroups}.

\begin{example}
    \label{exa:bundlescovers}
    \begin{itemize}
    \item[]
        \item [$\SL_n$:] The standard representation $\SL_n(\TT)\to \GL_n(\TT)$, where the map $S_n\to S_n$ on the Weyl groups is the identity but the lattice map is $\ZZ^n_0\to \ZZ^n$, fits into the short exact sequence of tropical reductive groups
        \[
        1 \longrightarrow
            \SL_n(\TT) \longrightarrow 
                \GL_n(\TT) \longrightarrow
                    \RR \longrightarrow
                     0
        \]
        where the second map $\GL_n(\TT)\to \RR=\GL_1(\TT)$ is the tropical determinant. Denote by $\det(T)$ the tropical line bundle on $\Ga$ associated to a $\GL_n(\TT)$-bundle $T$.  By Corollary \ref{cor:exact sequence}, there is an equivalence of categories between $\SL_n(\TT)$-bundles on $\Ga$ and pairs $(T, \phi)$ consisting of a $\GL_n(\TT)$-torsor $T$ on $\Ga$ and a trivialization $\phi:\mc H_\Gamma\to \det(T)$. The $\GL_n(\TT)$-bundle $T$ may be represented by a multi-line bundle on a degree $n$ cover $\Ga'\to \Ga$, whose fiber over $x\in \Gamma$ is the disjoint union of $\R$-torsors $L_1,\ldots, L_n$. Under this identification, the fiber of $\det(T)$ over $x$ can be naturally identified with $L_1\otimes\cdots \otimes L_n$.

        \item[$\Sp_{2n}$:] The description of $\Sp_{2n}(\TT)$ as a matrix group over $\TT$ gives a representation $(f,\rho):\Sp_{2n}(\TT)\to \GL_{[n]\sqcup -[n]}(\TT)$, which determines a sequence of tropical linear groups
        \[
            \Sp_{2n}(\TT) \longrightarrow
                \R^{[n]\sqcup -[n]}\rtimes S_n^B \longrightarrow 
                \R^n \rtimes S_n=\GL_n(\TT)\ .
        \]
        Here the first morphism is the identity map $S_n^B\to S_n^B$ on the groups and the diagonal embedding $f:\ZZ^n\to \ZZ^{[n]\sqcup -[n]}$, $e_i\mapsto e_i-e_{-i}$ on the lattices, while the second morphism is the natural quotient map $S_n^B\to S_n$ on the groups and the map $(x_i)_{i\in [n]\sqcup -[n]}\mapsto (x_i+x_{-i})_{i\in [n]}$ on the lattices. By Corollary \ref{cor:exact sequence on lattices} the category of $\Sp_{2n}(\T)$-bundles on $\Ga$ is equivalent to the category of pairs $(T,\phi)$ consisting of a $\R^{[n]\sqcup -[n]}\rtimes S_n^B$-bundle $T$ on $\Ga$ and an isomorphism $\phi:T_{\GL_n}^0\to T$. By Corollary \ref{cor:torsors as multi-line bundles with rho-structure}, a $\R^{[n]\sqcup -[n]}\rtimes S_n^B$-bundle $T$ on $\Ga$ is equivalent to a line bundle on a $\rho$-cover $\Ga'\to \Ga$. We have seen in Example \ref{exa:phocovers} that $\rho$-covers are in turn equivalent to pairs $(\Gamma'\to \Gamma ,\iota)$ consisting of a degree $2n$ cover and a fixed-point-free involution $\iota$ of the cover. Hence the $\R^{[n]\sqcup -[n]}\rtimes S_n^B$-bundle is equivalent to a multi-line bundle $(\Gamma'\to \Gamma, L)$ with a fixed-point-free involution $\iota$ of $\Gamma'\to \Gamma$. The associated $\GL_n(\TT)$-torsor on $\Ga$ corresponds to the degree $n$ cover $\Gamma'/\iota\to \Gamma$ equipped with the line bundle $(L\otimes \iota^{-1} L)/\iota$, whose fiber over $\overline x\in \Gamma'/\iota$ is $L_x\otimes L_{\iota(x)}$. 
        In summary, the category of $\Sp_{2n}(\TT)$-bundles on $\Ga$ is equivalent to the category of quadruples $(f:\Gamma'\to \Gamma, \iota, L, \phi)$ consisting of a degree $2n$ free cover $f$, a fixed-point-free involution $\iota$ of $f$, a tropical line bundle $L$ on $\Gamma'$, and a trivialization $\phi:\mc H_{\Gamma'/\iota}\to \ (L\otimes \iota^{-1}L)/\iota$. Figure~\ref{fig:sp4cover} below illustrates this in the case $n=2$: an  $\Sp_4(\TT)$-bundle on a tropical elliptic curve $\Ga$, represented by a degree $4$ cover $\Ga' \to \Ga$ with a fixed-point-free involution $\iota:\Gamma' \to \Gamma'$ of the cover and a compatible tropical line bundle $L$ on $\Ga'$. Concretely, the line bundle $L$ can be represented by a divisor $D$ on $\Gamma'$ such that $(D + \iota^{-1}D)/\iota \sim 0$.

\begin{figure}[ht]
  \centering
  {
  \begin{tikzpicture}[scale=1.8,declare function={f(\x)=0.2*sin(\x)+\x/1000;},
rubout/.style={/utils/exec=\tikzset{rubout/.cd,#1},
decoration={show path construction,
      curveto code={
       \draw [white,line width=\pgfkeysvalueof{/tikz/rubout/line width}+2*\pgfkeysvalueof{/tikz/rubout/halo}]
        (\tikzinputsegmentfirst) .. controls
        (\tikzinputsegmentsupporta) and (\tikzinputsegmentsupportb)  ..(\tikzinputsegmentlast);
       \draw [line width=\pgfkeysvalueof{/tikz/rubout/line width},shorten <=-0.1pt,shorten >=-0.1pt]
        (\tikzinputsegmentfirst) .. controls
        (\tikzinputsegmentsupporta) and (\tikzinputsegmentsupportb) ..(\tikzinputsegmentlast);  
      }}},rubout/.cd,line width/.initial=0.4pt,halo/.initial=0.2pt]


\draw[line width=1.0pt,white]   plot[variable=\x,domain=-50:1300,samples=160,smooth] ({cos(\x)},{f(\x)}) to[out=0,in=195] cycle;
\draw[line width=0.4pt,black]   plot[variable=\x,domain=-50:1300,samples=160,smooth] ({cos(\x)},{f(\x)}) to[out=0,in=195] cycle;

\draw (0,-2) arc(-90:270:1cm and 0.2cm);

\draw[-stealth]  (0,-0.4) -- (0,-1.4);

\coordinate (P) at ({cos(1300)},{f(1300)});
\fill (P) circle (1pt) node[above right,xshift=0.3pt,yshift=0.2pt,scale=0.8] {$-3$};

\coordinate (Q) at ({cos(580)},{f(580)});
\fill (Q) circle (1pt) node[above right,xshift=0.3pt,yshift=0.2pt,scale=0.8] {$3$};

\coordinate (R) at ({cos(300)},{f(300)});
\fill (R) circle (1pt) node[below right,xshift=0.3pt,yshift=0.7pt,scale=0.8] {$-7$};

\coordinate (S) at ({cos(1020)},{f(1020)});
\fill (S) circle (1pt) node[below right,xshift=0.6pt,yshift=0.5pt,scale=0.8] {$7$};

\draw[->]
  (1.2,{f(300)+0.27})
  .. controls (1.4,{0.5*(f(300)+f(1020))+0.27})
  .. (1.2,{f(1020)+0.27}) 
  node[below right,xshift=3.5pt,yshift=-3pt,scale=1] {$\iota$};

\end{tikzpicture}
 
  \caption{An $\Sp_4(\TT)$-bundle on a tropical elliptic curve.}
  \label{fig:sp4cover}
  }
\end{figure}
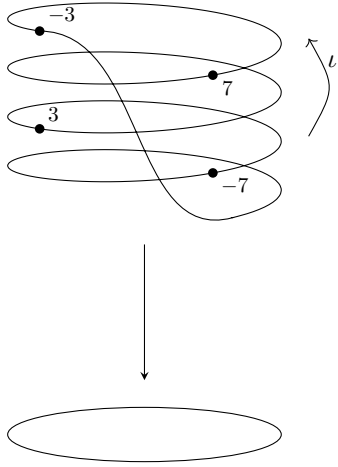
        \item[$\SO_{2n+1}$:] 
        Since $\SO_{2n+1}(\T)$ and $\Sp_{2n}(\T)$ are isomorphic as tropical linear groups, we obtain the same description as for $\Sp_{2n}(\T)$-bundles. 
        \item[$\SO_{2n}$:] We have a short exact sequence of tropical linear groups
        \[
        1 \longrightarrow 
            \SO_{2n}(\TT) \longrightarrow 
                \Sp_{2n}(\TT) \longrightarrow   
                    S_2 \longrightarrow 
                        1 \ .
        \]
        By Corollary \ref{cor:exact sequence}, the category of $\SO_{2n}(\TT)$-bundles on a metric graph $\Gamma$ is equivalent to the category of quintuples $(f:\Gamma'\to \Gamma, \iota, L, \phi, \psi)$, where $(f:\Gamma'\to \Gamma, \iota, L, \phi)$ is as in the $\Sp_{2n}(\TT)$-case and $\psi$ is a trivialization of the orientation double cover $O(\Ga')\to \Ga$ associated to $f$ (compare with the $\mathrm{SO}_{2n}$-case in Example \ref{exa:phocovers}).

        \item[$\mathrm{G}_2$:]
        In Subsection~\ref{subsec:G2} we gave an explicit presentation of $G_2(\TT)$ inside $\GL_7(\TT)$, which restricts to a homomorphism $F=(f,\rho):G_2(\TT)\to \GL_6(\TT)$. Here $\rho\colon W=D_6\to S_6$ is the standard embedding and the image of $f:\ZZ^2\to \ZZ^6$ is the sublattice $M$ given by the four relations
        \[
        x_1+x_4=0 \ ,\quad x_2+x_5=0 \ , \quad x_3+x_6=0 \ , \quad x_1+x_3+x_5=0 \ .
        \]
        Let $V$ be the lattice between $M$ and $\Z^6$ given by the first three of these relations. Then there is a sequence of tropical linear groups
        \[
        G_2(\TT) \longrightarrow V_{\RR}\rtimes D_6 \longrightarrow \Sp_2 (\TT) \ ,
        \]
        where the second map sends $((x_i)_i,\sigma)$ to $(x_1+x_3+x_5,\sgn(\sigma))$. We note that an element $\sigma\in D_6$ is odd (as a permutation in $S_6$) if and only if it exchanges the sets $\{1,3,5\}$ and $\{2,4,6\}$, hence this map is a group homomorphism. Applying Corollary \ref{cor:exact sequence on lattices}, we see that the category of $G_2(\TT)$-bundles on a metric graph $\Ga$ is equivalent to the category of $V_{\RR}\rtimes D_6$-bundles on $\Ga$ together with an involution-invariant trivialization  of the line bundle of the associated $\Sp_2(\T)$-cover.
        
        To describe $V_{\RR}\rtimes D_6$-bundles in terms of covers, we note that the image of $D_6$ in $S_6$ is contained in $S_3^B$, the signed permutation group that preserves the involution corresponding to reflecting the hexagon through the origin. Hence we consider the sequence
        \[
        V_{\RR}\rtimes D_6 \to \Sp_6(\TT) \to S^B_3
        \]
        of tropical linear groups, where the first map on the lattices is the identity. We apply Proposition \ref{prop:fiber product of neutral gerbes} with $\check{L}$ trivial and $Y=D_6$ and obtain that the category of $V_{\RR}\rtimes D_6$-torsors is equivalent to the category of $\Sp_6(\TT)$-covers, together with an order $6$ cycle graph structure on the fibers such that opposite vertices in the cycle graph are interchanged by the involution of the $\Sp_6(\TT)$-cover. The associated $\Sp_2(\TT)$-cover has one branch for each of the distinguished triangles, which are interchanged by the involution, and the fibers of the line bundle on the domain of the $\Sp_2(\TT)$-cover are the tensor products of the line bundles on the branches in each triangle.

        In summary, the category of $G_2(\TT)$-bundles on a metric graph $\Ga$ is equivalent to the category of quadruples $(\Gamma'\to \Gamma, L, \phi,\psi )$, where $\Gamma'\to \Gamma$ is a degree $6$ free cover with a locally trivial identification of each fiber with the Star of David, $L$ is a line bundle on $\Gamma'$, whereas $\phi$ is a trivialization of $(L\otimes \iota^{-1}L)/\iota$ on $\Gamma'/\iota$, where the involution $\iota:\Ga'\to \Ga'$ exchanges the opposite vertices in each star, and $\psi$ is an $\iota$-invariant trivialization of the line bundle on the domain of the associated $\Sp_2(\TT)$-cover whose fibers correspond to the two triangles.
    \end{itemize}
    
\end{example}

\section{Degree and stability in the algebraic setting}
\label{sec:Degree and stability in the algebraic setting}

Let $\bfG$ be a reductive linear algebraic group over an algebraically closed field $k$, and let $X$ be a smooth projective curve over $k$. In this section, we recall the notions of degree and stability for $\bfG$-bundles on $X$ and the associated stratification on the moduli space $\Bun_\bfG(X)$ of $\bfG$-bundles on $X$, following  \cite{schieder2015harder}. We also review the explicit description of the moduli spaces of stable and semistable $\bfG$-bundles on an elliptic curve, following~\cite{fratila2016stack} and~\cite{Fratila_revisitingelliptic}.

\subsection{Degree}

We fix $\bfT\subseteq \mathbf{B}\subseteq \bfG$ a maximal torus and a Borel subgroup. We denote by $M = \mathbb{X}^*(\bfT)$ and $\Mcheck = \mathbb{X}_{*}(\bfT)$ the character and cocharacter lattices of $\bfT$ and denote by $R\subseteq M$ and $\check{R}\subseteq \Mcheck$ the roots and coroots. The algebraic fundamental group of $\mathbf G$ is defined as the quotient of the lattice of cocharacters by the lattice generated by the coroots: $\pi_1(\bfG) = \Mcheck / \langle \alphacheck: \alphacheck \in \Rcheck \rangle$. For $\lambdacheck\in \Mcheck$, we denote the corresponding element of the fundamental group by $\lambdacheck_\bfG\in \pi_1(\bfG)$. With these definitions, we have:
\begin{equation*}\begin{split}
\pi_1(\bGL_n) &= \Z \ , \\ 
\pi_1(\mathbf{SL}_n) &= 1 \ , \\ 
\pi_1(\bPGL_n) &= \Z / n\Z \ , \\
\pi_1(\mathbf{Sp}_{2n}) &= 1 \ , \\
\pi_1(\mathbf{SO}_{2n}) &= \begin{cases} \Z / 4\Z \ , & \text{if }n\mbox{ odd}\ , \\
 (\Z / 2\Z)^2 \ , & \text{if }n\mbox{ even} \ .
 \end{cases}
 \end{split}
\end{equation*}

The choice of $\mathbf{B}$ determines a partition $R=R^+\sqcup R^-$ into positive and negative roots, as well as a set of simple positive roots $\{\alpha_i:i\in D\}\subseteq R^+$, where $D$ is the set of vertices of the Dynkin diagram. Hence $\mathbf B$ gives us a partial order on the cocharacter lattice $\Mcheck$: we say that $\lambdacheck \leq \mucheck$ if $\mucheck - \lambdacheck$ is a nonnegative linear combination of positive coroots. This order extends naturally to real coefficients $\Mcheck_\RR = \Mcheck \otimes_\Z \R.$ 

For a parabolic subgroup $\mathbf{P} \subseteq \bfG$, we denote by $\bfL = \bfP/U(\bfP)$ its Levi quotient, where $U(\bfP)$ is the unipotent radical.
A parabolic subgroup $\bfP$ of $\bfG$ containing $\bfB$ corresponds to a subset $D_\bfP\subseteq D$ of the simple roots (the Dynkin diagram of $\bfL$), in particular $\bfG$ itself corresponds to $D$. We denote the algebraic fundamental group of $\bfP$ by $\pi_1(\bfP) \coloneqq \pi_1(\bfL) = \Mcheck/\langle \check{\alpha}_i:i\in D_\bfP\rangle$, and for $\lambdacheck\in \Mcheck$ we denote by $\lambdacheck_\bfP$ the corresponding element of $\pi_1(\bfP)$.

We denote by $\Bun_\bfG(X)$ the moduli stack of $\bfG$-bundles on $X$, that is to say, \'etale $\bfG$-torsors on $X$. It is well-known that the connected components of $\Bun_\bfG(X)$ are in bijection with $\pi_1(\bfG)$ (see \cite[Theorem 5.8]{Hoffmann} for a proof). We call elements of $\pi_1(\bfG)$ \emph{degrees}, and for $\lambdacheck_\bfG\in \pi_1(\bfG)$ we denote the corresponding connected component by $\Bun^{\lambdacheck_\bfG}_\bfG(X)$. For a parabolic subgroup $\bfP \subseteq \bfG$ with Levi quotient $\bfL$, the moduli spaces $\Bun_\bfP(X)$ and $\Bun_\bfL(X)$ have the same connected components and are in bijection with $\pi_1(\bfP)=\pi_1(\bfL)$. 

\subsection{The slope map and semistability}

We now define the \emph{slope map} $\phi_\bfP: \pi_1(\bfP) \to \Mcheck_\R$ for a parabolic subgroup $\bfB \subseteq \bfP \subseteq \bfG$, for details we refer to \cite{schieder2015harder}. We recall that the center of a reductive group $\bfL$ with maximal torus $\bfT$ is the intersection
\[
    Z(\bfL) = \bigcap_{\alpha \text{ root of }\bfL} \ker(\alpha) \subseteq \bfT \ .
\] The natural map $Z(\bfL) \to \bfT$ induces a map on the cocharacters. Taking the quotient by the coroots, we obtain a map
\[
    \mathbb{X}_{*}(Z(\bfL)) \to \mathbb{X}_{*}(\bfT) = \Mcheck \to \Mcheck / \langle \alphacheck \text{ coroot of } \bfL \rangle = \pi_1(\bfL)=\pi_1(\bfP) \ ,
\]
which becomes an isomorphism $\mathbb{X}_{*}(Z(\bfL))_\RR \simeq \pi_1(\bfP)_\RR$ after tensoring with $\RR$.

\begin{definition}
    For a parabolic subgroup $\bfB \subseteq \bfP \subseteq \bfG$ with Levi subgroup $\bfL$ we define the \emph{slope map} $\phi_\bfP: \pi_1(\bfP) \to \Mcheck_\RR$ as 
     \[
     \pi_1(\bfP) \to \pi_1(\bfP)_\RR \cong \mathbb{X}_{*}(Z(\bfL))_\RR \to \Mcheck_\RR \ .
    \]
The \emph{slope} of a $\bfP$-bundle $F_\bfP \in \Bun_\bfP^{\lambdacheck_\bfP}(X)$ is the element $\phi_\bfP(\lambdacheck_\bfP)$. We say that a $\bfG$-bundle $F \in \Bun^{\lambdacheck_\bfG}_\bfG(X)$ on $X$ is \emph{(semi)stable} if for any proper parabolic subgroup $\bfP \subseteq \bfG$ and for any reduction $F_\bfP$ of $F$ to $\bfP$ of degree $\lambdacheck_\bfP$ we have
\[    
    \phi_\bfP(\lambdacheck_\bfP) \,\,\ \substack{<\\(\leq)} \,\,\  \phi_\bfG(\lambdacheck_\bfG) \ .
\]
\end{definition}
For $\bfG= \bGL_n$, these notions reduce to the standard notions of slope and stability for vector bundles on $X$ (see~\cite[Section 2.2.4]{schieder2015harder}). 

We will use the following result of \Fratila~(see Lemma 2.12 in~\cite{fratila2016stack}). First, we note that for a parabolic subgroup $\bfB\subseteq \bfP\subseteq \bfG$, taking the quotient by the remaining roots defines a natural quotient map $p:\pi_1(\bfP)\to \pi_1(\bfG)$.

\begin{theorem}[{\cite[Lemma 2.12]{fratila2016stack}}]
\label{thm:Fratila existence of minimal parabolic}
Let $\bfG$ be a reductive group and let $\lambdacheck_\bfG\in\pi_1(\bfG)$ be a degree. Then there exists a parabolic subgroup $\bfP\subseteq \bfG$ and a degree $\lambdacheck_\bfP\in \pi_1(\bfP)$ such that $\phi_\bfP(\lambdacheck_\bfP)=\phi_\bfG(\lambdacheck_\bfG)$ and $p(\lambdacheck_\bfP)=\lambdacheck_\bfG$, and which is minimal with this property. The parabolic subgroup $\bfP$ is unique up to conjugation.

\end{theorem}

We note that $\bfP$ is explicitly given by the set of roots $\{i \in D : \langle \omega_i ,\phi_\bfG(\lambdacheck_\bfG) - \lambdacheck \rangle \notin \Z \} $, where $\lambdacheck \in \Mcheck$ is a lift of $\lambdacheck_\bfG$ and the $\omega_i$ are the fundamental weights.

\subsection{Stable and semistable $\bfG$-bundles over an elliptic curve} 
\label{subsec:stability for bundles on elliptic curves}

We now recall \Fratila's description of the moduli space of semistable $\bfG$-bundles on an elliptic curve $X$ (see~\cite{fratila2016stack} and~\cite{Fratila_revisitingelliptic}). For $\lambdacheck_\bfG\in \pi_1(\bfG)$ we denote by $\mc M_\bfG^{\lambdacheck_\bfG, \semist}(X)$ and $\mc M_\bfG^{\lambdacheck_\bfG,\st}(X)$ the moduli spaces of semistable and stable $\bfG$-bundles on $X$ of degree $\lambdacheck_\bfG$, respectively, and we usually suppress the $X$ from the notation.

We recall that the derived subgroup $\bfG^{\mathrm{der}}=[\bfG,\bfG]$ of $\bfG$ is semisimple with the same Weyl group $W$. The intersection $\bfT\cap \bfG^{\mathrm{der}}$ is a maximal torus in $\bfG^{\mathrm{der}}$ and its character and cocharacter lattices are given by $M /\Rcheck^\perp$ and $\langle \Rcheck\rangle^\sat$, respectively (see \cite[Corollary 8.1.9]{Springer_linalggroups}). The \emph{cocenter} $Z^c(\bfG)$ of $\bfG$ is the quotient $\bfG/\bfG^{\mathrm{der}}$, which is a torus with character and cocharacter lattices $R^\perp$ and $\Mcheck/\langle \Rcheck\rangle^\sat$. The quotient map $\bfG\to Z^c(\bfG)$ is called the \emph{determinant} (for $\bfG=\bGL_n$, we have $Z^c(\bfG)=\G_m$ and this map is the usual matrix determinant). 

First, we have the following explicit description of the moduli space of stable $\bfG$-bundles on $X$:

\begin{theorem}[{\cite[Corollary 4.3]{fratila2016stack}}, {\cite[Theorem 1.4]{Fratila_revisitingelliptic}}] 
\label{thm:Fratilla determinant is isomorphism}
Let $\bfG$ be a reductive group and let $\lambdacheck_\bfG\in \pi_1(\bfG)$. 

\begin{enumerate}
    \item The moduli space $\mc M_\bfG^{\lambdacheck_\bfG,\st}$ is nonempty only if $\bfG$ is of type $\prod_i A_{n_i-1}$.

    \item \label{item:stable} Suppose that $\bfG$ is of type $\prod_i A_{n_i-1}$, so that
    \[
    \bfG^{\ad}=\bfG/Z(\bfG)=\prod_i \bPGL_{n_i} \ , \quad \pi_1(\bfG^{\ad})=\prod_i {\Z/n_i\Z} \ .
    \]
    Then $\mc M_\bfG^{\lambdacheck_\bfG,\st}$ is nonempty if and only if the image of $\lambdacheck_\bfG$ in $\pi_1(\bfG^{\ad})$ is of the form $(d_i)_i$, where $\gcd(d_i,n_i)=1$ for all $i$. Furthermore, in this case the determinant map
\[
    \mc M_\bfG^{\lambdacheck_\bfG,\st} \to \mc M_{Z^c(\bfG)}^{\det(\lambdacheck_\bfG),\semist}
\]
is an isomorphism. 

\end{enumerate}
\end{theorem}

\begin{definition}\label{def:stable_degree}
A degree $\lambdacheck_\bfG\in \pi_1(\bfG)$ satisfying condition~\eqref{item:stable} above is called \emph{stable}.
\end{definition}

This description generalizes to semistable bundles.

\begin{theorem}[{\cite[Theorem 1.2]{Fratila_revisitingelliptic}}]
\label{thm:Fratilas results}
Let $\lambdacheck_\bfG\in \pi_1(\bfG)$. Let  $\bfL = \bfL_{\lambdacheck_\bfG}\subseteq \bfG$ be the Levi subgroup corresponding to the parabolic subgroup $\bfP\subseteq \bfG$ given by Theorem~\ref{thm:Fratila existence of minimal parabolic}, and let $\lambdacheck_\bfL=\lambdacheck_\bfP\in \pi_1(\bfL)=\pi_1(\bfP)$ be the corresponding degree. Then
\begin{enumerate}
    \item the inclusion $\bfL\subseteq \bfG$ induces a map $\pi\colon\mc M_\bfL^{\lambdacheck_\bfL,\semist}\to \mc M_\bfG^{\lambdacheck_\bfG,\semist}$, and all semistable $\bfL$-bundles in $\mc M_\bfL^{\lambdacheck_\bfL,\semist}$ are stable. 
    \item $\pi$ is finite and generically Galois with Galois group $W_{\bfL,\bfG}=N_\bfG(\bfL)/\bfL$, where $N_\bfG(\bfL)$ is the normalizer of $\bfL$ in $\bfG$. 
    \item the quotient map
    \[
        \mc M_\bfL^{\lambdacheck_\bfL,\st}/W_{\bfL,\bfG}\to \mc M_{\bfG}^{\lambdacheck_\bfG,\semist}
    \]
    is an isomorphism. 
   \end{enumerate}
\end{theorem}

\section{Moduli spaces of tropical principal bundles}

\label{sec:moduliofbundles}

Let $\Ga$ be a metric graph and let $G$ be a tropical reductive group. In this section, we define tropical notions of degree and stability for $G$-bundles on $\Ga$ and describe the connected components of the moduli space of tropical $G$-bundles on $\Ga$. On a tropical elliptic curve (that is to say, a metric circle) we describe the main components of the moduli spaces of stable and semistable bundles, establishing tropical analogues of Theorems~\ref{thm:Fratilla determinant is isomorphism} and~\ref{thm:Fratilas results}.

\subsection{Degree and stability in the tropical setting}

Let $\Phi=(M,R,\Mcheck, \Rcheck)$ be a root datum and let $G =  (\Mcheck\otimes_{\ZZ}\RR) \rtimes W_{\Phi}$ be the corresponding tropical reductive group. In analogy with the algebraic setting, we define the \emph{fundamental group} of $\Phi$ as
\[
\pi_1(\Phi)=\Mcheck/\langle \Rcheck\rangle \ . 
\]
In a slight abuse of notation we refer to $\pi_1(\Phi)$ as the \emph{fundamental group of $G$} and denote it by $\pi_1(G)$. We denote the image of a cocharacter $\lambdacheck\in \Mcheck$ in the fundamental group by $\lambdacheck_G\in \pi_1(G)$.

To define the degree of a tropical $G$-bundle, we first show that there is a natural projection from $\Mcheck \rtimes W_\Phi$ to the fundamental group of $G$.

\begin{lemma}
    \label{lem:subgroup is normal}
    The group $\langle \Rcheck\rangle \rtimes W_\Phi$ is normal in $\Mcheck \rtimes W_\Phi$. Hence there exists a well-defined surjective homomorphism 
    \[
    \Mcheck \rtimes W_\Phi \to \Mcheck / \langle \Rcheck \rangle =\pi_1(G)  \quad\text{given by} \quad (m ,w)\mapsto \overline{m}\ .
    \]

\end{lemma}

\begin{proof}
    Let $(r,w) \in \langle \Rcheck\rangle \rtimes W_\Phi$ with $r \in \langle \Rcheck \rangle$ and $w \in W_\Phi$. Then for $(m,v) \in \Mcheck \rtimes W_\Phi$ we have $(m,v)\cdot(r,w)\cdot(m,v)^{-1} = (m+v.r-vwv^{-1}.m, vwv^{-1})$. To show that $m+v.r-vwv^{-1}.m \in \langle \Rcheck \rangle$ we observe that $\Rcheck$ is invariant under the $W_\Phi$-action, i.e., $v.r \in \langle \Rcheck \rangle$ and that $w.m -m  \in \langle \Rcheck \rangle$ for every $w \in W_\Phi$ and $m \in \Mcheck$. The latter fact we show by induction on the length of $w$. Indeed, if $w = s_\alpha$ for a root $\alpha$, then $w.m-m = \langle \alpha,m \rangle \alphacheck \in \langle \Rcheck \rangle$. Let $l(w) > 1$, and write $w = s_\alpha w'$ with $l(w') = l(w)-1$. Then $w.m-m = s_\alpha(w'.m-m) + s_\alpha.m-m \in \langle \Rcheck \rangle.$ This shows that $\langle \Rcheck\rangle \rtimes W_\Phi$ is normal in $\Mcheck \rtimes W_\Phi$.
    Hence there is a homomorphism 
    \[
    \Mcheck \rtimes W_\Phi \to (\Mcheck \rtimes W_\Phi) / (\langle \Rcheck \rangle \rtimes W_\Phi) \cong   \Mcheck / \langle \Rcheck \rangle =\pi_1(G)\ . \qedhere
    \]
\end{proof}

We recall that the set of isomorphism classes of $G$-bundles on $\Ga$ is the non-abelian cohomology set $H^1(\Ga,G_{\Ga})$, where $G_{\Ga}$ is the sheaf $(\Mcheck\otimes_{\ZZ}\calH) \rtimes W_{\Phi}$ of $G$-valued harmonic functions on $\Ga$. We define a degree map 
\[
H^1(\Ga,G_\Ga)\to \pi_1(G)
\]
as follows. Recall that the sheaf $\Omega$ of \emph{harmonic 1-forms} on $\Ga$ is the cokernel of the map $\RR\to \calH$, where $\RR$ is the constant sheaf. The quotient map $\calH\to \Omega$ (which sends a harmonic function to its derivative) and the homomorphism of Lemma~\ref{lem:subgroup is normal} induce maps of sheaves
\[
G_{\Ga}=(\Mcheck\otimes_{\ZZ}\calH) \rtimes W_{\Phi}\longrightarrow  (\Mcheck\otimes_{\ZZ}\Omega) \rtimes W_{\Phi}\longrightarrow  \Omega\otimes_{\ZZ} \pi_1(G) \ ,
\]
which in turn induce a map of pointed sets $H^1(\Ga,G_{\Ga})\to H^1(\Ga,\Omega\otimes_{\ZZ} \pi_1(G))$ in cohomology. Since $\Omega\otimes_{\ZZ}\pi_1(G)$ is a sheaf of abelian groups, its $H^1$ is a group that we can compute using the universal coefficient theorem. On a graph, all cohomology groups vanish in dimensions 2 and above. By Lemma 1.4 in~\cite{2022GrossUlirschZakharov} there is a natural isomorphism $H^1(\Ga,\Omega)\cong \ZZ$, hence we obtain
\[
H^1(\Gamma,\Omega\otimes_{\ZZ} \pi_1(G))\cong H^1(\Gamma,\Omega)\otimes_{\ZZ}\pi_1(G)\cong \ZZ\otimes_{\ZZ}\pi_1(G)\cong \pi_1(G) \ .
\]
The degree map is obtained by composing all of the above maps.

\begin{definition}
\label{def:degree}
Let $F\in H^1(\Gamma, G_{\Ga})$ be a $G$-bundle on a metric graph $\Gamma$. The \emph{degree} $\check{\lambda}_G\in \pi_1(G)$ of $F$ is the image of $F$ in $\pi_1(G)$.
\end{definition}

The degree is closely related to the \emph{determinant map}, which is the quotient morphism $\det\colon G\to \Mcheck_\R/\langle \Rcheck\rangle_\R$, which is well-defined by Lemma \ref{lem:subgroup is normal}. The induced morphism 
\[
\Mcheck/\langle \Rcheck\rangle =\pi_1(G)\xrightarrow{\det_*} \pi_1(\Mcheck_\R/\langle\Rcheck\rangle_\R)= \Mcheck/\langle \Rcheck\rangle^\sat 
\]
is the quotient map that divides out the torsion of $\pi_1(G)$. It follows that if $F$ is a $G$-bundle, then the degree of $\det(F)$ is the torsion-free part of $\deg(F)$. In particular, if $\pi_1(G)$ is torsion-free, as for example in the case $G = \GL_n(\TT)$, the degree of $F$ coincides with the degree of $\det(F)$.  Consequently, Definition~\ref{def:degree} agrees with the notion of degree for tropical vector bundles introduced in Section 2.4 of~\cite{2022GrossUlirschZakharov} in the case $G=\GL_n(\TT)$.

To define stability, we first define parabolic subgroups in the tropical setting. Fix a splitting $R=R^+\sqcup R^-$ and let $\{\alpha_i\in R^+:i\in D\}$ be the set of simple roots. Let $D_P\subseteq D$ be a subset of the simple roots and let $W_{\Phi,P}\subseteq W_{\Phi}$ be the subgroup generated by the reflections in $\{\alpha_i:i\in D_P\}$.

\begin{definition} The \emph{standard parabolic subgroup} corresponding to $D_P\subseteq D$ is the tropical reductive group
\[
P= (\Mcheck\otimes_{\ZZ}\RR) \rtimes W_{\Phi,P} \subseteq G =  (\Mcheck\otimes_{\ZZ}\RR) \rtimes W_{\Phi} \ .
\]
\end{definition}

We note that there is no notion of unipotent groups in the tropical setting, hence a parabolic subgroup is the same as its associated Levi subgroup. In particular, parabolic subgroups are reductive, contrary to the algebraic situation.

We now define slope in analogy with the algebraic setting. First, we compute the center of a tropical reductive group.

\begin{lemma}
\label{lem:center}
The center of a tropical reductive group $G=\Mcheck_{\R}\rtimes W_{\Phi}$ is
\[
Z(G)=R^\perp=\{(m,1)\in \Mcheck_\R\rtimes W_{\Phi} : \langle R,m\rangle=0\} \ .
\]
In particular, we have
\[
Z(G)\cong\Mcheck_{\R}/\langle \Rcheck\rangle_{\RR}=\pi_1(G)_{\RR} \ .
\]
\end{lemma}
\begin{proof}
    Let $(k,w) \in Z(\Mcheck_\R \rtimes W_\Phi)$. Then for all $m \in \Mcheck_\R$ we have $(k,w) \cdot (m,1) = (m,1) \cdot (k,w)$ which is equivalent to $(k+w.m,w) = (m+k,w)$. Now, $w.m = m$ for all $m \in \Mcheck$ implies $w = 1$ since the action of $W$ on $\Mcheck$ is free.
    To see that $(k,1) \in Z(\Mcheck_\R \rtimes W_\Phi)$ if and only if $k \in R^\perp=\bigcap_{\alpha\in R} \ker(\alpha)$, let $\alpha \in R$ and $m \in \Mcheck$. Then $(k,1) \cdot (m,s_\alpha) = (m,s_\alpha) \cdot (k,1)$ is equivalent to $k+m = m + s_\alpha.k$. The latter is equivalent to $k = s_\alpha.k = k - \langle \alpha, k \rangle\alphacheck$ which is equivalent to $k \in \ker(\langle  \alpha, \cdot \rangle) = \ker \alpha$.

    The ``in particular'' statement follows from the decomposition $\Mcheck_\R=R^\perp\oplus\langle \Rcheck\rangle_\R$.
\end{proof}

We note that by definition, the fundamental group of a standard parabolic subgroup is
\[
\pi_1(P)=\Mcheck/\langle \check{\alpha_i}: i \in D_P \rangle \ .
\]

\begin{definition} Let $P$ be a standard parabolic subgroup corresponding to $D_P\subseteq D$, and let $F_P$ be a $P$-bundle on $\Gamma$ of degree $\check{\lambda}_P \in \pi_1(P)$. The \emph{slope} of $F_P$ is the image of $\lambdacheck_P$ under the map 
\begin{equation*}
     \phi_P\colon \pi_1(P) = \Mcheck/\langle \check{\alpha_i}: i \in D_P \rangle\longrightarrow 
     \Mcheck_\R/\langle \check{\alpha_i}: i \in D_P \rangle_\RR \cong
     Z(\Mcheck_\R \rtimes W_{\Phi,P})\longrightarrow 
     \Mcheck_\R \ ,
\end{equation*}
where the second map is the isomorphism given in Lemma~\ref{lem:center}. 
\end{definition}

Let $F$ be a $G$-bundle on $\Ga$ and let $P\subseteq G$ be a parabolic subgroup. We say that $F$ \emph{admits a reduction to $P$} if there exists a $P$-bundle $F_P$ on $\Gamma$ such that $i_*(F_P)=F$, where $i:P\to G$ is the inclusion.

\begin{definition} \label{def:slopestability}
    Let $G =  (\Mcheck\otimes_{\ZZ}\RR) \rtimes W_{\Phi}$ be a tropical reductive group and let $F$ be a $G$-bundle on $\Gamma$ of degree $\check{\lambda}_G\in \pi_1(G)$. We say that $F$ is \textit{(semi)-stable} if for every subset $D_P \subseteq D$ and for every element $\check{\lambda}_P \in \Mcheck_P$ such that $F$ admits a reduction $F_P$ to $P$ of degree $\check{\lambda}_P\in \pi_1(P)$ we have 
    \[
    \phi_P(\check{\lambda}_P) \,\,\ \substack{<\\(\leq)} \,\,\  \phi_G(\check{\lambda}_G) \ .
    \]
\end{definition}

For $G = \GL_n(\TT)$, this definition agrees with the notion of slope semi-stability for tropical vector bundles introduced in~\cite{2022GrossUlirschZakharov}. Recall that the slope of a tropical vector bundle $E$ on a metric graph $\Gamma$ of degree $\lambdacheck_G \in \pi_1(G)\cong\Z$ is defined to be the quotient $\mu(E) = \frac{\lambdacheck_G}{\rk(E)}$. The slope of $E$ is given by the formula
\[
\phi_G(\lambdacheck_G) = (\mu(E),\ldots,\mu(E)) \in \R^n
\]
where $\Mcheck_\R$ is canonically identified with $\R^n$ (see Section 2.2.4 in  \cite{schieder2015harder}). More generally, let $F_P$ be a $P$-bundle of degree $\lambdacheck_P \in \pi_1(P)$ for some parabolic $P \subseteq \GL_n(\T)$. Let $E_1,\ldots, E_m$ be the summands of $E$.
Then the slope $\phi_P(\lambdacheck_P)$ is given as
\[
\phi_P(\lambdacheck_P) = (\mu(E_1),\ldots,\mu(E_1),\ldots,\mu(E_m),\ldots,\mu(E_m)) \in \R^n \ ,
\]
where each $\mu(E_i)$ is repeated $\rk(E_i)$ times. Using the same arguments as in the algebraic setting, it is then elementary to show that Definition~\ref{def:slopestability} is equivalent to Definition 5.1 of~\cite{2022GrossUlirschZakharov}.

\subsection{Moduli of tropical $G$-bundles} We now interpret the set $H^1(\Gamma,G_{\Gamma})$ of isomorphism classes of $G$-bundles on $\Gamma$ as a moduli space. 

\label{subsec:moduli}

\begin{definition} Let $\Ga$ be a metric graph and let $G$ be a tropical linear group. The \emph{moduli space of $G$-bundles on $\Gamma$} is the set $\calM_G(\Ga)=H^1(\Gamma,G_{\Gamma})$ of isomorphism classes of $G$-bundles on $\Ga$. If $G$ is reductive, then for a degree $\lambdacheck_G\in \pi_1(G)$, we denote by $\calM^{\lambdacheck_G}_G(\Ga)$ the set of isomorphism classes of $G$-bundles having degree $\lambdacheck_G$. As in the algebraic setting, we normally suppress $\Ga$ from the notation. 
\end{definition}

We now describe $\mc M_G$ as the non-abelian \v{C}ech cohomology set $\check{H}^1(\calU(\Ga),G_{\Ga})$, computed using a canonical (except when $\Ga$ is a circle) acyclic cover $\calU(\Ga)$ of $\Ga$ (in Theorem~\ref{thm:description of component of tropical moduli space}, we give a more explicit description of $\mc M_G$ as a disjoint union of finite quotients of torsors under tropical abelian varieties). We fix an oriented simple loopless model for $\Gamma$, also denoted by $\Gamma$ by abuse of notation, by placing a vertex at the midpoint of each loop, and similarly splitting all multiedges. For an edge $e\in E(\Ga)$ denote by $U_e\subseteq \Gamma$ the corresponding open subset (not containing the root vertices of $e$). Similarly, for a vertex $v\in V(\Ga)$, denote by $U_v$ the \emph{star} around $v$, which is the union of $v$ and the $U_e$ for all edges $e$ incident to $v$. We call $\calU(\Gamma) = \{U_v\}_{v \in V(\Gamma)}$ the \emph{star cover} of $\Gamma$. 

First, we explicitly describe the sections of $G_{\Gamma}$. For an oriented edge $e\in E(\Ga)$, identify $U_e$ with the interval $(0,\ell(e))$. A section $g_e\in G_{\Gamma}(U_e)$ is an affine linear function with integer slopes valued in $G$:
\[
g_e:U_e=(0,\ell(e))\longrightarrow G \quad \text{given by}\quad g_e(t)=A_et+B_e\ ,
\]
where $A_e\in \Mcheck \rtimes W_{\Phi}$ and $B_e\in \Mcheck_{\RR}\rtimes W_{\Phi}$.
Similarly, let $v\in V(\Gamma)$ be a vertex with incident edges $e_1,\ldots,e_k$ oriented outwards. Identifying each $U_{e_j}$ with $(0,\ell(e_j))$, a section $f_v\in G_{\Gamma}(U_v)$ is a $k$-tuple of functions
\[
f_{v,e_i}(t)=A_{v,e_i}t+B_v\quad\text{for}\quad i=1,\ldots,k \ ,
\]
where $A_{v,e_i}\in \Mcheck\rtimes W_{\Phi}$ and $B_v\in \Mcheck_{\RR}\rtimes W_{\Phi}$, 
such that the $A_{v,e_1}+\cdots+A_{v,e_k}=0$. 

We now explicitly describe the set $\check{H}^1(\Gamma,G_{\Gamma})$ using the star cover. For $v,w\in V(\Gamma)$, the intersection $U_v\cap U_w$ is $U_{vw}$ if there is an edge $vw\in E(\Gamma)$, and empty otherwise. Furthermore, all triple intersections are empty. Since each $U_{vw}$ is contractible, the star cover is acyclic for the sheaf $G_{\Gamma}$. Therefore
\[
    \check H^1(\Ga,G_\Ga)=\check H^1(\mc U(\Gamma), G_{\Gamma}) = \big\{(g_{vw}) \in \Pi_{vw\in E(\Gamma)} G_{\Gamma}(U_{vw})\big\} / \sim \ ,
\]
where $(g_{vw}) \sim (g'_{vw})$ if there exists a tuple $(f_v) \in \Pi_{v \in V(G)} G_{\Gamma}(U_v)$ such that $g_{vw} = f_v g'_{vw} f_w^{-1}$. We note that all triple intersections are empty, hence the cocycle condition is trivially satisfied.

\begin{example} 
[$G$-bundles on a metric circle] 
\label{ex:explicit description of torsors on circle}
Let $j >0$ be a real number, and let $\Gamma = \R / j\Z$ be a circle of length $j$. Let $(G,l)$  be an oriented model with two vertices $v_1,v_2 \in V(G)$ and consider the associated cover $\calU(G) = \{U_{v_1}, U_{v_2}\}$. Note that here $\calU(G)$ is not the star cover since the intersection $U_{v_1} \cap U_{v_2} = e_1 \sqcup e_2$ is the disjoint union of the two open edges $e_1,e_2$. An element in $\calH(e_i)$ is an affine linear function with integer slope valued in $G$, hence $G_\Gamma(e_i) =  (\Mcheck \times \Mcheck_\R) \rtimes W_\Phi$, where the Weyl group acts diagonally. Thus,
    \[
    G_{\Gamma}(U_{v_1} \cap U_{v_2}) = G_{\Gamma}(e_1) \times G_{\Gamma}(e_2) = ((\Mcheck \times M_\R) \rtimes W_\Phi) \times ((\Mcheck \times M_\R) \rtimes W_\Phi) \ ,
    \]
    and hence $\check H^1(\calU(G),G_{\Ga})$ is the set of tuples
    \[
    (a,b) \in ((\Mcheck \times M_\R) \rtimes W_\Phi) \times ((\Mcheck \times M_\R) \rtimes W_\Phi)
    \]
    modulo the relation $(a,b) \sim \big(f_1|_{e_1}af_2^{-1}|_{e_1}, f_1|_{e_2}bf_2^{-1}|_{e_2}\big)$ for $f_i \in G_{\Gamma}(U_{v_i})$. Let
    \[f_1 = (k,\beta,w) \in G_{\Gamma}(U_{v_1}) = (\Mcheck \times M_\R) \rtimes W_\Phi
    \]
    and denote the translated function by $f_1' = (k,\beta + kj,w)$. By setting $f_2 = f'_1 b$, the map $(a,b) \mapsto ab^{-1}$ yields an isomorphism of pointed sets
    \[
    H^1(\Gamma,G_{\Gamma}) = G_{\Gamma}(e_1) / \sim \ ,
    \]
    where $c \sim f_1|_{e_1}cf'^{-1}_1|_{e_1}$ for $c \in G_{\Gamma}(e_1) =(\Mcheck \times M_\R) \rtimes W_\Phi, f_1 \in G_{\Gamma}(U_{v_1})$.
    
    Explicitly, let $(m,\alpha,w), (k,\beta,w) \in  (\Mcheck \times M_\R) \rtimes W_\Phi$. Then 
    \[
    \check H^1(\Gamma,G_{\Gamma}) = ((\Mcheck \times M_\R) \rtimes W_\Phi) / \sim 
    \]
    where
    \[
    (m,\alpha,w) \sim (k,\beta,v)(m,\alpha,w)(k,\beta+jk,v)^{-1}\sim (k+v.m-vwv^{-1}.k, \beta+v.\alpha-vwv^{-1}.(\beta+jk),vwv^{-1}) \ .
    \]

The isomorphism class of a $W$-torsor $\tau$ on a metric circle $\Gamma$ corresponds to the conjugacy class of $w\in W$. The automorphism group $\Aut(\tau)$ can then be identified with the centralizer $C_W(w)$.

\end{example}

We recall that in Section~\ref{subsec:pushforward} we defined the  pushforward of a $G$-bundle along a homomorphism of tropical reductive groups. 

\begin{lemma}
\label{lem:surjective morphism of groups yields surjective morphism of moduli}
Let $F=(f,\phi)\colon G\to H$ be a morphism of tropical linear groups such that both $f$ and $\phi$ are surjective. Then the induced morphism
\[
    F_*\colon \mc M_G\longrightarrow \mc M_H
\]
is surjective. 
\end{lemma}

\begin{proof}
Let $\mc U$ be the star cover of $\Gamma$. Then, since triple intersections of sets in $\mc U$ are empty and the cocycle condition is trivially satisfied, the surjectivity of $f$ implies that the map
\[
    \mc M_G \cong \check H^1 (\mc U,G_\Ga)\xlongrightarrow{\check F} \check H^1(\mc U,H_\Ga)\cong \mc M_H \ ,
\]
which agrees with $F_*$, is surjective as well. 
\end{proof}

Let $G= \Mcheck_\R \rtimes W$ be a tropical reductive group. The quotient morphism $q\colon G\to W$ yields a map
\[
     q_*\colon\mc M_G\longrightarrow \mc M_W \ .
\]
For $\tau\in \mc M_W$, the fiber under this map 
\[
     \mc M_{G,\tau}\coloneqq q_*^{-1}\{\tau\}
\]
is the set of isomorphism classes of $G$-bundles $E$ on $\Ga$ whose associated $W$-torsor $q_*(E) = E_W$ is isomorphic to $\tau$. We denote by $\wmc M_{G,\tau}$ the set of isomorphism classes of pairs $(E,\phi)$, where $E$ is a $G$-bundle and $\phi\colon E_W\to \tau$ is an isomorphism. For a degree $\lambdacheck_G\in \pi_1(G)$, we denote by $\mc M^{\lambdacheck_G}_{G,\tau}$ and $\wmc M^{\lambdacheck_G}_{G,\tau}$ the corresponding moduli spaces of bundles with degree $\lambdacheck_G$. There is a canonical right action of $\Aut(\tau)$ on $\widetilde{\mc M}_{G,\tau}$ coming from postcomposing $\phi$ with an automorphism of $\tau$, and by definition we have 
\[
    \mc M_{G,\tau}=\widetilde{\mc M}_{G,\tau}/\Aut(\tau) \ .
\]

\begin{proposition}
\label{prop:fiber over trivial W-torsor}
Denoting the trivial $W$-torsor on $\Gamma$ by $W_\Gamma$, we have natural bijections 
\begin{align*}
    \widetilde{\mc M}_{G,W_\Gamma}
        &= \Pic(\Gamma)\otimes_\Z \Mcheck, \\[3pt]
    \mc M_{G,W_\Gamma}
        &= (\Pic(\Gamma)\otimes_\Z \Mcheck)/W \ .
\end{align*}

\end{proposition}

\begin{proof} By Corollary~\ref{cor:exact sequence}, the short exact sequence
    \[
    0\longrightarrow \Mcheck_\R\longrightarrow G\longrightarrow W\longrightarrow 1
    \]
    yields an equivalence of categories between $\Mcheck_\R$-bundles and pairs $(E,\phi)$ consisting of a $G$-bundle $E$ and an isomorphism $E_W\xrightarrow{\phi}W_\Gamma$. Therefore, there is a natural bijection
    \[
    \mc M_{\Mcheck_\R} \xrightarrow{\cong} \widetilde{\mc M}_{G,W_\Gamma} \ .
    \]
    We note that $\RR=\GL_1(\TT)$ as a tropical reductive group, hence $\mc M_\R(\Gamma)=\Pic(\Gamma)$. Since $\Mcheck$ is free, we have a canonical bijection
    \[
    \mc M_{\Mcheck_\R}\xrightarrow{\cong}\Pic(\Gamma)\otimes_\Z \Mcheck \ ,
    \]
    showing the first isomorphism. For the second, we note that $\Aut(W_\Gamma)=W$ because $\Gamma$ is connected.
\end{proof}

\begin{example}(\textit{$G$-bundles on metric trees})
\label{exa:bundlestrees}
Let $\Gamma$ be a compact and connected metric tree. Recall that up to isomorphism, there is exactly one line bundle $\calH_{\Gamma}(d)$ of degree $d$ on $\Gamma$, i.e., $\Pic(\Gamma) \cong \ZZ.$ Let $G = \Mcheck_\RR \rtimes W_\Phi$ be a tropical reductive group and let $E$ be a $G$-bundle on $\Gamma$. Since the fundamental group of $\Gamma$ is trivial, the associated $W$-torsor $E_W$ is isomorphic to the trivial torsor $W_\Gamma$. Hence, by Proposition \ref{prop:fiber over trivial W-torsor} we obtain a natural bijection 
\[
\calM_G \cong (\Pic(\Gamma) \otimes_\ZZ \Mcheck)/W \cong \Mcheck /W \ .
\]

This is a tropical analogue of a theorem of Grothendieck which states that given a split reductive group $\bfG$ and a maximal split torus $\bfT$, any $\bfG$-bundle on $\PP^1$ has a reduction of structure group to the maximal torus $\bfT$ unique up to the action of the Weyl group $W$ (see \cite{Grothendieck_vectorbundlesonP1},~\cite{1968Harder}, or~\cite[Theorem 0.3]{MartensThaddeus_Grothendieck}). If $G = \GL_n(\T)$, this means that every vector bundle on a metric tree splits as a direct sum of line bundles (see \cite[Example 3.3]{2022GrossUlirschZakharov}).
    
\end{example}

Our next goal is to describe the moduli space of $G$-bundles on a metric graph $\Gamma$ as a rational polyhedral space. We first note that the moduli space $\calM_G$ decomposes as a finite disjoint union by the isomorphism type of the associated $W$-torsor:
\[
\calM_G=\coprod_{\tau\in \calM_{W}}\calM_{G,\tau} \ .
\]
We now describe these moduli spaces.

\begin{theorem}
\label{thm:description of component of tropical moduli space}
    Let $\tau$ be a $W$-torsor on $\Gamma$. Then $\wmc M_{G,\tau}$ is a disjoint union of torsors under tropical abelian varieties. Therefore, $\mc M_{G,\tau}$ is the quotient of a disjoint union of torsors under tropical abelian varieties by the finite group $\Aut(\tau)$.
\end{theorem}

We note that the $\wmc M_{G,\tau}$ are rational polyhedral spaces, but the $\mc M_{G,\tau}$, and hence the moduli space $\mc M_G$, are only finite group quotients of rational polyhedral spaces. We first prove several preliminary lemmas.

Let $\tau$ be a $W$-torsor on $\Gamma$. Then its total space $\tau\to \Gamma$ is a free finite covering of graphs, and we can pull back $\tau$ itself and obtain a $W$-torsor on $\tau$. Of course, the total space of $\tau$ has a $W$-action and so it makes sense to talk about $W$-equivariant objects in a category of torsors over $\tau$. Recall that for a tropical reductive group $H$,  a $W$-equivariant $H$-bundle on $\tau$ is an $H$-bundle $E$ together with morphisms $m_w\colon E\to l_w^*E$ of $H$-bundles, one for each $w\in W$, where $l_w$ is multiplication by $w$ on the left on $\tau$ and which satisfy the obvious compatibility.

\begin{lemma}
\label{lem:pull-back of tau by tau}
    There is a $W$-equivariant isomorphism $\tau\times_\Gamma\tau\cong W\times \tau$, where we equip $W\times \tau$ with the $W$-action as follows: 
    \[
    w.(v,t)=(vw^{-1},wt) \ .
    \]
\end{lemma}

\begin{proof}
    The map 
    \[
        \tau\times_\Gamma\tau\to W\times \tau \quad \text{given by}\quad (t_1,t_2)\mapsto (t_1t_2^{-1}, t_2)
    \]
    on total spaces is clearly a morphism of torsors and $W$-equivariant.
\end{proof}

Consider the short exact sequence
\[
0\to
    \Mcheck_\R \to
        G \to 
            W\to 
                1
\]
By Corollary \ref{cor:exact sequence}, there is an induced bijection between $\mc M_{\Mcheck_\R}$ and $\wmc M_{G,W_\Gamma}$. In particular, the $\Aut(W_\Gamma)$-action on $\wmc M_{G,W_\Gamma}$ induces a right $\Aut(W_\Gamma)$-action on $\mc M_{\Mcheck_\R}$. Embedding $W$ into $\Aut(W_\Gamma)$ via right multiplication, we obtain a right $W$-action on $\mc M_{\Mcheck_\R}$. On the other hand, $W$ acts on $\Mcheck_\R$ by conjugation, and thus there exists an induced action
\[
\mc M_{\Mcheck_\R} \times W\to \mc M_{\Mcheck_\R}\ , \;\; (L,w)\mapsto L^w\coloneqq (c_w)_*L \ ,
\]
where $c_w\colon \Mcheck_\R\to \Mcheck_\R$ is given by $c_w(m)=m^w=w^{-1} m w$. We now show that the two $W$-actions on $\mc M_{\Mcheck_\R}$ coincide.

\begin{lemma}
\label{lem:meaning of conjugation}
    Let $w\in W$. Under the bijection $\wmc M_{G,W_\Gamma}\cong \mc M_{\Mcheck_\R}$, if a pair $(E,\phi)$ corresponds to an $\Mcheck_\R$-bundle $L$, then the pair $(E,\phi\circ w^{-1})$ corresponds to $L^w$. Here, we embed $W$ in $\Aut(W_\Gamma)$ via right multiplication.
\end{lemma}

\begin{proof}
We consider the induced morphisms of sheaves of groups
$\Mcheck_{\R,\Ga}  \xrightarrow{i} G_\Ga$ and $G_\Ga \xrightarrow{\pi} W$. Let $L$ be an $\Mcheck_\R$-bundle and let $E = L_{G_\Ga}$ be the induced $G$-bundle under the bijection $\wmc M_{G,W_\Gamma}\cong \mc M_{\Mcheck_\R}$. We recall that $E$ is the sheaf associated to the presheaf
\[
U\longmapsto  \big( G_\Ga(U) \times L(U) \big) \big/ \sim \ ,
\]
where the equivalence relation is $(g \cdot i(m)^{-1}, m \cdot x) \sim (g, x)$ for $m \in \Mcheck_\R(U)$, $g \in G_\Ga(U)$, and $x \in L(U)$. 
Now, for $w \in W$, the $\Mcheck_\R$-bundle $L^w$ as a sheaf is $L^w = L$ but with left $\Mcheck_{\R,\Ga}$-action $m \cdot_w x = c_w^{-1}(m) \cdot x = wmw^{-1} \cdot x$. The induced $G$-bundle $E^w = (L^w)_{G_\Ga}$ is the sheaf associated to
\[
U\longmapsto  \big( G_\Gamma(U) \times L^w(U) \big) \big/ \sim_w \ ,
\]
but now the equivalence is $(g \cdot i(m)^{-1}, wmw^{-1} \cdot x) \sim_w (g, x)$ for $m \in \Mcheck_\R(U)$.

For both $E$ and $E^w$, the associated $W$-torsors $E_W$ and $E_W^w$ have canonical trivializations $\phi_{\mathrm{can}}: W_\Ga \xrightarrow{\sim} E_W$ and $\phi_{\mathrm{can}}^w: W_\Ga \xrightarrow{\sim} E_W^w$. Explicitly, $\phi_{\mathrm{can}}$ is the unique map such that the composite
\[
E\longrightarrow E/\Mcheck_{\R,\Ga}=E_W\xrightarrow{\phi_{\mathrm{can}}^{-1}} W_\Gamma 
\]
maps the equivalence class $[(g,x)]$ to $\pi(g)$, and similarly for $E^w$.

In the bijection $\wmc M_{G,W_\Gamma}\cong \mc M_{\Mcheck_\R}$, the bundle $L$ corresponds to $(E,\phi_\mathrm{can})$ and $L^w$ corresponds to $(E^w,\phi^w_\mathrm{can})$. The following $G_\Ga$-equivariant map is an isomorphism of $G$-bundles
\[
\psi \colon E \xlongrightarrow{\sim} E^w \quad \text{and}\quad[(g, x)] \mapsto [(g w, x)] \ .
\]
The isomorphism $\psi$ induces an isomorphism of the associated $W$-torsors $\psi_W \colon E_W \xrightarrow{\sim} E^w_W$ such that $\phi_\mathrm{can} \circ w^{-1} = \psi_W^{-1} \circ \phi_\mathrm{can}^w$, i.e., the line bundle $L^w$ corresponds to the pair $(E^w,\phi^w_\mathrm{can}) \cong (E,\phi_\mathrm{can} \circ w^{-1})$. 
\end{proof}

\begin{lemma}
\label{lem:description of torsors as equivariant torsors on cover}
Let $\tau$ be a $W$-torsor and let $f\colon \Gamma'\to \Gamma$ be a free Galois cover, for which there exists an isomorphism $f^*\tau\xrightarrow{\chi} W_{\Gamma'}$. Then there is a natural bijection
\[
    \widetilde{\mc M}_{G,\tau}(\Ga)\cong \mc M_{\Mcheck_\R}(\Gamma')^{\Aut(f)} \ ,
\]
where on the right side we take $\Aut(f)$-invariants under the following action: every automorphism $a:\Gamma'\to \Gamma'$ of the cover $f$ induces an automorphism 
\[
    W_{\Gamma'}\xrightarrow{\chi^{-1}}f^*\tau\cong a^*(f^*\tau)\xrightarrow{a^*\chi}a^*W_{\Gamma'}\cong W_{\Gamma'} \ ,
\]
so that we obtain a morphism $\sigma\colon\Aut(f)\to \Aut(W_{\Gamma'})$. The group $\Aut(W_{\Gamma'})$ in turn acts by conjugation on $\mc M_{\Mcheck_\R}(\Gamma')$. The action on $\Aut(f)$ on $\mc M_{\Mcheck_\R}(\Gamma')$ is given by $(a,E)\mapsto \prescript{\sigma(a)}{}{(a^*E)}$.

Furthermore, there is a map $\delta\colon \Aut(\tau)\to \Aut(W_{\Gamma'})$ induced by the trivialization $\chi$ such that the natural $\Aut(\tau)$-action on $\widetilde{\mc M}_{G,\tau}(\Ga)$ is given by
\begin{equation*}\begin{split}
\mc M_{\Mcheck_\R}(\Gamma')^{\Aut(f)}\times \Aut(\tau)&\longrightarrow \mc M_{\Mcheck_\R}(\Gamma')^{\Aut(f)}\\
(E,t)&\longmapsto {E}^{\delta(t)} \ .
\end{split}\end{equation*}
\end{lemma}

\begin{proof}
Let $\mc B_{G,\tau}(\Gamma)$ be the category of pairs $(E, \phi)$ consisting of a $G$-bundle $E$ on $\Ga$ and an isomorphism $\phi\colon \tau \to E_W$. By definition, $\wmc M_{G,\tau}(\Gamma)$ is the set of isomorphism classes of the objects of $\mc B_{G,\tau}(\Gamma)$.
Because $f$ is a free cover, a function on $\Ga$ is harmonic if and only if its pullback to $\Ga'$ is harmonic. Therefore $f^{*}G_{\Gamma}\cong G_{\Gamma'}$, and hence the pullback $f^{*}E$ of a $G$-bundle $E$ on $\Gamma$ is a $G$-bundle on $\Gamma'$. Moreover, this pull-back is naturally an $\Aut(f)$-equivariant bundle via the canonical morphisms
\[
    f^{*}E \xlongrightarrow{\cong} a^{*}f^{*}E
\]
for $a\in\Aut(f)$. Because $f$ is Galois, the category of $G$-bundles on $\Gamma$ is in fact equivalent to the category of $\Aut(f)$-equivariant $G$-bundles on $\Gamma'$ (one recovers $E$ as $f^{*}E/\Aut(f)$). Arguing similarly for $W$-torsors, we obtain an equivalence of $\mc B_{G,\tau}(\Gamma)$ with the category $\mc B^{\Aut(f)}_{G,f^{*}\tau}(\Gamma')$ of pairs $(E',\phi')$ consisting of an $\Aut(f)$-equivariant $G$-bundle $E'$ on $\Ga'$ and an $\Aut(f)$-equivariant isomorphism $f^{*}\tau \xrightarrow{\phi'} E'_W$, where similarly $f^{*}\tau$ is a $W$-torsor on $\Ga'$. Let $W_{\Gamma'}^\sigma$ denote the trivial $W$-torsor $W_{\Gamma'}$ together with the $\Aut(f)$-equivariant structure induced by $\sigma$. By construction of $\sigma$, the trivialization $\chi$ defines an equivariant isomorphism $f^{*}\tau\xrightarrow{\chi} W_{\Gamma'}^\sigma$, and therefore induces an equivalence of categories
\[
\mc B^{\Aut(f)}_{G,f^{*}\tau}(\Gamma') \longrightarrow \mc B^{\Aut(f)}_{G,W_{\Gamma'}^\sigma}(\Gamma')  \ .
\]
The objects of the target category are objects $(E,\phi)$ of $\mc B_{G,W_{\Gamma'}}(\Gamma')$, together with compatible morphisms 
\[
    (E,\phi)\longrightarrow (a^{*}E, a^{-1}\phi\circ \sigma(a))
\]
for $a\in \Aut(f)$. 

Let $\mc B_{\Mcheck_\R}(\Gamma')$ be the category of $\Mcheck_\R$-bundles on $\Ga'$ (note that the corresponding Weyl group is trivial). By Lemma \ref{lem:meaning of conjugation}, the category $\mc B_{\Mcheck_\R}(\Gamma')$ is equivalent to the category $\mc B_{G,W_{\Gamma'}}(\Gamma')$, in such a way that if $F\in \mc B_{\Mcheck_\R}(\Gamma')$ corresponds to $(E,\phi)$ then $\prescript{\sigma(a)}{}{F}$ corresponds to $(E,\phi\circ \sigma(a))$. Composing all of the equivalences above, we conclude that $\mc B_{G,\tau}(\Gamma)$ is equivalent to the category $\wmc B_{\Mcheck_\R}^{\Aut(f)}(\Gamma')$ of \emph{twisted $\Aut(f)$-equivariant $\Mcheck_\RR$-bundles}, that is to say, $\Aut(f)$-equivariant objects of $\mc B_{\Mcheck_\R}(\Gamma')$ with respect to the $\Aut(f)$-action given by $a.F=\prescript{\sigma(a)}{}{(a^{-1}F)}$. We can also track the action of $\Aut(\tau)$ through this equivalence. The $\Aut(\tau)$-action on $\mc B_{G,\tau}(\Gamma)$ corresponds to the $\Aut(\tau)$-action on $\mc B^{\Aut(f)}_{G, W_{\Gamma'}^\sigma}(\Gamma')$ induced by the morphism $\delta$, which in turn corresponds to the action $F.t= F^{\delta(t)}$ on twisted equivariant $\Aut(f)$-bundles by Lemma \ref{lem:meaning of conjugation}. 

Forgetting the twisted equivariant structure and taking isomorphism classes assigns to every object of $\wmc B^{\Aut(f)}_{\Mcheck_\R}(\Gamma')$ an $\Aut(f)$-invariant element of $\mc M_{\Mcheck_\R}(\Gamma')$. To finish the proof, it suffices to show that every element  $[\widetilde S]\in\mc M_{\Mcheck_\R}(\Gamma')$ that is fixed by $\Aut(f)$ determines a unique isomorphism class of twisted $\Aut(f)$-equivariant $\Mcheck_\R$-bundles on $\Gamma'$. 
 Because the isomorphism class $[\widetilde S]$ of  $\widetilde S$ is $\Aut(f)$-invariant, there are isomorphisms $m^0_a\colon \widetilde S\to \prescript{\sigma(a)}{}(a^{-1}\widetilde S)$ for all $a\in \Aut(f)$. Chosen at random, these will, in general, not define a twisted $\Aut(f)$-equivariant structure. The obstruction for this is the triviality of the automorphisms
\begin{equation}
\label{eq:bar cocycle}
    \psi(a_1,a_2) \colon \widetilde S
    \xrightarrow{m^0_{a_1}} \prescript{\sigma(a_1)}{}(a_1^{-1}\widetilde S)
    \xrightarrow{\prescript{\sigma(a_1)}{}(a_1^{-1}m^0_{a_2})} \prescript{\sigma(a_1a_2)}{}((a_1a_2)^{-1}\widetilde S)
    \xrightarrow{(m^0_{a_1a_2})^{-1}} \widetilde S    \ .
\end{equation}
Note that because $\Mcheck$ is abelian, we have $\Aut(\widetilde S)= H^0(\Gamma', \Mcheck_\R)$, and because $\Gamma'$ is compact this is a finite-dimensional $\R$-vector space. Every other choice of isomorphism $\widetilde S\to \prescript{\sigma(a)}{}(a^{-1}S)$ is of the form $m^0_a\circ \eta(a)$ for some $\eta(a)\in \Aut(\widetilde S)$. Replacing all $m_{a}^0$ by $m_a^0\circ \eta(a)$ in Equation \eqref{eq:bar cocycle}, we see that $(m_a^0\circ \eta(a))_{w\in W}$ defines a twisted $\Aut(f)$-equivariant structure if and only if 
\[
    (d\eta)(a_1,a_2)\coloneqq a_1\eta(a_2)-\eta(a_1 a_2)+\eta(w_1)
\]
is equal to $-\psi(a_1,a_2)$ in $\Aut(\widetilde S)$ for all $(a_1,a_2)$. The notation $d\eta$ is not an accident: $\eta$ defines an inhomogeneous $1$-cochain, which is an element in $C^1(\Aut(f), \Aut(\widetilde S))$, and $d\eta$ is precisely its differential. One also checks that $\psi$ (and hence $-\psi$), is an inhomogeneous $2$-cocycle. Together, this shows that the obstruction for finding a twisted $\Aut(f)$-equivariant structure on $\widetilde S$ is the vanishing of $\psi$ in the second group cohomology $H^2(\Aut(f),\Aut(\widetilde S))$. We already pointed out that $\widetilde S$ is a finite-dimensional $\R$-vector space, so as a consequence of Maschke's theorem we have $H^2(\Aut(f),\Aut(\widetilde S))=0$. We conclude that a twisted $\Aut(f)$-equivariant structure exists. The vanishing of $H^1(\Aut(f),\Aut(\widetilde S))$ tells us that if we are given two twisted $\Aut(f)$-equivariant structures on $\widetilde S$, there exists an $a\in \Aut(f)$ such that $\widetilde S\xrightarrow{l_a} \widetilde S$ is an $\Aut(f)$-equivariant isomorphism (domain and target being equipped with the two given twisted $\Aut(f)$-equivariant structures).
\end{proof}

\begin{example} 
\label{ex:computation of MGLn using cover} Let $\Gamma = \R / l\ZZ$ be a metric circle of length $l$. Let $f: \Gamma' = \R/nl\Z \to \Gamma$ be the connected free Galois cover of degree $n$ that corresponds to $1 \in \Z/n\Z$ under the identification $H^1(\Gamma,\Z/n\Z) \cong \Z / n\Z$.
Let $\tau$ be the $S_n$-torsor on $\Gamma$ that arises as the image of $f$ under the map
\[
H^1(\Gamma,\Z/n\Z) \longrightarrow H^1(\Gamma,S_n) \ .
\]
induced by the morphism $\Z/n\Z\to S_n$ mapping $1$ to $(1 2\cdots n)$.

Note that this determines an isomorphism $\chi: f^* \tau \to (S_n)_{\Gamma'}$ of $S_n$-torsors on $\Ga'$.

In this example, we compute $\wmc M_{\GL_n, \tau}(\Gamma), \wmc M_{\SL_n,\tau}(\Gamma)$ and $\wmc M_{\PGL_n,\tau }(\Gamma)$.

\begin{enumerate}[(a)]
    \item
    
    We show that $\wmc M_{\GL_n, \tau}(\Ga) \cong \Pic(\Gamma')$. Using the standard identification $\Mcheck \cong \Z^n$, by the previous lemma there is a natural bijection
    \[
    \wmc M_{\GL_n, \tau}(\Ga) \cong \calM_{\Mcheck_\R}(\Gamma')^{\Aut(f)} \cong (\Pic(\Gamma') \otimes_\Z \Mcheck)^{\Aut(f)} \cong \left(\bigoplus_{i =1}^n \Pic(\Gamma')\right)^{\Aut(f)}
    \]
    The group $\Aut(f)$ is the cyclic group of order $n$ generated by the automorphism $$g: \Gamma' \to \Gamma' \quad\text{given by}\quad x \to x + l\ . $$ The morphism $\sigma: \Aut(f) \to \Aut((S_n)_{\Gamma'}) \cong S_n$ is given by mapping $g$ to the $n$-cycle $(1 2\cdots n)$. Hence the action of $\Aut(f)$ on $\bigoplus_{i =1}^n \Pic(\Gamma')$ is given as follows: For $(L_1,\dots,L_n) \in \bigoplus_{i =1}^n \Pic(\Gamma')$ we have
    \[
    g \cdot (L_1,\dots,L_n) =  \prescript{\sigma(g)}{}{(g^{*}(L_1,\ldots,L_n))} = \prescript{\sigma(g)}{}{(g^{*}L_1,\ldots,g^{*}L_n)}= (g^{*}L_2,\ldots,g^{*}L_n,g^{*}L_1)
    \]
    Thus, $(L_1,\dots,L_n) \in (\bigoplus_{i =1}^n \Pic(\Gamma'))^{\Aut(f)}$ if and only if $L_i = g^{*}L_{i-1} = (g^{*})^{i+1}L_1$ for $i=2,\ldots,n$. We obtain a bijection
    \[
    \Pic(\Gamma') \longrightarrow \left(\bigoplus_{i =1}^n \Pic(\Gamma')\right)^{\Aut(f)} \quad\text{given by}\quad  L \mapsto (L,g^{*}L,\dots,(g^{*})^{n-1}L) \ .
    \]
    In particular, since $\Aut(\tau) = \Aut(f)$, we obtain 
    \[
    \calM_{\GL_n,\tau} (\Gamma) \cong \Pic(\Gamma')/ \Aut(f).
    \]

    \item  We show that $\wmc M_{\SL_n,\tau}(\Ga)$ consists of the $n$-torsion points of $\Gamma'$. 
    The short exact sequence
    \[
    0 \longrightarrow \Z^n_0 \longrightarrow \Z^n \xlongrightarrow{\det} \Z \longrightarrow 0
    \]
    induces an exact sequence
    \[
    0 \longrightarrow (\Pic(\Gamma') \otimes_\Z \Z^n_0)^{\Aut(f)} \longrightarrow (\Pic(\Gamma') \otimes_\Z \Z^n)^{\Aut(f)} \xlongrightarrow{\det} (\Pic(\Gamma') \otimes_\Z \Z)^{\Aut(f)}
    \]
    We show that the kernel of the induced map $\det$ is isomorphic to the $n$-torsion points of $\Gamma'$.
    Let $(L,g^{*}L,\dots,((g^{*})^{n-1}L))$ be in the kernel of $\det$ for $L \in \Pic(\Gamma')$. Then 
    \begin{equation}
    \label{eq:L}
        L \otimes g^{*}L \otimes \dots \otimes (g^{*})^{n-1}L \cong \calH_{\Gamma'}
    \end{equation}
    which implies that $\deg L = 0$. 
    We observe that for $L \in \Pic^0(\Gamma')$ we have $g^{*}L \cong L$.
    Therefore (\ref{eq:L}) shows that $L^{\otimes n} \cong \calH_{\Gamma'}$. Conversely, if $L \in \Pic^0(\Gamma')$ is an $n$-torsion point, then $(L,g^{*}L,\dots) = (L,L,\dots)$ lies in $\ker(\det)$.

    In summary,
    \[
    \wmc M_{\SL_n,\tau}(\Ga) \cong (\Pic(\Gamma') \otimes_\Z \Z^n_0)^{\Aut(f)} \cong \Pic(\Gamma')[n] \ .
    \]
    Since $\Aut(\tau) = \Aut(f)$ acts trivially, we obtain $\calM_{\SL_n,\tau}(\Ga) = \wmc M_{\SL_n,\tau}(\Ga)$. In particular, we have 
    \[
    |\calM_{\SL_n,\tau}(\Ga)| = |\wmc M_{\SL_n,\tau}(\Ga)|=|\Pic(\Gamma')[n]|=n \ .
    \]
    \item 
    We show that $\wmc M_{\PGL_n,\tau}(\Gamma) \cong \Z / n\Z$. The short exact sequence
    \[
    0 \to \Z \to \Z^n \to \Z^n/\Z(1,\dots,1) \to 0
    \]
    induces a short exact sequence
    \[
    0 \to (\Pic(\Gamma') \otimes_\Z \Z)^{\Aut(f)} \to (\Pic(\Gamma') \otimes_\Z \Z^n)^{\Aut(f)} \to (\Pic(\Gamma') \otimes_\Z \Z^n / \Z(1,\ldots,1))^{\Aut(f)} \to 0 \ .
    \]
    It is right exact because the last map $\wmc M_{\GL_n,\tau}(\Ga) \to \wmc M_{\PGL_n,\tau}(\Ga)$ is surjective. This follows from Lemma \ref{lem:surjective morphism of groups yields surjective morphism of moduli} since the underlying morphism of $\GL_n(\T) \to \PGL_n(\T)$ on the lattices is surjective.
    Hence, $\wmc M_{\PGL_n,\tau}(\Ga)$ is the cokernel of the map $\Pic(\Gamma')^{\Aut(f)} \to \Pic(\Gamma')$. Since $\Pic(\Gamma')^{\Aut(f)} = \Pic^{n\Z}(\Gamma')$, the cokernel is identified with $\Z / n\Z$, and likewise $\calM_{\PGL_n,\tau}(\Ga) = \wmc M_{\PGL_n,\tau}(\Ga)$.
\end{enumerate}
\end{example}

We are now ready to prove our main result.

\begin{proof}[Proof of Theorem \ref{thm:description of component of tropical moduli space}]
\label{proof of thm about components}
Let $f\colon\Gamma'\to \Gamma$ be the éspace étalé of $\tau$, which is a Galois cover because $\tau$ is a torsor over a discrete group. We equip $\Gamma'$ with the induced sheaf of harmonic functions to make $f$ a free cover. By Lemma \ref{lem:pull-back of tau by tau},  the pull-back $f^{*}\tau$ has a canonical trivialization. We can thus apply Lemma \ref{lem:description of torsors as equivariant torsors on cover} and obtain a bijection
\[
    \wmc M_{G,\tau}\cong (\Pic(\Gamma')\otimes_\Z \Mcheck)^{\Aut(f)} \ .
\]
Since $\Aut(f)$ acts by pulling back and conjugation, the group action of $\Aut(f)$ on the components is by translates of morphisms of tropical abelian varieties. Therefore, $(\Pic(\Ga')\otimes_\Z \Mcheck)^{\Aut(f)}$ is a union of torsors over tropical abelian varieties as well.

The statement for $\mc M_{G,\tau}$ follows immediately, as it is the $\Aut(\tau)$-quotient of $\wmc M_{G,\tau}$.
\end{proof}

 We note that the proof shows that each $\wmc M_{G,\tau}$ is in fact a group, not simply a torsor.

\subsection{Stable bundles on tropical elliptic curves}
\label{subsec:stable bundles on tropical elliptic curves}

In this section, we prove a tropical analogue of Theorem~\ref{thm:Fratilla determinant is isomorphism}, which classifies stable $G$-bundles on an elliptic curve. We restrict our attention to tropical reductive groups of type $\prod A_{n_i-1}$, since in the algebraic setting there are no stable $G$-bundles of other types. Our main result is Theorem~\ref{thm:determinant is isomorphism for tropical moduli of stable bundles}, which classifies stable tropical $G$-bundles of type $\prod A_{n_i-1}$ on a metric circle. The analogous tropical statement is not true on the nose, but one has to restrict to tropical bundles whose degree is stable and whose underlying $W$-torsor is indecomposable in the sense defined below.

\begin{definition} Let $W=\prod_i S_{n_i}$ be a Weyl group of type $\prod_{i}A_{n_i-1}$. An element of $W$ is called \emph{indecomposable} if it is a product of $n_i$-cycles. A $W$-torsor on a metric circle is called \emph{indecomposable} if it is defined by an indecomposable element of $W$.
\end{definition}

We note that an $S_n$-torsor on a metric circle $\Ga$ is indecomposable if and only if the associated degree $n$ cover is connected. Furthermore, all indecomposable elements of $W=\prod_i S_{n_i}$ are conjugate to each other, hence there is a unique (up to isomorphism) indecomposable $W$-torsor on $\Gamma$, which we denote by $\ind\in H^1(\Ga,W)$.

Let $\Phi=(M,R,\Mcheck,\Rcheck)$ be a root datum. Its associated \emph{adjoint root datum} is given by $\Phi^\ad= (\langle R\rangle,R, \Mcheck^\ad, \Rcheck)$, where 
\[
\Mcheck^\ad=\{\check m \in \Mcheck_\Q : \langle r, \check m\rangle\in \Z \text{ for all }r\in R \} \ .
\]
The tropical reductive group $G^\ad=G_{\Phi^\ad}$ associated to the adjoint root datum of a tropical reductive group $G=G_\Phi$ is the \emph{adjoint group} of $G$, and the inclusion $\Mcheck \to \Mcheck^\ad$ induces a canonical morphism $G\to G^\ad$.

\begin{lemma}
\label{lem:stabilizers in Aut(tau) in An case}
Let $G=\Mcheck_\R\rtimes S_n$ be a tropical reductive group of type $A_{n+1}$ and suppose that $\langle \Rcheck\rangle$ is saturated in $\Mcheck$. Let $\lambdacheck_G\in \pi_1(G)$ be a stable degree (see Definition~\ref{def:stable_degree}) and let $E\in \wmc M_{G,\ind}^{\lambdacheck_G}$. Then the stabilizer group $\Aut(\ind)_E=1$ is trivial.
\end{lemma}

\begin{proof}
Let $Z_\R=Z(G)$ be the center, which is the vector space associated to the lattice $Z=Z(\Mcheck\rtimes W)$. Then $G/Z_\R$ is simple of type $A_{n-1}$ and the natural map $G/Z_\R\to G^\ad$ induces an inclusion $\pi_1(G/Z_\R)\to \pi_1(G^\ad)$. The degree $\lambdacheck_G$ can only be stable if its image generates $\pi_1(G^\ad)$, which implies that $G/Z_\R=G^\ad=\PGL_n$. We now reduce to the case where $Z$ has rank $1$. Note that because $\langle \Rcheck\rangle$ is saturated, we cannot have $Z=0$, because that would imply $\Mcheck=\langle \Rcheck\rangle$ and hence $G=\SL_n$, which is a contradiction to $G/Z_\R=\PGL_n$. Consider the morphism $\phi\colon Z\to \Mcheck/\langle\Rcheck\rangle$. It is injective and its cokernel equals $\pi_1(G)=\Z/n\Z$. Because $\phi$ has cyclic cokernel, all but one of the invariant factors in its Smith normal form are $1$, so there is a rank $\rk Z-1$ sublattice $K$ of $Z$ such that $\phi(K)$ is saturated. In particular, the image of $\langle \Rcheck\rangle$ is saturated in the quotient $H=G/K_\R$. If $f\colon G\to H$ denotes the quotient map, then $f_*E$ has degree $\lambdacheck_H=\lambdacheck_G\in \pi_1(H)=\pi_1(G)$. As 
\[
    f_*\colon \wmc M_{G,\ind}^{\lambdacheck_G}\longrightarrow \wmc M_{H,\ind}^{\lambdacheck_H}
\]
is $\Aut(\ind)$-equivariant, it suffices to prove the statement for $f_*E$ and $H$, which satisfy all the hypotheses of the assertion, and $H$ has center equal to $Z/K$, which has rank $1$. 

Now assume $Z$ has rank $1$.  We compute $\wmc M_{G,\ind}$ similarly as in Example \ref{ex:computation of MGLn using cover}. Let $f\colon \Gamma'\to \Gamma$ be a connected cyclic degree $n$ cover with $W$-torsor $\ind$, then by Lemma \ref{lem:description of torsors as equivariant torsors on cover} we have
\[
    \wmc M_{G,\ind}=(\Pic(\Gamma')\otimes_\Z \Mcheck)^{\Aut(f)} \ .
\]
As $\langle \Rcheck\rangle$ is saturated, the degree map factors through the determinant map, which is given by the morphism
\[
    \left(\Pic(\Gamma')\otimes_\Z \Mcheck\right)^{\Aut(f)}
        \to \left(\Pic(\Gamma')\otimes_\Z \Mcheck/\langle \Rcheck\rangle\right)^{\Aut(f)}
            = \Pic(\Gamma)\otimes_\Z \Mcheck/\langle\Rcheck\rangle 
\]
The element $L\in\left(\Pic(\Gamma')\otimes_\Z \Mcheck\right)^{\Aut(f)}$ corresponding to $E$ satisfies $\prescript{\sigma(a)}{}{(a^*L)}=L$ for all $a\in \Aut(f)$, where $\sigma\colon \Aut(f)\to \Aut(W_{\Gamma'})$ as in Lemma \ref{lem:description of torsors as equivariant torsors on cover}. On the other hand, $t\in\Aut(\ind)$ acts by $L.t=L^{\delta(t)}$, where $\delta\colon \Aut(\ind)\to \Aut(W_{\Gamma'})$ is as in Lemma \ref{lem:description of torsors as equivariant torsors on cover}. But in our case, $\Aut(\ind)$ and $\Aut(f)$ are identified because $\Gamma'$ is the total space of $\ind$, hence $\sigma=\delta$. We conclude that for $a\in \Aut(\ind)=\Aut(f)$ we have
\[
    L^{\sigma(a)}= a^*L
\]
As pulling back along $a$ leaves degrees invariant, it follows that the degree of $L$ is invariant, so it is contained in $\Mcheck^{\Aut(\ind)}=Z$, where for the equality we use that $\langle \Rcheck\rangle^{\Aut(\ind)}=0$. So we have
\[
    L\in \Pic^0(\Gamma')\otimes_\Z \Mcheck+ \Pic(\Gamma')\otimes_\Z Z \ .
\]
Let $z\in Z$ and $m\in \Mcheck/\langle \Rcheck\rangle$ be generators. As $Z+\langle\Rcheck\rangle$ has index $n$ in $\Mcheck$, the map $\Z\cong Z\to \Mcheck/\langle \Rcheck\rangle\cong \Z$ maps $z$ to $\pm n m$. So if $\deg(L)=k\cdot z$, then $\deg(\det(E))=k\cdot m$. As the degree factors through the determinant, $\lambdacheck_G$ can only be stable if $\gcd(k,n)=1$. 

It now suffices to prove that for $1\neq a\in \Aut(f)$, we have $a^*L\otimes L^{-1}\neq 0$. Since $a^*\widetilde{L}\otimes \widetilde{L}^{-1} \cong 0$ for all $\widetilde{L}\in \Pic^0(\Gamma')\otimes \Mcheck$, we see that $a^*L\otimes L^{-1}\neq 0$ only depends on $\deg(L)$ and we may replace $L$ by any element of $\Pic(\Gamma')\otimes_\Z Z$. With the chosen isomorphism 
\[
    \Z\to Z,\;\; 1\mapsto z \ ,
\]
we have $\Pic(\Gamma')\otimes_\Z Z\cong \Pic(\Gamma')$ and the set of elements in $\Pic(\Gamma')\otimes_\Z Z$ of degree $k\cdot z$ corresponds to $\Pic^k(\Gamma')$. But because $\gcd(k,n)=1$, the $\Aut(f)$-action on $\Pic^k(\Gamma')$ is free, concluding the proof.
\end{proof}

Recall that by Lemma~\ref{lem:subgroup is normal} there is a natural map $\Mcheck\rtimes W\to \Mcheck/\langle \Rcheck\rangle=\pi_1(G)$. Tensoring with $\RR$, we obtain the \emph{determinant map}
\[
\det:G\longrightarrow \Mcheck_{\RR}/\langle\Rcheck\rangle_{\RR}=\pi_1(G)_{\RR} \ .
\]
By abuse of notation, we also use $\det$ to denote the corresponding map on degrees:
\[
\det:\pi_1(G)\longrightarrow \pi_1(\Mcheck_{\RR}/\langle\Rcheck\rangle_{\RR})=\Mcheck/\langle\Rcheck\rangle^{\sat}=\pi_1(G)^{\mathrm{tf}}.
\]
\begin{lemma}
\label{lem:determinant is isomorphism if roots are saturated}
Let $G=\Mcheck_\R \rtimes W$ be a tropical reductive group of type $\prod_{i=1}^k A_{n_i-1}$ such that $\langle\Rcheck\rangle$ is saturated in $\Mcheck$, and let $\lambdacheck_G\in \pi_1(G)$ be a stable degree. Then the determinant induces a homeomorphism
\[
    \det\colon \mc M^{\lambdacheck_G}_{G,\ind}\longrightarrow \mc M^{\det(\lambdacheck_G)}_{T} \ ,
\]
where $T=\Mcheck_{\RR}/\langle \Rcheck\rangle_{\RR}$.
\end{lemma}

\begin{proof}
Let $\ind$ be an indecomposable $W$-torsor and let $f\colon \Gamma'\to \Gamma$ be a cyclic degree $\mathrm{lcm}(n_1,\ldots, n_k)$ cover. Consider the short exact sequence
\[
    0 \longrightarrow 
        \Pic(\Gamma')\otimes_\Z \langle \Rcheck\rangle \longrightarrow
            \Pic(\Gamma')\otimes_\Z\Mcheck  
                \longrightarrow \Pic(\Gamma')\otimes \Mcheck/\langle\Rcheck\rangle \longrightarrow
                    0\ .
\]
Proposition~\ref{prop:fiber over trivial W-torsor} identifies each term with a moduli space of bundles on $\Ga'$ with trivial $W$-torsor $W_{\Ga'}=f^*(\ind)$. Taking $\Aut(f)$-invariants and applying Lemma \ref{lem:description of torsors as equivariant torsors on cover}, we obtain a left exact sequence of abelian groups
\[
0 \longrightarrow 
    \wmc M_{G^{\simc},\ind}(\Ga) \longrightarrow 
        \wmc M_{G,\ind}(\Ga) \xlongrightarrow{\det} 
            \mc M_{T}(\Ga) \longrightarrow
                0 \ ,
\]
where we denote $G^{\simc}=\prod \SL_{n_i}$. By Lemma~\ref{lem:surjective morphism of groups yields surjective morphism of moduli} applied to the surjective map $G\mapsto T\times W$, the second map is surjective and therefore the sequence is also exact on the right. Because the degree factors through the determinant, it follows that the determinant map induces a bijection
\[
    \wmc M^{\lambdacheck_G}_{G,\ind}(\Ga)/\wmc M_{G^{\simc},\ind}(\Ga) \longrightarrow 
        \mc M^{\det(\lambdacheck_G)}_{T}(\Ga) \ ,
\]
where we do not write the degree in the quotient because $\pi_1(G^{\simc})$ is trivial. The determinant map is also invariant under action of $\Aut(\ind)$ and we need to show that we also have
\[
    \mc M^{\lambdacheck_G}_{G,\ind}(\Ga)=\wmc M^{\lambdacheck_G}_{G,\ind}(\Ga)/\Aut(\ind)\cong 
        \mc M^{\det(\lambdacheck_G)}_{T}(\Ga) \ .
\]
As we have seen in part (2) of Example \ref{ex:computation of MGLn using cover}, we have 
\[
    \big|\wmc M_{G^{\simc},\ind}(\Gamma)\big|=\prod \big|\wmc M_{\SL_{n_i},\ind}(\Gamma)\big| =\prod n_i\ ,
\]
which coincides with $\big|\Aut(\ind)\big|$. Therefore, it suffices to show that $\Aut(\ind)$ acts freely on $\wmc M^{\lambdacheck_G}_{G,\ind}(\Ga)$. To show this, consider for $1\leq i\leq k$ the quotient morphism
\[
    q^i\colon G\to G^i\coloneqq \Mcheck_{i,\R} \rtimes S_{n_i}\quad\text{given by}\quad \Mcheck_i=\Mcheck /\sum_{j\neq i} \langle \Rcheck_j\rangle 
    \ ,
\]
where $\Rcheck_j$ is the set of coroots of the factor $\SL_{n_i}$ of $G^{\simc}$. The induced push-forward
\[
q^i_*\colon \wmc M^{\lambdacheck_G}_{G,\ind}(\Ga)\longrightarrow \wmc M^{q^i_*\lambdacheck_G}_{G^i,q^i_*\ind}(\Ga) 
\]
respects the $\Aut(\ind)$-action. In particular, for 
\[
    \mathbf t=(t_1,\ldots, t_k)\in \prod \Aut(q^i_*\ind)=\Aut(\ind) \ ,
\]
and $E\in \wmc M^{\lambdacheck_G}_{G,\ind}(\Ga)$, we have
\[
    q^i_*(E.\mathbf t)= (q^i_*E).t_i \ .
\]
Because $q^i_*\lambdacheck_G$ is stable, $\Aut(q^i_*\ind)$ acts freely on $\wmc M^{q^i_*\lambdacheck_G}_{G^i,q^i_*\ind}(\Ga)$ by Lemma \ref{lem:stabilizers in Aut(tau) in An case} and it follows that $E.\mathbf t\cong E$ if and only if $\mathbf t=(1,\ldots,1)$.
\end{proof}

We are now ready to prove the tropical counterpart to Theorem \ref{thm:Fratilla determinant is isomorphism}.

\begin{theorem}
\label{thm:determinant is isomorphism for tropical moduli of stable bundles}
Let $G= \Mcheck_\R \rtimes W$ be a tropical reductive group of type $\prod_{i=1}^k A_{n_i-1}$ and let $\lambdacheck_G\in \pi_1(G)$ be a stable degree. Then the determinant induces a homeomorphism
\[
    \det\colon \mc M^{\lambdacheck_G}_{G,\ind}\longrightarrow \mc M^{\det(\lambdacheck_G)}_{T} \ ,
\]
where $T=\Mcheck_{\RR}/\langle \Rcheck\rangle_{\RR}$. 
\end{theorem}

We first prove the following algebraic fact.

\begin{lemma}
    \label{lem:existence of cover used for degree calculation}
    Let $\mathbf{G}$ be a reductive linear algebraic group. Then there is a morphism $p\colon \mathbf{G}'\to \mathbf{G}$ of reductive groups with the properties that
    \begin{enumerate}
        \item $\ker(p)\subseteq Z(\mathbf{G}')$,
        \item $p$ is surjective,
        \item $\ker(p)$ is connected,
        \item $\pi_1(\mathbf{G}')$ is torsion-free.
    \end{enumerate}
\end{lemma}

\begin{proof}
    Choose a maximal torus $\mathbf{T}\subseteq \mathbf{G}$ and let $(M,R,\Mcheck, \Rcheck)$ denote the root datum corresponding to $(G,T)$.  Choose a family $(n_i)_{1\leq i\leq k}$ generating $\Mcheck$, let $\Lambdacheck=\langle \Rcheck\rangle \oplus \Z^k$, let $\Lambda=\Lambdacheck^\vee$, and consider the morphism
    \[
    \pi\colon \Lambdacheck \longrightarrow \Mcheck
    \]
    that is the inclusion on $\langle \Rcheck\rangle$ and maps the $i$th generator of $\Z^k$ to $n_i$. Let $\Phicheck=\Rcheck\subseteq \Lambdacheck$ and let $\Phi=\pi^*(R)$. Then $(\Lambda,\Phi,\Lambdacheck,\Phicheck)$ is a root datum and $\pi$ defines a morphism of root data. By the existence theorem for reductive groups \cite[Theorem 10.1.1]{Springer_linalggroups} and the isogeny theorem \cite[Theorem 9.6.5]{Springer_linalggroups}, there is a reductive group $\mathbf{G}'$  with maximal torus $\mathbf{T}'$ corresponding to $(\Lambda,\Phi,\Lambdacheck,\Phicheck)$ and a central isogeny $p\colon \mathbf{G}'\to \mathbf{G}$ mapping $\mathbf{T}'$ to $\mathbf{T}$ and inducing $\pi$ on the level of root data. As $\pi$ is surjective, $p$ is surjective and $\ker(p)$ is connected. Because $\langle\Phicheck\rangle$ is saturated in $\Lambdacheck$, $\pi_1(\mathbf{G}')$ is torsion-free. We observe that $\mathbf{G}'$ has the same Weyl group as $\mathbf{G}$.
\end{proof}

\begin{proof}[Proof of Theorem \ref{thm:determinant is isomorphism for tropical moduli of stable bundles}]
By Lemma \ref{lem:existence of cover used for degree calculation}, there exists a morphism
\[
    \phi= (\pi, \id) \colon G'= \Lambdacheck_\R \rtimes W  \longrightarrow \Mcheck_\R \rtimes W
\]
of tropical reductive groups (having the same Weyl group) such that the map $\pi: \Lambdacheck\to\Mcheck$ is surjective and $\langle \Phicheck\rangle$ is saturated in $\Lambdacheck$, where $\Phicheck$ denotes the set of coroots in $\Lambdacheck$. Denote by $K$ the kernel of the morphism  $\Lambdacheck\to \Mcheck$; we have $K\subseteq Z(\Lambdacheck_\R \rtimes W)$. Let $f\colon \Gamma'\to \Gamma$ be the cyclic and connected degree $\mathrm{lcm}(n_1,\ldots,n_k)$ cover of $\Gamma$ and consider the short exact sequence
\[
    0 \longrightarrow
        \Pic(\Gamma')\otimes_\Z K \longrightarrow  
            \Pic(\Gamma')\otimes_\Z \Lambdacheck \longrightarrow
                \Pic(\Gamma')\otimes_\Z \Mcheck \longrightarrow
                    0 \ .
\]
Because the pullback along $f$ of an indecomposable $W$-torsor trivializes, Lemma \ref{lem:description of torsors as equivariant torsors on cover} implies that taking $\Aut(f)$-invariants in the sequence above yields a left exact sequence
\[
    0 \longrightarrow
        \mc M_{K_\R} \longrightarrow
            \wmc M_{G',\ind} \longrightarrow
                \wmc M_{G,\ind} \longrightarrow
                    0 \ .
\]
The sequence is also right exact by Lemma \ref{lem:surjective morphism of groups yields surjective morphism of moduli} applied to the surjective morphism $\phi: G' \to G$, which is also surjective on the cocharacter lattices. 
We also have a commutative diagram
\[
\begin{tikzcd}
    0 \arrow[r]&
        K \arrow[r]\arrow[d,"="] & 
            \Lambdacheck/\langle\Phicheck\rangle \arrow[r]\arrow[d,"="]&
                \Mcheck/\langle\Rcheck\rangle \arrow[r]\arrow[d,"="] &
                    0 \\
    0 \arrow[r] &
        \pi_1(K_\R) \arrow[r] & 
            \pi_1(G') \arrow[r,"\phi_*"] &
                \pi_1(G) \arrow[r] &
                    0 \ .
\end{tikzcd}
\]
The top row, and hence the bottom one as well, is exact because $\phi(\Phicheck)=\Rcheck$ and $\langle\Phicheck\rangle \cap Z(G')=0$.  It follows that for every $\lambdacheck_{G'}\in \pi_1(G')$ with $\phi_*\lambdacheck_{G'}=\lambdacheck_G$, the map $\phi$ induces a bijection
\[
    \wmc M^{\lambdacheck_{G'}}_{G',\ind}/\mc M^0_{K_\R} \xlongrightarrow{\cong} \wmc M^{\lambdacheck_G}_{G,\ind} \ .
\]
Now consider the commutative diagram
\[
\begin{tikzcd}
    \mc M^{\lambdacheck_{G'}}_{G',\ind} \arrow[r,"\phi_*"] \arrow[d,"\det"]&
        \mc M^{\lambdacheck_G}_{G,\ind} \arrow[d,"\det"]\\
    \mc M^{\det(\lambdacheck_{G'})}_{\Lambdacheck_\RR/\langle\Phicheck\rangle_\RR} \arrow[r,"\phi_*"] &
        \mc M^{\det(\lambdacheck_G)}_{\Mcheck_\RR/\langle\Rcheck\rangle_\RR} \ .
\end{tikzcd}
\]
By Lemma \ref{lem:determinant is isomorphism if roots are saturated}, the determinant map on the left is a homeomorphism. The morphism 
\[
\Lambdacheck/\langle\Phicheck\rangle\longrightarrow \Mcheck/\langle\Rcheck\rangle^\sat
\]
has finite index, and hence the lower morphism of the square is surjective. It follows that the determinant map on the right is surjective. It remains to show that the determinant on the right is injective. Let $E,E'\in \mc M^{\lambdacheck_G}_{G,\ind}$ with $\det(E)\cong\det(E')$. We can lift both $E$ and $E'$ first to $\wmc M^{\lambdacheck_G}_{G,\ind}$ and then to elements $F,F'\in \wmc M^{\lambdacheck_{G'}}_{G',\ind}$ with $\phi_*\det(F)\cong \phi_*\det(F')$. The map
\[
    \mc M_{K_\R}^0 \longrightarrow 
        \mc M^0_{\Lambdacheck_\R/(K+\langle\Phicheck\rangle)^\sat_\R} 
\]
is surjective because $K\to (K+\langle\Phicheck\rangle)^\sat/\langle\Phicheck\rangle$ has finite index, and the sequence 
\[
    0 \longrightarrow
        \mc M^0_{(K+\langle\Phicheck\rangle)^\sat_\R/\langle\Phicheck \rangle_\R} \longrightarrow  
            \mc M^0_{\Lambdacheck_\R/\langle\Phicheck\rangle_\R} \longrightarrow
                \mc M^0_{\Mcheck_\R/\langle\Rcheck\rangle^\sat_\R}\longrightarrow   
                    0
\]
is exact. Therefore, there exists $L\in \mc M^0_{K_\R}$ with $\det(F')=\det(L \otimes F)$. As $\phi_*(L\otimes F)=\phi_*F$ we may replace $F$ by $L\otimes F$ and assume that $\det(F')=\det(F)$.
By Lemma \ref{lem:determinant is isomorphism if roots are saturated}, there exists an automorphism $t$ of the indecomposable $W$-torsor such that $F'=F.t$. It follows that 
\[
    \phi_*(F')=\phi_*(F.t) =\phi_*(F).t
\]
and hence $E\cong E'$, concluding the proof.
\end{proof}

\subsection{Semistable bundles on tropical elliptic curves}

We now give a tropical analogue of Theorem~\ref{thm:Fratilas results} that describes semistable bundles on an elliptic curve.

\begin{lemma}
\label{lem:centralizer contained in normalizer}
Let $W'$ be a Weyl group of a root datum and let $W$ be a parabolic subgroup of type $\prod A_{n_i-1}$. Let $w\in W$ be indecomposable. Then we have
\[
    C_{W'}(w)\subseteq N_{W'}(W) \ .
\]
\end{lemma}

\begin{proof}
Let $g\in C_{W'}(w)$, then we have $w\in W\cap g^{-1}Wg$. By \cite[Lemma 2.25]{Buildings}, the intersection $W\cap g^{-1}Wg$ is a parabolic subgroup of $W$.  Now observe that no proper standard parabolic subgroup of $\prod_{i=1}^k S_{n_i}$ contains an indecomposable element. As all indecomposable elements are conjugate, it follows that no proper parabolic subgroup of $\prod_{i=}^k S_{n_i}$, and hence of $W$, contains an indecomposable element. We conclude that $W\cap g^{-1}Wg=W$, that is $g^{-1}Wg\subseteq W$, which we needed to show.
\end{proof}

For the next lemma, we note that an element $w\in W$ of a Weyl group $W$ of type $\prod A_{n_i-1}$ is indecomposable if and only if it has maximal \emph{reflection length}, which is the minimum number of reflections (not necessary simple) in a representation of $w$. 

\begin{lemma}
\label{lem:quotient of centralizers is relative Weyl group}
Let $W$ be a parabolic subgroup of type $\prod A_{n_i-1}$ of a Weyl group $W'$ of a root datum, and let $w\in W$ be indecomposable. Then the natural homomorphism
\[
    C_{W'}(w)/C_W(w)\longrightarrow N_{W'}(W)/W \ ,
\]
which exists by Lemma \ref{lem:centralizer contained in normalizer}, is an isomorphism.
\end{lemma}

\begin{proof}
Injectivity is clear from the fact that 
\[
    C_W(w)=C_{W'}(w)\cap W \ .
\]
For surjectivity, let $g\in N_{W'}(W)$. The reflections in $W$ are the reflections of $W'$ that are contained in $W$ (by \cite[Lemma 2.25]{Buildings} applied with $|K|=1$), so conjugation by $g\in N_{W'}(W)$ preserves reflection length of elements of $W$. As the indecomposable elements of $W$ are precisely those of maximal reflection length, $g^{-1}wg$ is indecomposable as well. But all indecomposable elements of $W$ are conjugate in $W$, that is $g^{-1}wg=h^{-1}wh$ for some $h\in W$. It follows that $gh^{-1}\in C_{W'}(w)$ and thus that $gW=(gh^{-1})W$ is in the image of $C_{W'}(W)\to N_{W'}(W)/W$. 
\end{proof}

Given a parabolic subgroup $W$ of type $\prod A_{n_i-1}$ of a Weyl group $W'$ and a metric circle $\Gamma$, we denote by $\ind_{W'}$ the $W'$-torsor induced by the indecomposable $W$-torsor $\ind$ on $\Gamma$ via the inclusion $W\to W'$.

\begin{corollary} Let $\Gamma$ be a metric circle. Let $W$ be a parabolic subgroup of type $\prod A_{n_i-1}$ of a Weyl group $W'$ of a root datum $\Phi$. Moreover, let $G'=\Mcheck_\R \rtimes W'$ be the tropical reductive group corresponding to $\Phi$ and let $G=\Mcheck_\R \rtimes W$.  Then the action of $N_{W'}(W)/W$ on $\mc M_{G'}(\Gamma)$ induces an action on $\mc M_{G',\ind_{W'}}(\Ga)$.
\end{corollary}

\begin{proof}
We have computed in Example \ref{ex:explicit description of torsors on circle} that isomorphism classes of $W$-torsors on $\Gamma$ are in bijection with conjugacy classes of elements of $W$. The (unique) isomorphism class of indecomposable covers corresponds to the conjugacy class of indecomposable elements of $W$. By Lemma \ref{lem:quotient of centralizers is relative Weyl group}, this conjugacy class is fixed by conjugation by elements in $N_{W'}(W)/W$,  which implies that the action of $N_{W'}(W)/W$ leaves the isomorphism class of the associated $W$-torsor of an element in $\mc M_{G',\ind_{W'}}(\Ga)$ invariant. 
\end{proof}

\begin{theorem}
\label{thm:tropical semistable bundles are quotients of stable ones}
Let $\Gamma$ be a metric circle. Let $W$ be a parabolic subgroup of type $\prod_{i=1}^k A_{n_i-1}$ of a Weyl group $W'$ of a root datum $\Phi$. Moreover, let $G'= \Mcheck_\R \rtimes W'$ be the tropical reductive group associated to $\Phi$ and let $G= \Mcheck_\R \rtimes W$. Then the natural map
\[
    \mc M_{G',\ind_{W'}}(\Ga)/(N_{W'}(W)/W)\longrightarrow  \mc M_{G,\ind}(\Ga)
\]
is a bijection.
\end{theorem}

\begin{proof}
Let $f\colon \Gamma'\to \Gamma$ be a cyclic cover of degree $\mathrm{lcm}(n_1,\ldots,n_k)$. Let $\ind$ be an indecomposable $W$-torsor. Then $f^*\ind$ is trivial, and so is $f^*(\ind_{W'})$. We can thus apply Lemma \ref{lem:description of torsors as equivariant torsors on cover} and conclude that
\begin{equation}\begin{split}
\label{eq:two quotients by automorphism groups}
    \mc M_{G,\ind}(\Gamma)&= (\Pic(\Gamma')\otimes_\Z \Mcheck)^{\Aut(f)}/\Aut(\ind) \ ,\\
    \mc M_{G',\ind_{W'}}(\Gamma)&= (\Pic(\Gamma')\otimes_\Z \Mcheck)^{\Aut(f)}/\Aut(\ind_{W'}) \ .
\end{split}\end{equation}
In this description, the automorphism group $\Aut(\ind)$ (resp.\ $\Aut(\ind_{W'})$) acts by conjugation by elements in the image of a morphism 
\[
    \delta\colon\Aut(\ind)\to \Aut(f^*\ind)\cong \Aut(W_{\Gamma'})=W \ ,
\]
where the last morphism is induced by the trivialization of $f^*\ind$ (and similarly, there is a morphism $\delta'\colon \Aut(\ind_{W'})\to W'$). We have seen in Example \ref{ex:explicit description of torsors on circle} that there is a trivialization of $\ind$ on an open subset on which $\Aut(\ind)$ can be identified with the right action of the centralizer $C_W(\widetilde w)$ for some $\widetilde w\in W$ whose conjugacy class determines the isomorphism class of $\ind$. Because $\ind$ is indecomposable, the element $w$ is indecomposable as well. The morphism $\delta$ is not necessarily defined using the same trivialization, so the image of $\delta$ equals $C_W(w)$ for some conjugate $w$ of $\widetilde w$, which is again indecomposable. Since we can choose the trivialization of $f^*(\ind_{W'})=(f^*\ind)_{W'}$ to be induced by the trivialization of $f^*\ind$, we may assume that the image of $\delta'$ is given by $C_{W'}(w)$. Together with \eqref{eq:two quotients by automorphism groups}, we see that 
\begin{align*}
    \mc M_{G,\ind}(\Gamma)&= (\Pic(\Gamma')\otimes_\Z \Mcheck)^{\Aut(f)}/C_{W}(w) \ ,\\
    \mc M_{G',\ind_{W'}}(\Gamma)&= (\Pic(\Gamma')\otimes_\Z \Mcheck)^{\Aut(f)}/C_{W'}(w) \ ,
\end{align*}
where $C_W(w)$ and $C_{W'}(w)$ act by conjugation. By Lemma \ref{lem:quotient of centralizers is relative Weyl group}, $C_W(w)$ is normal in $C_{W'}(w)$ so that there exists an induced $C_{W'}(w)/C_W(w)$-conjugation action on the quotient $\mc M_{G,\ind}(\Gamma)$, and 
\[
    \mc M_{G',\ind_{W'}}(\Gamma)=\mc M_{G,\ind}(\Gamma)/(C_{W'}(w)/C_W(w)) \ .
\]
Also by Lemma \ref{lem:quotient of centralizers is relative Weyl group}, we have $C_{W'}(w)/C_W(w)=N_{W'}(W)/W$, concluding the proof.
\end{proof}

\section{Tropicalization of principal bundles}\label{sec_tropicalization}

Let $K$ be an algebraically closed field that is complete with respect to a nontrivial non-Archimedean absolute value $|\cdot|$ of equicharacteristic $0$, let $\mathbf{G}$ be a reductive group over $K$ and let $X$ be a Mumford curve over $K$. In this section, we consider the process of tropicalization for $\mathbf{G}$-bundles on $X$ together with a reduction of structure group to the normalizer of a given maximal torus $\mathbf{T}$ in $\mathbf{G}$. When $X$ is a Tate curve, we show that every semistable $\mathbf{G}$-bundle on $X$ is equivalent to one that admits such a reduction. This allows us to establish our main result, Theorem \ref{thm:tropical stable bundles skeleton of stable bundles}, which identifies the essential skeleton of the moduli space of semistable principal $\mathbf{G}$-bundles on $X$ with a moduli space of tropical semistable principal $\mathbf{G}^{\trop}$-bundles on the minimal skeleton $\Gamma_X$ of $X$.

Let $\mathbf{T} \subseteq \mathbf{G}$ be a maximal torus with character lattice $M = \mathbb{X}^*(\mathbf{T})$ and cocharacter lattice $\Mcheck = \mathbb{X}_*(\mathbf{T})$. We denote the associated root datum by $\Phi = (M,R,\Mcheck,\Rcheck)$. Let $N_\mathbf{G}(\mathbf{T})$ be the normalizer of $\bfT$ in $\bfG$ and let $W=N_\bfG(\bfT)/\bfT$ be the Weyl group. The Weyl group acts on $\bfT$ by conjugation and therefore there is an induced action of $W$ on $\Mcheck$. We define the \emph{tropical reductive group} $\bfG^\trop$ associated to $\bfG$ as 
\[
\bfG^\trop= \Mcheck_\R \rtimes W \ .
\]
Note that the isomorphism type of $\bfG^\trop$ does not depend on the choice of $\bfT$ as all maximal tori in $\bfG$ are conjugate. 

\subsection{Tropicalizing $\bfT$-bundles over a Mumford curve}

Let $\Gamma_X$ be the minimal skeleton of the Berkovich analytic space $X^{\an}$. Given a $\bfT$-bundle $E$ on $X$, we obtain for every character $(\bfT\xrightarrow{m}\G_m)\in M$ an induced $\G_m$-bundle $m_*(E)$ on $X$, which we can tropicalize to a tropical line bundle $\Trop(m_*(E))$ on $\Gamma_X$. Since tropicalization respects tensor products of line bundles, we obtain, for every non-Archimedean field extension $L/K$, a bilinear map
\[
\mc M_\bfT(X)(L)\times M\longrightarrow \mc M_{\R}(\Ga_X)  \quad\text{given by}\quad (E,m) \longmapsto \Trop(m_*(E)) \ .
\]

Equivalently, by the tensor-hom adjunction, there is a linear map
\[
\mc M_\bfT(X)(L)\longrightarrow  \Hom(M,\mc M_\R(\Ga_X)) \cong \mc M_\R(\Ga_X)\otimes_\Z  \Mcheck.
\]
If we choose a basis $(m_1,\dots,m_n)$ of $M$, the map is given by $E \mapsto \sum_i \Trop((m_{i})_*E) \otimes m_i^\vee$.
As we have defined this for an arbitrary non-Archimedean field extension of $K$, we have defined a map
\begin{equation}
\label{eq:line bundles over tori}
(\mc M_\bfT(X))^\an\longrightarrow \mc M_\R(\Ga_X)\otimes_\Z \Mcheck \ .
\end{equation}
Similarly, there is a canonical bilinear map
\[
 \mc M_{\Mcheck_\R}(\Ga_X) \times M\longrightarrow \mc M_\R(\Ga_X) \quad\text{given by}\quad (F,m) \longmapsto m_*(F)
\]
which induces an isomorphism 
\[
\mc M_{\Mcheck_\R}(\Ga_X)\xlongrightarrow{\cong} \mc M_\R(\Ga_X)\otimes_\Z \Mcheck \ .
\]
Composing the map in \eqref{eq:line bundles over tori} with (the inverse of) this isomorphism yields a tropicalization map
\[
\Trop \colon (\mc M_\bfT(X))^\an\longrightarrow \mc M_{\Mcheck_\R}(\Ga_X).
\]

\begin{proposition}
\label{prop:tropicalization is retraction for torus bundles}
Let $\mathbf{T}$ be an algebraic torus with cocharacter lattice $\Mcheck$. Then for every $\lambdacheck \in \pi_1(\bfT)$ there exists a homeomorphism 
\[
    \tau \colon \Sigma(\mc M^{\lambdacheck}_{\bfT}(X)) \xlongrightarrow{\sim} \mc M_{\Mcheck_\R}^{\lambdacheck}(\Gamma_X) \ ,  
\]
where $\Sigma(\mc M^{\lambdacheck}_{\bfT}(X))$ is the essential skeleton of $(\mathcal{M}_\bfT^{\lambdacheck}(X))^{\an}$, that fits into a commutative diagram
\[
\begin{tikzcd}
    & 
        \Sigma(\mc M^{\lambdacheck}_{\bfT}(X))\arrow[dd,"\cong"] \\
    (\mathcal{M}_\bfT^{\lambdacheck}(X))^{\an} \arrow[ur,"\rho"]\arrow[dr,"\Trop"]&
        & \\
    &
        \mc M_{\Mcheck_\R}^{\lambdacheck}(\Gamma_X)    \ ,\\
\end{tikzcd}
\]
where $\rho$ is the retraction map.
\end{proposition}

\begin{proof} 
When $\dim \bfT=1$ and the degree is equal to zero, this is \cite[Theorem 1.3]{BakerRabinoff}. This is generalized to an arbitrary degree in a special case of \cite[Theorem 6.2]{2022GrossUlirschZakharov}, by twisting by a base point and its tropicalization.  When $\dim \bfT\geq 1$, we observe that $\mc M_\bfT(X) \cong \Pic(X) \otimes_\Z \Mcheck$ in order to deduce the general case from the one-dimensional situation. We also note that $\mc M_\bfT(X)$ is a Calabi--Yau variety and hence has an essential skeleton. 
\end{proof}

\subsection{Tropicalizing $N_\bfG(\bfT)$-bundles over a Mumford curve}
\label{sec:tropicalizing NG-bundles}
Let $\rho\colon X^{\an}\to \Gamma_X$ denote the retraction map to the skeleton. Then pulling back along $\rho$ defines a fully faithful functor from the category of $W$-torsors on $\Gamma_X$ to the category of $W$-torsors on $X^\an$. The \'{e}space \'{e}tal\'{e} of any $W$-torsor over $X^\an$ is a covering space of $X^\an$, which induces an analytic structure on the latter.  In this way, we obtain a fully faithful functor from the category of $W$-torsors on $X^\an$ to the category of principal homogeneous spaces for $W$ over $X^\an$ in the category of $K$-analytic spaces. Since $W$ is finite, by the GAGA-principle, this category is in turn equivalent to the category of principal homogeneous spaces for $W$ over $X$ in the category of schemes. These in turn are all \'{e}tale over $X$, and the category of principal homogeneous spaces for $W$ over $X$ is thus equivalent to the category of \'{e}tale $W$-torsors on $X$ \cite[III, Theorem 4.3]{MilneEtale}. Taking the composition of these embeddings and equivalences, we obtain a fully faithful functor from the category of $W$-torsors on $\Gamma_X$ to the category of \'{e}tale $W$-torsors on $X$. For a $W$-torsor $\tau$ over $\Gamma_X$, we denote the associated \'{e}tale $W$-torsor on $X$ by $\rho^*\tau$. For a $W$-torsor $\psi$ on $X$, we denote by $\Bun_{N_\bfG(\bfT),\psi}(X)$ the stack of $N_\bfG(\bfT)$-torsors on $X$ whose associated $W$-torsor is $\psi$. 

We can intrinsically characterize those $W$-torsors on $X$ that are of the form $\rho^*\tau$ as follows. Consider the analytification of the total space of a $W$-torsor on $X$. This space comes with a $W$-action. If this action is free, then the quotient is a $W$-torsor $\tau$ on $X^{\an}$, and the original torsor on $X$ is $\rho^*\tau$.

Let $\tau$ be a $W$-torsor on $\Gamma_X$, let $L/K$ be a non-Archimedean extension, and let $E\in \Bun_{N_\bfG(\bfT),\rho^*\tau}(X)(L)$ be a bundle defined over $L$. To tropicalize $E$, we view $\rho^*\tau$ as a principal homogeneous space for $W$ over $X$ in the category of schemes. Let $\pi\colon\rho^*\tau\to X$ denote the structure map. Then, exactly as in Lemma \ref{lem:pull-back of tau by tau}, the $W$-torsor $\pi^*(\rho^*\tau)$ on $\rho^*\tau$ is canonically trivial. Therefore, exactly as in Lemma \ref{lem:description of torsors as equivariant torsors on cover}, the pullback $\pi^*E$ is induced by a $\bfT$-bundle on $\rho^*\tau$ that is invariant under the $\Aut(\pi)$-action, and this $\bfT$-bundle is well-defined up to the $\Aut(\tau)$-action. 
More concisely, we obtain a map
\[
\Bun_{N_\bfG(\bfT),\rho^*\tau}(X)(L)\longrightarrow \mc M_{\bfT}(\rho^*\tau)(L)^{\Aut(\pi)}/\Aut(\tau) \ .
\]
The skeleton of $(\rho^*\tau)^{\an}$ is the total space of $\tau$. Because the tropicalization of $\bfT$-bundles on $\rho^*\tau$ respects the actions of both $\Aut(\pi)$ and $\Aut(\tau)$, we obtain a tropicalization map
\[
\mc M_{\bfT}(\rho^*\tau)(L)^{\Aut(\pi)}/\Aut(\tau)\longrightarrow \mc M_{\Mcheck_\R}(\tau)^{\Aut(\pi)}/\Aut(\tau) \ .
\]
As $\Aut(\pi)=\Aut(\tau/\Gamma)$, Lemma \ref{lem:description of torsors as equivariant torsors on cover} yields an isomorphism
\[
\mc M_{\Mcheck_\R}(\tau)^{\Aut(\pi)}/\Aut(\tau)\xlongrightarrow{\cong} \mc M_{\bfG^\trop,\tau}(\Gamma_X) \ .
\]
Composing the three maps defines a tropicalization map for $N_\bfG(\bfT)$-bundles defined over the field extension $L$. 
Since the extension was arbitrary, we have in fact defined a tropicalization map
\[
\Trop\colon |\Bun_{N_\bfG(\bfT),\rho^*\tau}(X)^\an|\longrightarrow \mc M_{\bfG^\trop,\tau}(\Gamma_X)\, ,
\]
where $|\Bun_{N_\bfG(\bfT),\rho^*\tau}(X)^\an|$ denotes the points of the stack $\Bun_{N_\bfG(\bfT),\rho^*\tau}(X)^\an$. 

\begin{example}
In the case of $\mathbf G= \mathbf{GL}_n$, our construction of $\Trop$ differs from the one given in \cite{2022GrossUlirschZakharov}. To compare the two constructions, suppose we are given an $\G_m^n\rtimes S_n$-torsor $E$ on $X$ as a line bundle $L$ on $\rho^*\Gamma'$, where $\Gamma'\xrightarrow{f} \Gamma_X$ is a free degree $n$ cover of $\Gamma_X$. The tropicalization $\Trop_{\mathrm{GUZ}}(E)$ of $E$ in the sense of \cite{2022GrossUlirschZakharov} is given by the $\GL_n(\T)$-bundle represented by the tropical line bundle $\Trop(L)$ on the domain $\Gamma'$ of the cover $f$. 

Let $\tau$ be the $W$-torsor on $\Gamma_X$ that corresponds to the cover $f$. Since the pull-back of $\tau$ to its \'espace \'etal\'e (which we also denote by $\tau$) is canonically trivial, the cover 
\[
    \tau\times_{\Gamma_X}\Gamma' \to \tau
\]
is canonically trivial as well. The canonical trivialization determines $n$ sections of this cover, or equivalently $n$ morphisms $s_i\colon \tau\to \Gamma'$ over $\Gamma_X$. The twisted $\Aut(\tau)$-equivariant $\R^n$-bundle on $\tau$ corresponding to $\Trop_{\mathrm{GUZ}}(E)$ is given by 
\[
    \bigoplus_{i=1}^n s_i^*\Trop(L) \ .
\]

Pulling the $n$ sections back via the retraction $\rho$, we obtain $n$ morphisms 
\[
    t_i\coloneqq \rho^*s_i\colon\rho^*\tau\longrightarrow \rho^*\Gamma' \ ,
\]
which induce the canonical trivialization of the cover 
\[
    \rho^*\tau\times_X \rho^*\Gamma'\longrightarrow \rho^*\tau \ . 
\]
The twisted $\Aut(\tau)$-equivariant $\G_m^n$-bundle on $\rho^*\tau$ corresponding to $E$ is given by $\bigoplus_{i=1}^n t_i^*L$. Therefore, the tropicalization $\Trop(E)$ in the sense of the present paper is the $\GL_n(\T)$-bundle on $\Gamma_X$ corresponding to the twisted $\Aut(\tau)$-equivalent $\R^n$-bundle 
\[
    \bigoplus_{i=1}^n \Trop(t_i^*L) \ .
\]
As this equals $\bigoplus_{i=1}^n s_i^*\Trop(L)$, we conclude that
\[
    \Trop(E)=\Trop_{\mathrm{GUZ}}(E) \ .
\]
\end{example}

Let us now return to the general situation. We have defined the tropicalization map using the techniques from Lemma \ref{lem:description of torsors as equivariant torsors on cover} applied to the cover $\rho^*\tau\to X$, which can be canonically obtained from the $N_\bfG(\bfT)$-bundle we are tropicalizing. It is useful to also allow other covers. 
Let $f\colon \Gamma'\to \Gamma_X$ be a free Galois  cover such that there exists a trivialization $f^*\tau\xrightarrow{\chi}W_{\Gamma'}$.
Let $\rho^*f\colon X'\to X$ be the induced étale cover.
Then, exactly as in our definition of tropicalization above, we obtain, for every field extension $L/K$, a sequence of maps
\[
\Bun_{N_\bfG(\bfT),\rho^*\tau}(X)(L) \longrightarrow \Bun_{\bfT}(X')(L)^{\Aut(f)}/\Aut(\tau) \longrightarrow \mc M_{\Mcheck_\R}(\Gamma')^{\Aut(f)}/\Aut(\tau)\xlongrightarrow{\cong} \mc M_{\bfG^\trop,\tau}(\Gamma_X) \ ,
\]
and we denote by
\[
\Trop_{f,\chi}\colon \big|\Bun_{N_\bfG(\bfT),\rho^*\tau}(X)^\an\big|\longrightarrow \mc M_{\bfG^\trop,\tau}(\Gamma_X)
\]
the map induced by the composition. We note that, for $\tau$ a $W$-torsor on $\Ga_X$, we have $\Trop=\Trop_{\pi_0,\chi_0}$, where $\pi_0:\tau\to \Ga_X$ is the total space and $\chi_0:\pi_0^*\tau\to W_{\tau}$ is the canonical trivialization. 

\begin{lemma}
\label{lem:tropicalization wrt a cover}
With notation as above, we have 
\[
    \Trop_{f,\chi}(E)=\Trop(E)
\]
for every $E\in \big|\Bun_{N_\bfG(\bfT),\rho^*\tau}(X)^\an\big|$.
\end{lemma}

\begin{proof}
It suffices to show that $\Trop_{f,\chi}=\Trop_{g,\psi}$ for two covers $f$ and $g$ and trivializations $\chi$ and $\psi$. First we treat the case where $f=g$ are the same map $\Gamma'\to \Gamma_X$, but $\chi$ and $\psi$ are allowed to differ. The two trivializations $\chi$ and $\psi$ differ by 
\[
\psi \circ \chi^{-1} \colon W_{\Gamma'}\longrightarrow W_{\Gamma'}\ ,
\]
which is given by right multiplication by some $w\in H^0(\Gamma',W_{\Gamma'})$. By the algebraic analogue of Lemma \ref{lem:meaning of conjugation} (with analogous proof), if $E\in |\Bun_{N_\bfG(\bfT),\rho^*\tau}(X)^\an|$ is represented by $L_\chi$ (resp.\ $L_\psi$) in the invariants of $\mc M_{\bfT}(\rho^*f)$ with respect to the $\Aut(f)$-action induced by $\chi$ (resp.\ $\psi$), then $L_\psi$ is obtained from $L_\chi$ by conjugating with $w$, up to the $\Aut(\tau)$-action. Therefore, the $\Mcheck_\R$-bundle $\Trop(L_\psi)$ on $\Gamma'$ is obtained from $\Trop(L_\chi)$ by conjugating with $w$, up to the $\Aut(\tau)$-action. Now using Lemma \ref{lem:meaning of conjugation} on the tropical side shows that $\Trop_{f,\chi}(E)=\Trop_{f,\psi}(E)$.

Now we treat the case where $f$ and $g$ differ. Any two covers can be dominated by a common cover, so we may assume that $g$ factors through $f$, that is $g=f\circ h$ for some free Galois cover $h$. We already showed that $\Trop_{g,\psi}$ does not depend on the choice of $\psi$, so we may assume that $\psi=h^*\chi$. Using the same notation as above, the $\bfT$-bundle $L_\psi$ on the domain of $\rho^*g$ agrees with $(\rho^*h)^*L_\pi$ up to the $\Aut(\tau)$-action. Therefore, the $\Mcheck_\R$-bundle $\Trop(L_\psi)$ agrees with $h^*\Trop(L_\chi)$. Applying the algebraic argument backwards on the tropical side then shows the desired equality
\[
    \Trop_{f,\chi}(E)=\Trop_{g,\psi}(E) \ . \qedhere
\]
\end{proof}

Let $(M, R,  \Mcheck, \Rcheck)$ and $(\Lambda, \Phi, \Lambdacheck, \Phicheck)$ be root data and let $(\bfG,\bfT)$ and $(\bfG',\bfT')$ be the corresponding reductive groups with maximal tori $T$ and $T'$, respectively. Recall that a surjective morphism
\[
\phi\colon \bfG\longrightarrow \bfG'
\]
with $\phi(\bfT)\subseteq \bfT'$ and $\ker(\phi)\subseteq Z(\bfG)$

induces a morphism
\[
f\colon \Mcheck \longrightarrow \Lambdacheck \ ,
\]
whose image has finite index and that defines bijections $\Rcheck\to \Phicheck$ and, dually,  $\Phi\to R$ (we note that the morphism $p:\mathbf{G}'\to \mathbf{G}$ in Lemma~\ref{lem:existence of cover used for degree calculation} is of this type). Conversely, given $f$, the morphism $\phi$ can be reconstructed up to conjugation by elements in $\mathbf T$ by \cite{Steinberg}. 

The pair $(f,\id)$ induces a morphism of the tropical reductive groups $\mathbf{G}^{\trop}$ and $(\mathbf{G}')^{\trop}$ corresponding to our root data, and we denote this morphism by
\[
    \phi^{\trop}=(f, \id) \colon \mathbf{G}^\trop\longrightarrow (\mathbf{G}')^\trop.
\]

\begin{lemma}
\label{lem:tropicalization commutes with central isogenies}
Let $\phi \colon(\mathbf{G},\mathbf{T})\to (\mathbf{G}',\mathbf{T}')$ and $\phi^{\trop}: \mathbf{G}^\trop \to (\mathbf{G}')^{\trop}$ be as above. Let $E$ be a tropicalizable principal $N_\bfG(\bfT)$-bundle on $X$, in other words, assume that the associated $W$-torsor of $E$ is of the form $\rho^*\tau$. Then we have
\[
    \Trop(\phi_*E)\cong \phi^\trop_*\Trop(E) \ .
\]
\end{lemma}

\begin{proof}
We first treat the case where $\mathbf{G}=\mathbf{T}$ and $\mathbf{G}'=\mathbf{T}'$ are tori. Here, the statement follows from the commutativity, for every extension $L/K$ of non-Archimedean fields, of the diagram
\[
\begin{tikzcd}
\mc M_{\bfT}(X)(L) \arrow[r,"\cong"]\arrow[d,"\phi_*"]
    & \Pic(X)(L)\otimes_\Z \Mcheck    \arrow[r, "\Trop\otimes\id"]\arrow[d,"\id\otimes f"]
        &[1cm] \Pic(\Gamma_X)\otimes_\Z \Mcheck \arrow[r,"\cong"]\arrow[d,"\id\otimes f"] 
            & \mc M_{\Mcheck_\R}(\Gamma_X) \arrow[d, "\phi^\trop_*"]\\
\mc M_{\bfT'}(X)(L)   \arrow[r,"\cong"]
    & \Pic(X)(L)\otimes_\Z \Lambdacheck \arrow[r,"\Trop\otimes\id"]
        & \Pic(\Gamma_X)\otimes_\Z \Lambdacheck \arrow[r,"\cong"]
            & \mc M_{\Lambdacheck_\R}(\Gamma_X) \ .
\end{tikzcd}
\]

For general $\mathbf{G}$ and $\mathbf{G}'$ we denote by $\pi\colon \rho^*\tau \to X$ the projection from the total space. For an extension $L/K$ of non-Archimedean fields, consider the diagram
\[
\begin{tikzcd}
\Bun_{N_\bfG(\bfT),\rho^*\tau}(X)(L) \arrow[d,"\phi_*"]\arrow[r]
    &[-1em]\mc M_{\bfT}(\rho^*\tau)^{\Aut(\pi)}/\Aut(\tau) \arrow[d,"(\phi|_T)_*"]\arrow[r,"\Trop"]
        &\mc M_{\Mcheck_\R}(\tau)^{\Aut(\pi)}/\Aut(\tau)  \arrow[d,"(\phi^\trop|_{\Lambdacheck_\R})_*"] \arrow[r,"\cong"]
            &\mc M_{\Mcheck_\R}(\Gamma_X)  \arrow[d,"\phi^\trop_*"] \\
\Bun_{N_{\bfG'}(\bfT'),\rho^*\tau}(X)(L)\arrow[r]
    &\mc M_{\bfT'}(\rho^*\tau)^{\Aut(\pi)}/\Aut(\tau) \arrow[r,"\Trop"]
        &\mc M_{\Lambdacheck_\R}(\tau)^{\Aut(\pi)}/\Aut(\tau)\arrow[r,"\cong"]
            &\mc M_{\Lambdacheck_\R}(\Gamma_X) \ .
\end{tikzcd}
\]
The compositions of the arrows in the two rows are the tropicalizations maps for $N_\bfG(\bfT)$- and $N_{\mathbf{G}'}(\mathbf{T}')$-bundles, respectively. To complete the proof, it  suffices to show that the diagram commutes. For the left and right square, commutativity follows from the compatibility of push-forward with pull-backs along covers. The square in the middle commutes by the case of tori that we treated first. 
\end{proof}

\begin{lemma}
\label{lem:tropicalization commutes with determinants}
Let $E$ be a tropicalizable principal $N_\bfG(\bfT)$-bundle. Then we have
\[
    \Trop(\det(E))\cong {\det} (\Trop(E)) \ .
\]
\end{lemma}

\begin{proof}
Let $\tau$ be $W$ torsor on $\Gamma_X$ with $E\in \Bun_{N_\bfG(\bfT),\rho^*\tau}(X)$, let $\pi_0\colon \tau\to \Gamma_X$ be the projection from the total space and let $\chi_0 \colon\pi_0^*\tau\to W_\tau$ be the canonical trivialization. Then exactly as in Lemma \ref{lem:tropicalization commutes with central isogenies}, we see that 
\[
    \det(\Trop(E))=\Trop_{\pi_0}(\det(E)) \ ,
\]
where $\Trop_{\pi_0}$ is missing the datum of the trivialization of the torsor because $\det(E)$ is a torus bundle and tori have trivial Weyl groups (and hence all torsors over those Weyl groups are canonically trivial). Applying Lemma \ref{lem:tropicalization wrt a cover} finishes the proof.
\end{proof}

Recall that for a reductive group $\mathbf{G}$ with fixed maximal torus $\mathbf{T}$, there is a natural identification $\pi_0(\Bun_{\mathbf{G}}(X))\cong \pi_1(\mathbf{G})$ defined as follows: every principal $\mathbf{G}$-bundle can be degenerated to a principal $\bfT$-bundle, so the natural map
\[
\mathbb{X}_*(\mathbf{T})\cong\pi_0(\Bun_{\mathbf{T}}(X))\longrightarrow \pi_0(\Bun_{\mathbf{G}}(X))
\]
is a surjection. Moreover, two elements of $\mathbb{X}_*(\mathbf{T})$ have the same image precisely if they agree modulo the sublattice generated by the coroots. We refer to \cite{Hoffmann} for details.

Recall that for a principal $\mathbf{G}$-bundle $E$ we denote by $\deg(E)$ the element of $\pi_1(\mathbf{G})$ corresponding to the component of $\Bun_{\mathbf{G}}(X)$ containing $E$  and call it the \emph{degree} of $E$.

\begin{proposition}
\label{prop:tropicalization respects degree} 
    Let $E_\bfG$ be a $\mathbf{G}$-bundle induced by a tropicalizable principal $N_\bfG(\bfT)$-bundle $E$. Then
    \[
    \deg(E_{\mathbf{G}})=\deg(\Trop(E)) \ .
    \]
\end{proposition}

\begin{proof}
    For line bundles, that is if $\mathbf{G}=\G_m$, the statement is well-known. This implies it is also true for products of $\G_m$, that is for algebraic tori (note that all tori are split because the base field is algebraically closed). 
    
    Now assume that $\pi_1(\mathbf{G})$ is torsion-free. Then $\langle \Rcheck\rangle$ is saturated and the determinant induces an isomorphism of $\pi_1(\mathbf{G})$ with the cocharacter lattice of the cocenter $Z^c(\bfG)$ (see \S\ref{subsec:stability for bundles on elliptic curves}). As tropicalization commutes with determinants by Lemma \ref{lem:tropicalization commutes with determinants} and taking degrees of tropical principal bundles is functorial, the assertion is reduced to the case $\mathbf{G}=Z^c(\mathbf{G})$, which we have already established because $Z^c(\bfG)$ is an algebraic torus.

    For $\mathbf{G}$ arbitrary, let $p\colon \mathbf{G}'\to \mathbf{G}$ be as in Lemma \ref{lem:existence of cover used for degree calculation} and let $\mathbf{T}'=p^{-1}(\mathbf{T})$. Since $\ker(p)$ is connected and central in $\mathbf{G}'$, it is an algebraic torus, and so is $\mathbf{T}'$. So by \cite[Proposition 3.1, Remark 3.3 ii)]{Hoffmann}, the morphism $\Bun_{N_{\mathbf{G}'}(\mathbf{T}')}\to \Bun_{N_\bfG(\bfT)}$ is surjective and we can lift $E$ to a principal $N_{\mathbf{G}'}(\mathbf{T}')$-bundle $E'$. Of course, $E'$ is tropicalizable as well. As $\pi_1(\mathbf{G}')$ is torsion-free, we have already seen that $\deg(E'_{\mathbf{G}'})=\deg(\Trop(E'))$. Using Lemma \ref{lem:tropicalization commutes with central isogenies}, we conclude that
    \[
    \deg(\Trop(E))=
    \deg(\Trop(p_*E'))
    =\deg(\pi_*\Trop(E'))
    =\pi_*\deg(\Trop(E'))
    =\pi_*\deg(E'_{\mathbf{G}'})
    =\deg(E_\mathbf{G}) \ .
    \]
\end{proof}

\subsection{Tropicalizing stable $\mathbf{G}$-bundles over a Tate curve}

We now assume the Mumford curve $X$ is a \emph{Tate curve} over $K$, i.e.~a smooth projective curve of genus one, whose analytification is given by $X^{\an} = \G_m^{\an} / q^\Z$  with $\val(q) > 0$. Then the minimal skeleton $\Gamma_X$ is isometric to a metric circle of circumference $\val(q)$. Theorems~\ref{thm:Fratilla determinant is isomorphism} and~\ref{thm:Fratilas results} describe the moduli spaces $\mc M_\mathbf{G}^{\lambdacheck_\mathbf{G},\mathrm{st}}(X)$ and $\mc M_\mathbf{G}^{\lambdacheck_\mathbf{G},\mathrm{ss}}(X)$ of stable and semistable $\mathbf{G}$-bundles on $X$, respectively. In this section and the next, we explain how to tropicalize these moduli spaces: stable bundles are tropicalized by reducing them to $N_\bfG(\bfT)$-bundles, and semistable bundles are reduced to stable bundles by passing to a Levi subgroup.

First, let $\lambdacheck_\mathbf{G}\in \pi_1(\mathbf{G})$ be a degree for which the moduli space $\mc M_\mathbf{G}^{\lambdacheck_\mathbf{G},\mathrm{st}}(X)$ is nonempty. By Theorem~\ref{thm:Fratilla determinant is isomorphism}, this only happens if $\mathbf{G}^\ad= \prod_i \bPGL_{n_i}$, so that the Weyl group is $W=\prod_i S_{n_i}$. We recall from \S\ref{subsec:stable bundles on tropical elliptic curves} that since $\Gamma_X$ is a circle, up to isomorphism there is a unique indecomposable $W$-torsor $\ind$ on $\Gamma_X$, which induces a $W$-torsor on $X$ that we also denote $\ind$ by abuse of notation.

\begin{proposition}
    Let $E$ be a stable principal $\mathbf{G}$-bundle on $X$. Then $E$ can only be in the image of 
    \[
    \Bun_{N_\bfG(\bfT),\rho^*\tau}(X)\longrightarrow \Bun_\mathbf{G}(X)
    \]
    if $\tau$ is indecomposable.
\end{proposition}

\begin{proof}
    For every stable $\mathbf{G}$-bundle the induced $\mathbf{G}^\ad$-bundle is stable as well. This reduces to the case where $\mathbf{G}=\prod_i \bPGL_{n_i}$. Treating each factor individually, we further reduce to the case $\mathbf{G}=\bPGL_n$. Denote by $\mathbf D_n\subseteq \bGL_n$ the diagonal torus and $\mathbf D_n/\G_m\subseteq \bPGL_n$ the corresponding torus in $\bPGL_n$. Let $E\in \Bun_{N_{\bPGL_n}(\mathbf D_n/\G_m),\rho^*\tau}(X)$ be such that the associated $\bPGL_n$-torsor $E_{\bPGL_n}$ is stable, and assume that $\tau$ is decomposable. Then $\tau$ is induced by an $S_{k_1}\times \ldots \times S_{k_d}$-torsor for some nontrivial partition $\sum_{i=1}^d k_i=n$.  The preimage of $S_{k_1}\times\ldots \times S_{k_d}$ in $N_{\bPGL_n}(\bfD_n/\G_m)$ is the normalizer $N$ of the maximal torus $\mathbf D_n/\G_m$ in the Levi subgroup $\left(\prod_{i=1}^d \bGL_{k_i}\right)/\G_m$ of $\bPGL_n$. By \cite[Lemma 2.2.1]{BiswasHoffmann} (see also \cite{BiswasHoffmannError}), the structure group of $E$ can be reduced to $N$. In particular, the structure group of $E_{\bPGL_n}$ can be reduced to $\left(\prod_{i=1}^d \bGL_{k_i}\right)/\G_m$. Lifting $E_{\bPGL_n}$ to a $\bGL_n$-bundle, which is possible by \cite[Corollary 3.4]{Hoffmann}, we obtain a stable $\bGL_n$-bundle whose structure group can be reduced to $\prod_{i=1}^d \bGL_{k_i}$. But this is absurd, because stable vector bundles are indecomposable.
\end{proof}

We now study the morphism $\Bun_{N_\bfG(\bfT),\ind}(X)\to \Bun_\mathbf{G}(X)$.

\begin{example}
    Let $\mathbf{G}=\bGL_n$. Then $N_\bfG(\bfT)\cong \G_m^n \rtimes S_n$, $W=S_n$, and the stack $\Bun_{N_\bfG(\bfT)}(X)$ is equivalent to the stack of pairs $(X'\xrightarrow{f} X, \mc L)$, where $f$ is a finite étale cover of degree $n$ and $\mc L$ is a line bundle on $X'$. The associated $W$-torsor of a pair $(f,\mc L)$ is represented by $f$; the stack of $S_n$-torsors is equivalent to the stack of finite étale covers of degree $n$. We can describe the cover $f$ corresponding to the $S_n$-torsor $\ind$ on $X$ defined by the unique indecomposable $S_n$-torsor $\ind$ on $\Gamma_X$ explicitly by using a uniformization of $X^\an$. Indeed, if $X^\an=\G_m^\an/q^\Z$ then $\rho^*\ind$ is represented by the quotient map
    \[
    \pi_n\colon \G_m^\an/q^{n\Z}\longrightarrow \G_m^\an/q^\Z \ .
    \]
    If $X_n$ is the unique elliptic curve with $X_n^\an=\G_m^\an/q^{n\Z}$, we conclude that 
    \[
    \Bun_{N_{\mathbf{G}}(\bfT),\ind}(X) \cong \mPic(X_n)/\Aut(\pi_n) = \coprod_{d\in \Z} \mPic^d(X_n)/\Aut(\pi_n) \ .
    \]
    Note that $\pi_n$ is a Galois cover and $\Aut(\pi_n)$ is cyclic of degree $n$. We denote the connected component $\mPic^d(X_n)/\Aut(\pi_n)$ of $\Bun_{N_\bfG(\bfT),\ind}(X)$ by $\Bun_{N_\bfG(\bfT),\ind}^{d}(X)$.

    The morphism 
    \[
    \Bun_{N_\bfG(\bfT),\ind}(X)\longrightarrow \Bun_\mathbf{G}(X)
    \]
    maps the $N_\bfG(\bfT)$-torsor corresponding to $(\pi_n,\mc L)$ to the $\bGL_n$-torsor corresponding to the vector bundle $(\pi_n)_*\mc L$. This vector bundle has degree $d=\deg(\mc L)$. If $d$ and $n$ are coprime, then the argument in \cite[Appendix A]{Tu_semistablebundles} shows that $(\pi_n)_*\mc L$ is stable. In particular, for $d$ coprime to $n$ we obtain an induced map
    \[
        \Bun_{N_{\bGL_n}(\bfT),\ind}^{d}(X)\longrightarrow \Bun_{\bGL_n}^{d,\st}(X) \ ,
    \]
    where $\Bun_{\bGL_n}^{d,\st}(X)$ denotes the open substack of $\Bun_{\bGL_n}(X)$ consisting of stable bundles of degree $d$. 
\end{example}

\begin{lemma}
\label{lem:comparison of bun-spaces for GL_n}
    Let $d$ and $n$ be coprime. Then the morphism 
    \[
        \Bun_{N_{\bGL_n}(D_n),\ind}^{d}(X)\longrightarrow \Bun_{\bGL_n}^{d,\st}(X) 
    \]
    is an isomorphism of algebraic stacks.
\end{lemma}

\begin{proof}
    Consider the composition
    \begin{equation}
    \label{eq:Aut(X_N)-torsor}
    \mPic^d(X_n)\longrightarrow \Bun_{N_{\bGL_n}(D_n),\ind}^{d}(X)\longrightarrow \Bun_{\bGL_n}^{d,\st}(X) \ , 
    \end{equation}
    which we denote by $\phi$. Since $\mPic^d(X_n)$ is an $\Aut(\pi_n)$-torsor over $\Bun_{N_{\bGL_n}(\bfD_n)}(X)$, it suffices to show that $\phi$ is an $\Aut(\pi_n)$-torsor. Given line bundles $\mc L$ and $\mc L'$ on $(X_n)_S=X_n\times_K S$ for some test $K$-scheme $S$ and an isomorphism $(\pi_n)_*\mc L\xrightarrow{\psi} (\pi_n)_*\mc L'$ on $X_S$, there is an induced isomorphism 
    \begin{equation}
    \label{eq:induced map on split bundle}
    \bigoplus_{\sigma\in \Aut(\pi_n)} \sigma^*\mc L\cong \pi_n^*(\pi_n)_*\mc L\xrightarrow{\pi_n^*\psi} \pi_n^*(\pi_n)_*\mc L'\cong \bigoplus_{\sigma\in \Aut(\pi_n)} \sigma^*\mc L'    
    \end{equation}
    on $(X_n)_S$.  Note that because $\Aut(\pi_n)$ acts freely on $\Pic^d(X_n)$, it follows that there is a unique $\sigma\in \Aut(\pi_n)(S)$ such that the morphism $\mc L\to \sigma^*\mc L'$ induced from \eqref{eq:induced map on split bundle} is an isomorphism. This shows that the morphism
    \begin{equation*}\begin{split}
    \Aut(\pi_n) \times _K \mPic^d(X_n)(S) &\longrightarrow \mPic^d(X_n)(S)\times_{\Bun_{\bGL_n}^{d,\st}(X)(S)} \mPic^d(X_n)(S) \\ (\sigma, \mc L)&\longmapsto (\mc L, \sigma^*\mc L)
    \end{split}\end{equation*}
    is essentially surjective. It is also fully faithful, because $\phi_S$ is already fully faithful: as stable bundles are simple, we have, for $\mc L$ a degree $d$ line bundle on $X_n$, that 
    \[
    \Aut((\pi_n)_*\mc L)=\G_m(S) =\Aut(\mc L) \ .
    \]

    To conclude the proof, it suffices to show that $\phi$ is faithfully flat. By \cite[Theorem 7.1]{2022GrossUlirschZakharov}, $\phi$ is surjective. Both factors of $\phi$ in  \eqref{eq:Aut(X_N)-torsor} are representable and locally of finite type by \cite[Fact 2.3]{Hoffmann}. Moreover, we have already shown that the fibers of $\phi$ are finite (they are $\Aut(\pi_n)$-torsors), and both the target $\Bun^{d,\st}_{\bGL_n}(X)$ and the source are smooth \cite[Proposition 4.1]{Hoffmann}, so we are done by miracle flatness. 
    \end{proof}

\begin{definition}
    We define $\Bun_{N_\bfG(\bfT)}^{\st}(X)$ as the preimage of $\Bun_\mathbf{G}^{\st}(X)$ in $\Bun_{N_\bfG(\bfT),\ind}(X)$ under the map $\Bun_{N_\bfG(\bfT)}(X)\to \Bun_\mathbf{G}(X)$.
\end{definition}

\begin{theorem}
    The morphism
    \[
    \Bun^{\st}_{N_\bfG(\bfT)}(X)\longrightarrow \Bun_\mathbf{G}^{\st}(X) 
    \]
    is an isomorphism.
    \label{thm:stableisomorphism}
\end{theorem}

\begin{proof}
    If there are stable $\mathbf{G}$-bundles, then $\mathbf{G}^\ad\cong \prod_{i}\bPGL_{n_i}$ by Theorem \ref{thm:Fratilla determinant is isomorphism}. Moreover, if $\lambdacheck\in \pi_1(\mathbf{G})$ is such that $\Bun^{\lambdacheck,\st}_\mathbf{G}(X)$ is nonempty, then $\lambdacheck^{\ad}\in \prod_i (\Z/n_i\Z)^*$. The component $\Bun^{\lambdacheck^\ad}_{\mathbf{G}^\ad}(X)$ consists of a single point corresponding to a stable $\mathbf{G}^\ad$-bundle. In particular, the map
    \[
    \Bun_\mathbf{G}(X) \longrightarrow \Bun_{\mathbf{G}^\ad}(X)
    \]
    maps $\Bun_\mathbf{G}^{\st}(X)$ to $\Bun^{\st}_{\mathbf{G}^\ad}(X)$. Let $\mathbf{T}^\ad=\mathbf{T}/Z(\mathbf{G})$ denote the maximal torus of $\mathbf{G}^\ad$ induced by $\mathbf{T}$. Then $N_\bfG(\bfT)= N_{\mathbf{G}^\ad}(\mathbf{T}^\ad)\times_{\mathbf{G}^\ad} \mathbf{G}$. Applying \cite[Lemma 2.2.1]{BiswasHoffmann} (see also \cite{BiswasHoffmannError}), we obtain a $2$-cartesian diagram
    \begin{equation}
    \label{eq:cartesian diagram of bundle spaces}
    \begin{tikzcd}
        \Bun^{\st}_{N_\bfG(\bfT)}(X) \arrow[d]\arrow[r] &
            \Bun^{\st}_{N_{G^\ad}(T^\ad)}(X) \arrow[d]\\
        \Bun^{\st}_G(X) \arrow[r] &
            \Bun^{\st}_{G^\ad}(X) \ . \\    
    \end{tikzcd}
    \end{equation}
    Therefore, we reduce to the case where $\mathbf{G}$ is a product of $\mathbf{PGL}$'s. This in turn can be directly reduced to the case $\mathbf{G}=\mathbf{PGL_n}$. In that case, we again use the cartesian diagram \eqref{eq:cartesian diagram of bundle spaces}, but with reversed roles: we set $\mathbf{G}=\mathbf{GL_n}$ in which case $\mathbf{G}^\ad=\mathbf{PGL_n}$. Then the left vertical morphism is an isomorphism by Lemma \ref{lem:comparison of bun-spaces for GL_n}. Moreover, the lower horizontal morphism is smooth by \cite[Corollary 4.2]{Hoffmann} and surjective. Since being an isomorphism is local on the target in the smooth topology and the vertical morphisms are representable by \cite[Fact 2.3]{Hoffmann}, we are done.
\end{proof}

\begin{definition}
Let $E$ be a stable $\mathbf{G}$-bundle. Then we define the \emph{tropicalization of $E$} by $\Trop(E')$ (see Section~\ref{sec:tropicalizing NG-bundles}), where $E'$ is the unique indecomposable principal $N_\bfG(\bfT)$-bundle corresponding to $E$ under the isomorphism of Theorem~\ref{thm:stableisomorphism}.
\label{def:stabletropicalization}
\end{definition}

\begin{corollary}
    Let $E\in \Bun_{\mathbf{G}}^{\st}(X)$. Then we have
    \[
    \deg(E)=\deg(\Trop(E)) \ .
    \]
\end{corollary}

\begin{proof}
    This follows directly from Proposition \ref{prop:tropicalization respects degree}.
\end{proof}

\begin{theorem}
\label{thm:tropical stable bundles skeleton of stable bundles}
Let $\mathbf{G}$ be a reductive group of type $\prod A_{n_i}$ and let $\lambdacheck\in \pi_1(\mathbf{G})$ be a stable degree. Then there exists a commutative diagram
\[
\begin{tikzcd}
    & 
        \Sigma(\mc M_{\bfG}^{\lambdacheck,\st}(X))\arrow[dd,"\cong"] \\
    (\mc M_{\bfG}^{\lambdacheck,\st}(X))^\an \arrow[ur,"\rho"]\arrow[dr,"\Trop"']&
        & \\
    &
        \mc M^{\lambdacheck}_{\bfG^\trop,\ind}(\Gamma_X) \ ,\\
\end{tikzcd}
\]
where $\rho$ denotes the retraction map to the essential skeleton $\Sigma(\mc M_{\mathbf{G}}^{\lambdacheck,\st}(X))$ of $(\mc M_{\mathbf{G}}^{\lambdacheck,\st}(X))^\an$. 
\end{theorem}

\begin{proof} Consider the diagram
\[
\begin{tikzcd}
    (\mc M_{\mathbf{G}}^{\lambdacheck,\st}(X))^\an \arrow[d,"\Trop"] \arrow[r,"\det"]&
        (\mc M_{Z^c(\mathbf{G})}^{\det(\lambdacheck)}(X))^\an \arrow[dr,"\rho"] \arrow[d,"\Trop"] &
        \\
    \mc M_{\bfG^\trop,\ind}^{\lambdacheck}(\Ga_X) \arrow[r,"\det"] &
        \mc M_{Z^c(\bfG^\trop)}^{\det(\lambdacheck)}(\Ga_X)&
            \Sigma(\mc M_{Z^c(\mathbf{G})}^{\det(\lambdacheck)}(X)) \ .\arrow[l,"\cong"]\\
\end{tikzcd}
\]
The square on the left is commutative by Lemma \ref{lem:tropicalization commutes with determinants}. By Theorem \ref{thm:Fratilla determinant is isomorphism}, the algebraic determinant map on the top of the square is an isomorphism. By Theorem \ref{thm:determinant is isomorphism for tropical moduli of stable bundles}, the tropical determinant map on the bottom of the square is an isomorphism. It thus suffices to show the existence of the lower right isomorphism such that the triangle on the right commutes. But this is Proposition \ref{prop:tropicalization is retraction for torus bundles}.
\end{proof}

\subsection{Tropicalizing semistable $\mathbf{G}$-bundles over a Tate curve} In this section, we continue our study of the tropicalization of principal bundles on a Tate curve $X$ by reducing the semistable case to the case of stable bundles from the previous section.

Given a semistable bundle $F\in\mc M_{\mathbf{G}}^{\lambdacheck_G,\mathrm{ss}}(X)$, let $\mathbf{L}$ be the Levi subgroup determined by Theorem~\ref{thm:Fratilas results}, corresponding to the degree $\lambdacheck_{\mathbf{G}}=\deg (F)\in \pi_1(\mathbf{G})$ and chosen such that $\mathbf{T}\subseteq \mathbf{L}$. Reducing the structure group, we obtain an $\mathbf{L}$-bundle $F_{\mathbf{L}}$ on $X$ of degree $\lambdacheck_{\mathbf{L}}$, unique up to the $W_{\mathbf{L},\mathbf{G}}$-action and stable by Theorem~\ref{thm:Fratilas results}. We then use Definition~\ref{def:stabletropicalization} to tropicalize $F_{\mathbf{L}}$ to an $\mathbf{L}^\trop$-bundle on $\Gamma_X$ of degree $\lambdacheck_{\mathbf{L}}$. The inclusion $\mathbf{L}\to \mathbf{G}$ induces a morphism $\mathbf{L}^\trop \to \mathbf{G}^\trop$ of tropical reductive groups and the induced map $\mc M_{\mathbf{L}^\trop}^{\lambdacheck_\mathbf{L}}(\Gamma_X)\to \mc M_{\mathbf{G}^\trop}^{\lambdacheck_\mathbf{G}}(\Gamma_X)$ is $W_{\mathbf{L},\mathbf{G}}$-equivariant. We then extend scalars to $\mathbf{G}^\trop$ to obtain the tropicalization of $F$.

\begin{example}
   Let $F \in \mc M^{d,\semist}_{\bGL_n}(X)$ be a semistable $\bGL_n$-bundle of degree $d \in \Z$. Note that in this case the Levi subgroup $\bfL \subseteq \bGL_n$ from Theorem \ref{thm:Fratilas results} is given by $\bfL = (\bGL_{\frac{n}{h}})^h$ and $W_{\mathbf{L},\mathbf{G}} = S_h$, where $h =\gcd(n,d)$. Then one can show that $F$ is equivalent to a direct sum $\oplus^h_{i=1} F_i$ of stable vector bundles of the same slope  (see \cite[\S 7]{2022GrossUlirschZakharov}) which is unique up to the $S_h$-action. In the more general framework, this is the same as a stable $\bfL$-bundle $F_\bfL$ on $X$ of degree $(\frac{d}{h},\ldots,\frac{d}{h})\in \Z^h$. Tropicalizing the stable $\bfL$-bundle, as explained in the previous section, corresponds to tropicalizing each summand $F_i$ individually, which is precisely what is done in \cite{2022GrossUlirschZakharov}. In this sense, this section generalizes the tropicalization construction of semistable $\bGL_n$-bundles on $X$ of \cite{2022GrossUlirschZakharov}.
\end{example}

First, we prove a lemma that bridges the gap between the algebraic structure group $\mathbf{L}$ and its tropical counterpart $\mathbf{L}^{\trop}$. 

\begin{lemma}
\label{lem:tropicalization of stable bundles is equivariant}
Let $\mathbf{G}$ be a reductive group, and let $\mathbf{L}$ be a Levi subgroup containing the (fixed) maximal torus $\mathbf{T}$ with Weyl group $W_\mathbf{L}\subseteq W_\mathbf{G}$. Then there exists a natural isomorphism 
\[
    W_{\mathbf{L},\mathbf{G}}=N_\mathbf{G}(\mathbf{L})/\mathbf{L}\xlongrightarrow{\cong} N_{W_\mathbf{G}}(W_\mathbf{L})/W_\mathbf{L} \ .
\]
Moreover, if $\bfL$ and $\lambdacheck_\bfL\in \pi_1(\bfL)$ are as in Theorem \ref{thm:Fratilas results}, then $N_{W_\bfG}(W_\bfL)/W_\bfL$ acts on $\mc M_{\bfL^\trop,\ind}^{ \lambdacheck_\bfL}(\Gamma_X)$ and the tropicalization map
\[
    \Trop\colon \mc M_{\bfL}^{\lambdacheck_\bfL,\st}(X) \longrightarrow \mc M_{\bfL^\trop,\ind}^{\lambdacheck_\bfL}(\Gamma_X)
\]
is $N_\bfG(\bfL)/\bfL$-equivariant.
\end{lemma}

\begin{proof}
We first show that the quotient map 
\[
N_\bfG(\bfL)\cap N_\bfG(\bfT)\longrightarrow N_\bfG(\bfL)/\bfL
\]
is surjective. If $nL\in N_\bfG(\bfL)/\bfL$, then $n^{-1} \bfT n$ is a maximal torus of $\bfL$. As $\bfL$ is reductive, all maximal tori in $\bfL$ are conjugate, that is there exists $l\in \bfL$ with $(nl)^{-1} \bfT(nl)=\bfT$. Then we have $nl\in N_\bfG(\bfL)\cap N_\bfG(\bfT)$ and $nl\bfL=n\bfL$. Since $N_\bfG(\bfL)\cap N_\bfG(\bfT)\cap\bfL= N_\bfL(\bfT)$ we obtain a natural isomorphism
\[
N_\bfG(\bfL)\cap N_\bfG(\bfT)/N_\bfL(\bfT) \xlongrightarrow{\cong}N_\bfG(\bfL)/\bfL \ .
\]
As $W_\bfL=N_\bfL(\bfT)/\bfT$ and $N_{W_\bfG}(W_\bfL)=N_{N_\bfG(\bfT)}(N_\bfL(\bfT))/\bfT$ we have
\[
N_{W_\bfG}(W_\bfL)/W_\bfL \cong N_{N_\bfG(\bfT)}(N_\bfL(\bfT))/N_\bfL(\bfT)
\]
by the third isomorphism theorem. Therefore, it suffices to show that 
\[
N_\bfG(\bfL)\cap N_\bfG(\bfT)= N_{N_\bfG(\bfT)}(N_\bfL(\bfT)) \ .
\]
The inclusion from left to right is clear, as any automorphism of $\bfL$ that fixes $\bfT$ also fixes the normalizer of $\bfT$. For the reverse inclusion let $n\in  N_{N_\bfG(\bfT)}(N_\bfL(\bfT))$. We need to show that $n$ normalizes $\bfL$. Let $Z$ be the connected component of the identity of $\bigcap_{\alpha\in \Psi}\ker(\alpha)$, where $\Psi$ is the set of roots of $\bfL$. By \cite[Section 30.2]{Humphreys}, we have 
\[
\bfL=C_\bfG(Z) \ , 
\]
 so it suffices to show that $n$ normalizes $Z$. As $n$ normalizes $N_\bfL(\bfT)$, it suffices to show that $Z$ is the connected center of $N_\bfL(\bfT)$. Because $\bfG$ is reductive, we have $C_\bfG(\bfT)=\bfT$ and hence $Z(N_\bfL(\bfT))\subseteq \bfT$. So $Z(N_\bfL(\bfT))$ is precisely the subset of $\bfT$ fixed by all reflections in $W$, which is precisely $\bigcap_{\alpha\in \Psi}\ker(\alpha)$, the identity component of which is $Z$.

The normalizer $N_{W_\bfG}(W_\bfL)$ is contained in $N_{\bfG^\trop}(\bfL^\trop)$ and acts on $\mc M_{\bfL^\trop,\ind}^{\lambdacheck_L}(\Gamma_X)$ by conjugation. Conjugation with inner automorphisms of $\bfL^\trop$ leaves $\bfL^\trop$-bundles unchanged, so $W_\bfL$ is in the kernel of the action and we obtain an action of $N_{W_\bfG}(W_\bfL)/W_\bfL$. Moreover, conjugation by an element in $N_{N_\bfG(\bfT)}(N_\bfL(\bfT))$ tropicalizes to the conjugation by its image in $W_\bfG$. Therefore, the equivariance of $\Trop$ follows from Lemma \ref{lem:tropicalization commutes with central isogenies}.
\end{proof}

Let $G= \Mcheck_\R \rtimes W$ be a tropical reductive group and let $\lambdacheck \in \pi_1(G)$. In general, it is not yet clear what it means for a principal $G$-bundle to have an indecomposable degree. But, by Theorem \ref{thm:Fratila existence of minimal parabolic} there exists a parabolic subgroup $P=\Mcheck_\R \rtimes W'$ in $G$ such that there exists $\lambdacheck_P\in \pi_1(P)$ with $\phi_P(\lambdacheck_P)=\phi_G(\lambdacheck)$ and which is minimal with respect to that property. The parabolic $P$ is unique up to conjugation and of type $\prod_i A_{n_i-1}$ by \cite[Cor. 4.2]{fratila2016stack}. Let $\tau$ be an indecomposable $W'$-torsor. Then we denote 
\[
    \mc M_{G,\ind}^{\lambdacheck}(\Ga_X)\coloneqq\mc M_{G,\tau_W}^{ \lambdacheck}(\Ga_X) \ .
\]

\begin{definition}
Let $\mathbf{G}$ be a reductive group and let $E$ be a semi-stable principal $\mathbf{G}$-bundle of degree $\lambdacheck \in \pi_1(\bfG)$. By Theorem \ref{thm:Fratilas results}, there exists a Levi subgroup $\mathbf{L} \subseteq \mathbf{G}$, uniquely determined up to conjugation, and a degree $\lambdacheck_\mathbf{L}\in \pi_1(\mathbf{L})$, such that $\mathbf{L}$ is of type $\prod_i A_{n_i-1}$, the degree $\lambdacheck_\mathbf{L}$ is stable, and $E$ can be reduced to a stable $\mathbf{L}$-bundle $E_\mathbf{L}$ of degree $\lambdacheck_\mathbf{L}$, uniquely up to the action of $N_\mathbf{G}(\mathbf{L})/\mathbf{L}$. By Lemma \ref{lem:tropicalization of stable bundles is equivariant}, tropicalizing yields an object $\Trop(E_\mathbf{L})\in \mc M_{\mathbf{L}^\trop,\ind}^{\lambdacheck_\mathbf{L}}(\Ga_X)$, well-defined up to the $N_{W_\mathbf{G}}(W_\mathbf{L})$-action. Pushing forward to $\mc M_{\mathbf{G}^\trop,\ind}(\Ga_X)$ then yields a uniquely determined element $\Trop(E)\in \mc M_{\mathbf{G}^\trop, \ind}^{\lambdacheck}(\Ga_X)$.
\end{definition}

\begin{theorem} 
\label{thm:tropical semistable bundles skeleton of semistable bundles}
Let $X$ be a Tate elliptic curve over an algebraically closed complete nontrivially valued non-Archimedean field $K$ of equicharacteristic $0$ with minimal skeleton $\Gamma_X$, let $\mathbf{G}$ be a reductive group and let $\lambdacheck\in \pi_1(\mathbf{G})$. Denote by $\rho$ the retraction map to the essential skeleton $\Sigma(\mc M_{\mathbf{G}}^{\lambdacheck,\semist}(X))$ of $(\mc M_{\mathbf{\mathbf{G}}}^{\lambdacheck,\semist}(X))^\an$. Then there exists a homeomorphism $\Sigma(\mc M_{\mathbf{G}}^{\lambdacheck,\semist}(X))\xrightarrow{\cong}\mc M_{\mathbf{G}^\trop,\ind}^{\lambdacheck}(\Ga_X)$ that makes the diagram
\[
\begin{tikzcd}
    & 
        \Sigma(\mc M_{\mathbf{G}}^{\lambdacheck,\semist}(X))\arrow[dd,"\cong"] \\
    (\mc M_{\mathbf{G}}^{\lambdacheck,\semist}(X))^\an \arrow[ur,"\rho"]\arrow[dr,"\Trop"]&
        & \\
    &
        \mc M_{\mathbf{G}^\trop,\ind}^{\lambdacheck}(\Ga_X)\\
\end{tikzcd}
\]
commute.
\end{theorem}

\begin{proof}
Let $\mathbf{L}$ and $\lambdacheck_{\mathbf{L}}$ be as in Theorem \ref{thm:Fratilas results}. Consider the diagram of solid arrows
\[
\begin{tikzcd}
    \left(\mc M_{\mathbf{L}}^{\lambdacheck_{\mathbf{L}},\st}(X)\right)^\an \arrow[dr,"\rho_{\mathbf{L}}"]\arrow[dd,"\Trop"]\arrow[rrr]&
        &
            &
                \left(\mc M_{\mathbf{G}}^{\lambdacheck,\semist}(X)\right)^\an\arrow[dl,"\rho"]\arrow[dd,"\Trop"]\\    
    &
        \Sigma(\mc M_{\mathbf{L}}^{\lambdacheck_{\mathbf{L}},\st}(X))\arrow[r]\arrow[dl,"\cong"] &
            \Sigma(\mc M_{\mathbf{G}}^{\lambdacheck,\semist}(X)) \arrow[dr,dashed] &
                \\
    \mc M_{\mathbf{L}^\trop,\ind}^{\lambdacheck_{\mathbf{L}}}(\Ga_X) \arrow[rrr]&
        &
            &
                \mc M_{\mathbf{G}^\trop,\ind}^{\lambdacheck}(\Ga_X) \ ,\\
\end{tikzcd}
\]
where $\rho_{\mathbf{L}}$ is the retraction to the essential skeleton of $\left(\mc M_{\mathbf{L}}^{\lambdacheck_\mathbf{L},\st}(X)\right)^\an $. The triangle to the left exists by Theorem \ref{thm:tropical stable bundles skeleton of stable bundles}. The retraction $\rho_{\mathbf{L}}$ is equivariant with respect to the action of $W_{\bfL,\bfG}\coloneqq N_\bfG(\bfL)/\bfL=N_{W_\bfG}(W_\bfL)/W_\bfL$  by functoriality of the essential skeleton, and the map $\Trop$ in that triangle is $W_{\mathbf{L},\mathbf{G}}$-equivariant by Lemma \ref{lem:tropicalization of stable bundles is equivariant}.

To show that the solid trapezoid on top exists, we first note that $\mc M_{\mathbf L}^{\lambdacheck_{\mathbf L},\st}$ is isomorphic to a product of elliptic curves by Theorem \ref{thm:Fratilla determinant is isomorphism}, and hence its canonical bundle is trivial. Therefore, all pluricanonical forms on $\mc M_{\mathbf L}^{\lambdacheck_{\mathbf L},\st}(X)$ define the same skeleton \cite[Proposition 4.4.5 (5)]{MustataNicaise}, namely the essential skeleton $\Sigma(\mc M_{\mathbf L}^{\lambdacheck_{\mathbf L},\st}(X))$. 
Let $\omega$ be a pluricanonical form on $\mc M_{\mathbf L}^{\lambdacheck_{\mathbf L},\st}(X)$. The group $W_{\bfL,\bfG}$ acts on $\omega$ by 
\[
w.\omega= \chi(w)\cdot\omega 
\]
for some character $\chi$. If $k$ is the order of $\chi$, then $\omega^{\otimes k}$ is $W_{\bfL,\bfG}$-invariant, and hence there exist $W_{\bfL,\bfG}$-invariant pluricanonical forms. By \cite[Proposition 6.1.9]{BrownMazzon}, every$W_{\bfL,\bfG}$-invariant pluricanonical form induces the same skeleton of $\left(M_{\mathbf{G}}^{\lambdacheck,\semist}(X)\right)^\an$, namely the essential skeleton $\Sigma(M_{\mathbf{G}}^{\lambdacheck,\semist}(X))$, and we have 
\[
    \Sigma(\mc M_{\mathbf{L}}^{\lambdacheck_\mathbf{L},\st}(X))/W_{\mathbf{L},\mathbf{G}}= \Sigma(\mc M_{\mathbf{G}}^{\lambdacheck,\semist}(X)) \ .
\]
The outer square is commutative by the construction of the tropicalization map for semistable bundles. Since we also have
\[
    \mc M_{\mathbf{L}^\trop,\ind}^{\lambdacheck_\mathbf{L}}(\Ga_X) /W_{\mathbf{L},\mathbf{G}} = \mc \mc M_{\mathbf{G}^\trop,\ind}^{\lambdacheck}(\Ga_X).
\]
by Theorem \ref{thm:tropical semistable bundles are quotients of stable ones}, it follows that the dashed arrow can be filled in uniquely by a homeomorphism that makes the whole diagram commutative.
\end{proof}

\bibliographystyle{amsalpha}
\bibliography{biblio}{}

\end{document}